\documentclass[10pt, onecolumn]{IEEEtran}
\usepackage{amsmath,amssymb,euscript,yfonts,psfrag,latexsym,graphicx}
\usepackage{bbm,color,amstext,wasysym,parskip}
\usepackage{amsthm}
\graphicspath{{./},{../figures/},{./figures/},{./fig_ensemble/},{./profile_cmp/},{./sinkhorn_cmp/}}
\usepackage{caption}
\usepackage{subcaption}
\usepackage{float}
\usepackage{array}
\usepackage{url}
\usepackage{algorithm}
\usepackage{algorithmic}

\usepackage{ulem} 

\newtheorem{thm}{Theorem}

\newtheorem{lemma}[thm]{Lemma}
\newtheorem{prop}{Proposition}

\newtheorem{remark}[thm]{Remark}

\newtheorem{ex}[thm]{Example}

\newcommand{\mb}{{\mathbf b}}

\newcommand{\bp}{{\mathbf p}}
\newcommand{\mx}{{\mathbf x}}
\newcommand{\bm}{{\mathbf m}}
\newcommand{\bn}{{\mathbf n}}
\newcommand{\bu}{{\mathbf u}}

\newcommand{\cB}{{\mathcal B}}

\newcommand{\cM}{{\mathcal M}}
\newcommand{\mL}{{\mathbb L}}

\newcommand{\chen}[1]{\textcolor{cyan}{#1}}

\newcommand{\has}[1]{\textcolor{magenta}{#1}}

\newtheorem{alg}[algorithm]{Algorithm}

\newcommand{\bB}{{\mathbf B}}

\newcommand{\mR}{{\mathbb R}}

\newcommand{\cF}{{\mathcal F}}

\newcommand{\cH}{{\mathcal H}}

\newcommand{\cL}{{\mathcal L}}

\newcommand{\cP}{{\mathcal P}}

\newcommand{\cU}{{\mathcal U}}

\newcommand{\bx}{{\bf x}}
\def\bC{{\bf C}}
\def\bK{{\bf K}}

\def\bU{{\bf U}}
\newcommand{\ett}{{\bf 1}}
\def\minwrt[#1]{\underset{#1}{\text{minimize}}}
\def\maxwrt[#1]{\underset{#1}{\text{maximize }}}

\newcommand{\diag}{\operatorname{diag}}

\newcommand{\tr}{\operatorname{trace}}

\newcommand{\argmin}{\operatorname{argmin}}

\providecommand{\keywords}[1]
{
  \small	
  \textbf{\textit{Keywords---}} #1
}

\definecolor{grey}{rgb}{0.6,0.6,0.6}
\definecolor{lightgray}{rgb}{0.97,.99,0.99}

\begin{document}
\title{Multi-marginal optimal transport and\\
probabilistic graphical models}

\author{Isabel Haasler, Rahul Singh, Qinsheng Zhang, Johan Karlsson, and Yongxin Chen
\thanks{This work was supported by the Swedish Research Council (VR), grant 2014-5870, SJTU-KTH cooperation grant and the NSF under grant 1901599 and 1942523.
}
\thanks{I.~Haasler and J.~Karlsson are with the Division of Optimization and Systems Theory, Department of Mathematics, KTH Royal Institute of Technology, Stockholm, Sweden. {\tt\small haasler@kth.se, johan.karlsson@math.kth.se}}
\thanks{R. Singh, Q. Zhang and Y. Chen are with the School of Aerospace Engineering,
Georgia Institute of Technology, Atlanta, GA, USA. {\tt\small \{qzhang419,rasingh,yongchen\}@gatech.edu}}
\thanks{I.~Haasler, R. Singh and Q. Zhang contribute equally to this paper.}}

\maketitle

\begin{abstract}
We study multi-marginal optimal transport problems from a probabilistic graphical model perspective. We point out an elegant connection between the two when the underlying cost for optimal transport allows a graph structure. In particular, an entropy regularized multi-marginal optimal transport is equivalent to a Bayesian marginal inference problem for probabilistic graphical models with the additional requirement that some of the marginal distributions are specified. This
relation on the one hand extends the optimal transport as well as the probabilistic graphical model theories, and on the other hand leads to fast algorithms for multi-marginal optimal transport by leveraging the well-developed algorithms in Bayesian inference.
Several numerical examples are provided to highlight the results.
\end{abstract}


\keywords{Optimal transport, Probabilistic graphical models, Belief Propagation, Norm-product, Iterative Scaling, Bayesian inference.}



\section{Introduction}

Optimal transport (OT) theory \cite{Eva99,Vil03} is a powerful tool in the study of probability distributions. The subject dates back to 1781, when the civil engineer Monge aimed to find an optimal strategy to move soil to road construction sites. Over 200 years of development have brought OT far beyond a civil engineering problem to a compelling mathematical framework which has found applications in economics, signal and image processing, systems and controls, statistics, and machine learning \cite{HakTanAng04,MueKarKolTan13,CheGeoPav14e,CheGeoPav15b,Gal16,Che16,ArjChiBot17}. The inherent properties of OT make it especially suitable for handling high-dimensional data with low-dimensional structure, which is the case in most machine learning settings. Thanks to the discoveries of several efficient algorithms such as iterative scaling, also called Sinkhorn iterations \cite{Cut13}, OT has become a powerful framework for a range of machine learning problems. In the 30s this algorithm has also been studied in the statistics community under the name contingence table \cite{DemSte40}.

The aim of standard OT problems is to find a joint distribution of two given marginals that minimizes the total transportation cost between them. 
In some applications, such as incompressible fluid flow modeling, video prediction, tomography, and information fusion problems, more than two marginal distributions are given. To tackle these problems, a multi-marginal generalization of OT has been developed, known as multi-marginal optimal transport (MOT) \cite{GanSwi98,carlier2003class,Pas12,Pas15,Nen16}. 
MOT was first proposed in \cite{Pas12} as a theoretical extension to OT. Since then, the problem has been studied from a theoretical viewpoint \cite{Pas15} as well as computational perspective \cite{Nen16}. It has found applications in signal processing \cite{ElvHaaJakKar20}, fluid dynamics \cite{BenCarNen19}, density functional theory \cite{ButDeGor12,KhoLinLinYin19}, and estimation and control \cite{CheConGeo18,CheKar18}.
Many results for the standard OT problem have been extended to the multi-marginal setting. In particular, the iterative scaling method \cite{Cut13} has been generalized to MOT \cite{BenCarCut15}. However, for the multi-marginal setting the computational complexity remains high, especially when the number of marginal distributions is large \cite{Nen16,lin2019complexity}.  

On a seemingly different topic, probabilistic graphical models (PGMs) \cite{WaiJor08,KolFri09,Att00} provide a framework for multi-dimensional random variables. They have been used for a large variety of applications including speech recognition, computer vision, communications, and bioinformatics \cite{BilZwe02,LarCalSan06,WaiJor08,KolFri09}.
PGMs capture the dependencies of a set of random variables compactly as a graph, and are an efficient and robust tool to study the relationship of several probabilistic quantities. Moreover, prior knowledge can be easily incorporated in the model. During the last decades, many efficient algorithms have been developed for inference and learning of PGMs. These algorithms leverage the underlying graph structure, making it possible to solve many otherwise extremely difficult problems. Well-known algorithms for the inference problem include, e.g., belief propagation, and the junction tree algorithm \cite{YedFreWei01,YedFreWei05,MurWeiJor99,YedFreWei03,AjiMce00}.

The purpose of this paper is to point out a surprising connection between MOT and PGMs.
More precisely, the main contribution of this work is to establish an equivalence between regularized MOT problems, where the cost function is structured according to a graph, 
and the inference problem for a PGM on the same graph, where some marginal distributions are fixed.
This connection leads to a novel interpretation for both MOT and PGMs. On the one hand, MOT can be viewed as an inference problem of a PGM with constraints on some of the marginals, that is, constrained inference problems, and on the other hand, the inference problem of a PGM is a MOT problem, where the available marginals are Dirac distributions. 
%

From a numerical point of view, this connection allows for adapting existing PGM algorithms to this class of MOT problems with graphical structured cost.
In this work, we focus on belief propagation (BP) \cite{Pea88,YedFreWei01} and an extension of it known as norm-product algorithm \cite{HazSha10}. These belong to the so called message passing algorithms which exploit the underlying graphical structure, and rely on exchanging information between the nodes. Thus, only local updates are needed, which greatly reduces the computational complexity of the inference problem.
If the graphical model is a tree, these algorithms converge globally to the exact conditional marginal distributions in a finite number of iterations. For general PGMs with cycles, there is no convergence guarantee, but both methods usually work well in practice and provide relatively accurate approximations of the marginals.
In this work, we develop algorithms for solving entropy regularized MOT problems, or equivalently constrained inference problems, by combining these message passing algorithms with the iterative scaling method. Moreover, we build on earlier results and establish global convergence of our algorithms for tree graphs. 
Similar constrained inference problems have previously been studied in \cite{TehWel02}. Interestingly, the algorithm presented therein is comparable to our proposed extension of the BP algorithm. With the connections to MOT, we provide a new motivation for studying this problem, which also leads to a more complete picture of the algorithms.

A promising application of our framework are in the inference problems for collective dynamics, for instance, in the estimation of the behavior of large groups from only aggregate measurements. Such types of filtering methods are crucial for collective dynamics since it is usually impossible to track the trajectories of each single agent in a large population, due to exploding computational complexity, lack of sensor data or for privacy considerations. Related problems have been studied under the name collective graphical models (CGMs) \cite{SheDie11,SheSunKumDie13,SunSheKum15}.
These works consider a large collection of identical graphical models, which are observed simultaneously, and aim to infer aggregate distribution over the nodes. Several heuristic algorithms \cite{SunSheKum15} have been proposed to solve the resulting inference problems.
Our MOT framework suggest a different observation model \cite{HaaRinChe19,SinHaaZha20}, which is reasonable in many scenarios. More importantly, our algorithms, which enjoy global convergence guarantee, provide a reliable machine to estimate collective dynamics in these models.

The rest of this paper is structured as follows. In Section \ref{sec:pre} we review some background knowledge in optimal transport and probabilistic graphical models. In Section \ref{sec:momt_bayes} we provide the main theoretical result in this paper, which is the equivalence between entropy regularized MOT and the inference problem for PGMs. We also modify the belief propagation algorithm to solve MOT problems. Another algorithm based on the norm-product algorithm is introduced in Section \ref{sec:const_norm_prod}. We test and verify our results through several numerical examples in Section \ref{sec:example}. This is followed by a brief concluding remark in Section \ref{sec:conclusion}. 

{\bf Notation:} The notation used throughout is mostly standard. However, with $\exp( \cdot )$, $\ln( \cdot )$, $\odot$, and $./$ we denote the element-wise exponential, logarithm, multiplication, and division of vectors, matrices, and tensors, respectively. Moreover, $\otimes$ denotes the outer product. By $\ett$ we denote a vector of ones, the size of which will be clear from the context. Throughout, we use bold symbols to represent vectors, e.g., $\mb_j, \boldsymbol\mu_j$, and regular symbol for the corresponding entries, e.g., $b_j(x_j), \mu_j(x_j)$.

 \section{Preliminaries}\label{sec:pre}
 In this section, we provide a quick overview of optimal transport theory, probabilistic graphical models and belief propagation algorithm. We only cover material that is most relevant to this work. The reader is referred to \cite{Vil03,WaiJor08,KolFri09} for more details.

\subsection{Optimal transport} \label{sec:omt}
In optimal transport (OT) problems, one seeks an optimal plan that transports mass from a source distribution to a target distribution with minimum cost. In its original formulation \cite{Mon81}, OT was studied over Euclidean space. However, in general, OT problems can be formulated in both continuous space and discrete space. 
In this work, we focus on optimal transport over discrete space. 

Let $\boldsymbol\mu_1\in \mR_+^{d_1},\boldsymbol\mu_2 \in \mR_+^{d_2}$ be two discrete distributions, viz., nonnegative vectors, with equal mass, that is $\sum_{i_1} \mu_1(i_1) = \sum_{i_2} \mu_2 (i_2)$. Here $\mu_1(i)$ denotes the amount of mass in the source distribution at  location $i$ and $\mu_2(i)$ denotes the amount of mass in the target distribution at  location $i$. Without loss of generality, we assume that both $\boldsymbol\mu_1$ and $\boldsymbol\mu_2$ are probability vectors, that is, the total mass is $\sum_{i_1} \mu_1(i_1) = \sum_{i_2} \mu_2 (i_2)=1$. The transport cost of moving a unit mass from point $i_1$ to $i_2$ is denoted by $C(i_1,i_2)$, and collected in the matrix $\bC=[C(i_1,i_2)]\in \mR^{d_1\times d_2}$. In the Kantorovich formulation \cite{Kan42} of OT, the goal is to find a transport plan between the two marginal distributions $\boldsymbol\mu_1$ and $\boldsymbol\mu_2$ that minimizes the total transport cost. A transport plan is encoded in a joint probability matrix $\bB = [B(i_1,i_2)] \in \mR_+^{d_1\times d_2}$ of $\boldsymbol\mu_1, \boldsymbol\mu_2$. Then the total transport cost is $\sum_{i_1,i_2} C(i_1,i_2) B(i_1,i_2)= \tr(\bC^T \bB)$ and therefore the OT problem reads
\begin{equation}\label{eq:omt_bi_discrete}
  \begin{aligned}
    \min_{\bB \in \mR_+^{d_1\times d_2}}~ & \tr(\bC^T \bB) \\
\text{ subject to } & \bB \ett = \boldsymbol\mu_1 \\
& \bB^T \ett = \boldsymbol\mu_2,
\end{aligned}  
\end{equation}
where $\ett$ denotes a vector of ones of proper dimension. The constraints are to enforce that $\bB$ is a joint distribution between $\boldsymbol\mu_1$ and $\boldsymbol\mu_2$. 

Even though the above OT problem \eqref{eq:omt_bi_discrete} is a linear program, in many practical applications it is too difficult to be solved directly using standard solvers due to the large number of variables~\cite{BenCarCut15, ElvHaaJakKar20}, especially in the case where the marginal distributions $\boldsymbol\mu_1,\boldsymbol\mu_2$ come from discretizations of continuous measures. Recently, a regularization of OT \cite{Cut13} was proposed that greatly reduces the computational complexity of (approximately) solving OT problems over discrete space. In this method, an entropy term 
	\begin{equation}\label{eq:entropy2d}
	\cH(\bB) = - \sum_{i_1,i_2} B(i_1,i_2) \ln\, B(i_1,i_2)
	\end{equation}
is added to regularize the problem, leading to
\begin{equation}\label{eq:omt_bi_discrete_regularized}
  \begin{aligned}
    \min_{\bB \in \mR_+^{d_1\times d_2}}~ & \tr(\bC^T \bB) - \epsilon \cH(\bB) \\
\text{ subject to } & \bB \ett = \boldsymbol\mu_1 \\
& \bB^T \ett = \boldsymbol\mu_2,
\end{aligned}  
\end{equation}
where $\epsilon>0$ is a regularization parameter.
The entropy regularized OT problem \eqref{eq:omt_bi_discrete_regularized} is strictly convex and thus the solution is unique. More importantly, it can be solved efficiently via the Sinkhorn algorithm~\cite{Sin64,Cut13}, also known as iterative scaling \cite{DemSte40,FraLor89}. Let $\bK= [K(i_1,i_2)] \in \mR^{d_1\times d_2}$ be defined as $K(i_1,i_2) = \exp(-C(i_1,i_2)/\epsilon)$, then the iterative scaling updates alternate between the two steps
\begin{equation}\label{eq:Sinkhorn2d}
\bu_1 \leftarrow \boldsymbol\mu_1./\bK \bu_2,\quad \bu_2 \leftarrow \boldsymbol\mu_2./ \bK^T \bu_1,
\end{equation}
where $./$ denotes element-wise division. 
The algorithm converges linearly to a unique pair of vectors $\bu_1\in \mR^{d_1}, \bu_2\in \mR^{d_2}$ up to a normalization \cite{FraLor89}. Given the limit point of the iteration, the solution to \eqref{eq:omt_bi_discrete_regularized} has the form 
	\begin{equation}
		\bB = \diag(\bu_1)\bK\diag(\bu_2),
	\end{equation}
that is, $B(i_1,i_2) = K(i_1,i_2) u_1(i_1) u_2(i_2)$ for all $1\le i_1 \le d_1, 1\le i_2 \le d_2$. 
\subsection{Probabilistic graphical models}
A probabilistic graphical model (PGM) is a graph-based representation of a collection of random vectors that captures the conditional dependencies between them. It provides a compact representation of their joint distributions through factorization: a graphical model consists of a collection of distributions that factorize according to an underlying graph structure.
In this work we focus on undirected graphs, which represent Markov random fields \cite{WaiJor08}. Note that directed graphs represent Bayesian networks, and can always be transformed into a Markov random field \cite{KolFri09}.

Among the many representations for Markov random fields, the factor graph representation has been widely used due to its elegance and flexibility \cite{WaiJor08}, and is also used in this paper. Consider a graphical model with underlying factor graph $G=(V,\, F,\,E)$ where $V$ denotes the set of variable nodes, $F$ denotes the set and factor nodes,
and $E$ stands for the edges connecting them (see Figure \ref{fig:factor} for an example). In such a factor graph $G$, the neighbors of a node $j\in V$ consists of factor nodes, $N(j)\subset F$, and the neighbors of a factor node $\alpha\in F$ are variable nodes $N(\alpha)\subset V$. Therefore, $G$ is a bipartite  graph \cite{AsrDenHag98}.
 Each variable node $j\in V$ is associated with a random variable $x_j$ which can be either discrete or continuous. Here we consider only the discrete cases and assume that the random variable $x_j$ can take $d_j$ possible values. Each factor node $\alpha \in F$ corresponds to the dependence between the variable nodes connected to $\alpha$, which are compactly denoted by $\bx_\alpha:=\{x_j~;~ j\in N(\alpha)\}$. 
In Markov random fields with underlying factor graph $G$, the joint probability is assumed to be of the form
	\begin{equation}\label{eq:MRF}
		p(\bx):= p(x_1 , x_2 , \ldots, x_J) = \frac{1}{Z}\prod_{j \in V} \phi_j(x_j) \prod_{\alpha \in F} \psi_\alpha(\mx_\alpha)
	\end{equation}
where $\boldsymbol\phi_j$ is the node/local potential corresponding to node $j$, $\boldsymbol\psi_\alpha$ is the factor node potential corresponding to factor node $\alpha$, and $Z$ is a normalization constant. A factor node potential $\psi_\alpha$ describes the dependence between random variables in $\{x_j~;~ j\in N(\alpha)\}$.
The node potentials normally come from two sources: prior belief and evidence from measurements. In the latter, $\phi_j(x_j)$ is short for $\phi_j(x_j, y_j)$ \cite{YedFreWei03} with $y_j$ being the measurement. In cases where all the local potentials are induced by evidence, a more precise formula for the model is 
	\begin{equation}\label{eq:MRFoutput}
		p(\bx)= \frac{1}{Z}\prod_{j \in V} \phi_j(x_j,y_j) \prod_{\alpha \in F} \psi_\alpha(\mx_\alpha).
	\end{equation}
Since the measurement is usually specified in inference problems, $y_j$ is often neglected to simplify the notation. In principle, the node potentials can be fully absorbed into the factor potentials, that is, the joint distribution becomes $p(\bx)\propto\prod_{\alpha \in F} \psi_\alpha(\mx_\alpha)$. For the ease of presentation, we adopt the formulation \eqref{eq:MRF}.

Apart from factor graphs, another popular representation of PGMs is the standard graph where the nodes are all variables. In the standard graph representation, the dependence between the variables are fully captured by the edges of the graphs. The two representations are equivalent and one can transform one to another easily as the following example illustrates. 
\begin{ex}
The factor graph in Figure~\ref{fig:factor1} models the joint distribution 
	\[
		p(x_1,\ldots,x_6) = \frac{1}{Z} \psi_{\alpha_1} (\bx_{\alpha_1})\psi_{\alpha_2} (\bx_{\alpha_2})\psi_{\alpha_3} (\bx_{\alpha_3})\prod_{j=1}^6 \phi_j(x_j)
	\]
with $\bx_{\alpha_1} = \{x_1,x_2,x_4,x_5\}$, $\bx_{\alpha_2} = \{ x_3,x_5\}$, $\bx_{\alpha_3} = \{x_4,x_6\}$. To convert this into a standard graph representation, the dependence among variables induced by the three factors have to be translated to edges. This is straightforward for factors $\boldsymbol\psi_{\alpha_2}$ and $\boldsymbol\psi_{\alpha_3}$. The factor $\boldsymbol\psi_{\alpha_1}$ involves $4$ variables and is more complicated. Without further assumptions on the structure of this factor, it may induce dependence among all of the $4$ variables, and thus a complete graph connecting them is required (see Figure \ref{fig:factor2}). 
\end{ex}
\begin{figure}[tb]
\centering
\begin{subfigure}{.43\textwidth}
\centering
\includegraphics[scale=1]{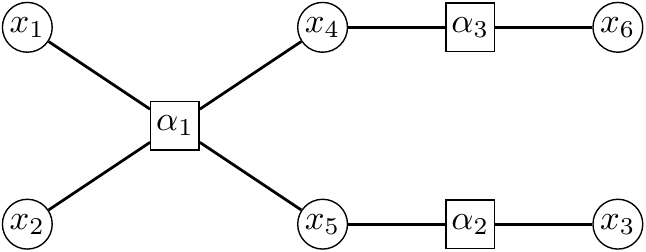}
\caption{}
\label{fig:factor1}
\end{subfigure}\hspace*{0.1cm}
\begin{subfigure}{.43\textwidth}
\centering
\includegraphics[scale=1]{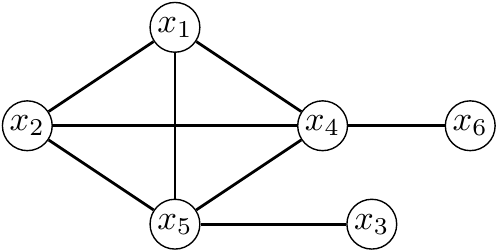}
\caption{}
\label{fig:factor2}
\end{subfigure} 
\caption{Factor graph representation vs standard graph representation.}
\label{fig:factor}
\end{figure}

The two fundamental problems in PGMs are inference and learning. Learning refers to estimating the underlying graphical models (often the parametrized factors) using available data sampled from the models. In inference problems, the parameters of the models are given. Instead, the goal is to infer the statistics of the node variables. The two main approaches to inference problems are 
maximum a posteriori estimation and Bayesian marginal inference \cite{WaiJor08}. 
Given a set of measurements $\{y_1,y_2,\ldots,y_J\}$, the aim of maximum a posteriori estimation is to find the most likely variable value $\{x_1,x_2,\ldots,x_J\}$ given the model and measurements. Instead, Bayesian inference calculates the posterior marginal distributions of each node. The focus of this paper is most relevant to Bayesian/marginal inference.

Formally, given a graphical model \eqref{eq:MRF}, the objective of Bayesian inference is to calculate marginal distributions $p(x_j)$ for $j\in V$. In cases where the nodes variables are discrete, the marginal $p(x_j)$ is defined to be 
	\begin{equation}\label{eq:inference}
		p(x_j) = \sum_{x_1,\ldots,x_{j-1},x_{j+1},\ldots,x_J} p(\bx) = \frac{1}{Z}\sum_{x_1,\ldots,x_{j-1},x_{j+1},\ldots,x_J} \prod_{i \in V} \phi_i(x_i) \prod_{\alpha \in F} \psi_\alpha (\mathbf{x}_\alpha).
	\end{equation}
The Bayesian inference problem can be reformulated as the optimization problem of minimizing 
	\begin{equation} \label{eq:free_energy}
         	\min_\mb \cF(\mb) := \cU(\mb) - \cH(\mb),
	\end{equation}
with 
	\begin{equation}\label{eq:var_av_energy}
        		\cU(\mb) = -\sum_\bx b(\bx) \left(\sum_{j \in V} \ln \phi_j(x_j)+\sum_{\alpha \in F} \ln\psi_\alpha (\mathbf{x}_\alpha)\right)
        \end{equation}
and
\begin{equation}\label{eq:var_entropy}
    \cH(\mb) = - \sum b(\bx)~ \ln~ b(\bx)
\end{equation}
over the space of all the probability distributions on $\bx$. By convention, $\cF, \cU, \cH$ are known as free energy, average energy and entropy respectively due to their similarity to thermodynamics \cite{Atk10}. To see this, we note that
the Kullback-Leibler (KL) divergence \cite{KulLei51} between $b(\bx)$ and $p(\bx)$ is 
	\[
		{\rm KL}(\mb\mid \bp) := \sum_\bx b(\bx) \ln \frac{b(\bx)}{p(\bx)} = \sum_\bx b(\bx) \ln \frac{b(\bx)}{1/Z\prod_{j \in V} \phi_j(x_j)\prod_{\alpha \in F} \psi_\alpha (\mathbf{x}_\alpha)}=
		F(\mb) +\ln Z.
	\]
Since the KL divergence is nonnegative and equals $0$ only if $\mb=\bp$, the unique minimizer of $\cF(\mb)$ is $\bp$ with the associated minimum being $-\ln Z$. 
The optimization formulation \eqref{eq:free_energy} of Bayesian inference is the basis for variational inference \cite{WaiJor08}, one of the most popular approximate inference techniques. In the variational inference approach, the approximate distribution $\mb$ is usually assumed to have some simple structure to ease the optimization, e.g., the mean field approximation $b(\bx)=b_1(x_1)\cdots b_J(x_J)$ \cite{JorGhaJaaSau99}. This work is not concerned with variational inference; \eqref{eq:free_energy} simply serves as a link to connect Bayesian inference with optimization. 
In the PGM literature \cite{WeiYanMel07,HazSha10}, it is common to introduce a temperature coefficient $\epsilon>0$ into \eqref{eq:free_energy}, which leads to a slightly more general optimization problem 
\begin{equation}\label{eq:variational_pgm}
    \min_\mb~ \cF(\mb) = \cU(\mb) - \epsilon \cH(\mb).
\end{equation}
It corresponds to the Bayesian inference for the model 
	\[
		p_\epsilon(\bx) = \frac{1}{Z} \prod_{j \in V} \phi_j(x_j)^\epsilon \prod_{\alpha \in F} \psi_\alpha (\mathbf{x}_\alpha)^\epsilon.
	\]
From an optimization point of view, \eqref{eq:variational_pgm} is a regularized version of the linear program
	\[
		\min_\mb~ \cU(\mb).
	\]
Interestingly, this linear program in fact corresponds to the maximum a posteriori problem \cite{HazSha10} for the model \eqref{eq:MRF}.

\subsection{Belief Propagation} \label{sec:bp}
In principle, Bayesian inference is achievable through the definition \eqref{eq:inference} by calculating the marginals using brute force summation.
The complexity of this summation however scales exponentially as the number of variable nodes $J$ goes up \cite{YedFreWei03}. Also the normalization factor $Z$ is extremely difficult to calculate when $J$ is large due to the same reason. 


During the last two decades, many methods have been developed to solve or approximately solve Bayesian marginal inference problems. 
One of the most widely used methods is a message-passing algorithm called Belief Propagation~\cite{Pea88}.
It updates the marginal distribution of each node through communications of beliefs/messages between them. 
In the factor graph representation, it reads 
	\begin{subequations}\label{eq:BP}
	\begin{eqnarray}
	m_{\alpha\rightarrow j} (x_j) &\propto& \sum_{\bx_\alpha\backslash x_j} \psi_\alpha(\bx_\alpha) 
	\prod_{i\in N(\alpha)\backslash j}n_{i\rightarrow \alpha}(x_i)
	\\
	n_{j\rightarrow \alpha}(x_j) &\propto& \phi_j(x_j) \prod_{\beta \in N(j)\backslash \alpha} m_{\beta\rightarrow j}(x_j),
	\end{eqnarray}
	\end{subequations}
where $m_{\alpha\rightarrow j} (x_j)$ denotes the message from factor node $\alpha$ to variable node $j$, and $n_{j\rightarrow \alpha}(x_j)$ represents the message from variable node $j$ to factor node $\alpha$. The symbol $\propto$ means ``proportional to'' and indicates that often a normalization is applied in the Belief Propagation algorithm. The messages in \eqref{eq:BP} are updated iteratively over the factor graph.
	
The Belief Propagation algorithm was first invented to solve Bayesian inference program over trees, in which case global convergence is guaranteed \cite{Pea88,YedFreWei03}. This method was later generalized to deal with inference problems involving general graphs under the name Loopy Belief Propagation \cite{YedFreWei01}. Even though there is no convergence proof and the algorithm does diverge in some occasions, it works well in practice and is widely adopted. 
When the algorithm converges, one can calculate the beliefs on the variables and factors by
	\begin{subequations}\label{eq:belief}
	\begin{eqnarray}
	b_j(x_j) &\propto& \phi_j(x_j) \prod_{\alpha \in N(j)} m_{\alpha\rightarrow j} (x_j)
	\\
	b_\alpha (\bx_\alpha) &\propto& \psi_\alpha(\bx_\alpha) \prod_{j\in N(\alpha)} n_{j\rightarrow \alpha} (x_j).
	\end{eqnarray}
	\end{subequations}
In cases where the factor graph has no cycles (i.e., it is a tree), the beliefs in \eqref{eq:belief} coincide with the true posterior marginals, that is, 
	\begin{subequations}\label{eq:marginalbelief}
	\begin{eqnarray}
		p(x_j) &=& b_j(x_j),~~\forall j \in V,~\forall x_j
		\\
		p(\bx_\alpha) &=& b_\alpha (\bx_\alpha),~~\forall \alpha\in F, ~\forall \bx_\alpha.
	\end{eqnarray}
	\end{subequations}
	
For general graphs with cycles, convergence is not guaranteed and even if it does converge, the beliefs in \eqref{eq:belief} are only approximations of the true marginals $p(x_j),\,p(\bx_\alpha)$. A remarkable discovery \cite{YedFreWei01,YedFreWei05} related to (Loopy) Belief Propagation is that if the updates \eqref{eq:BP} converge, then the beliefs in \eqref{eq:belief} form a fixed point of the Bethe free energy \cite{YedFreWei01,YedFreWei05} 
\begin{equation}\label{eq:Bethe_energy}
    \cF_{\rm Bethe}(\mb) =  \cU_{\rm Bethe}(\mb) - \cH_{\rm Bethe}(\mb),
    \end{equation}
where $\cU_{\rm Bethe}(\mb)$ is the Bethe average energy
\begin{equation}\label{eq:Bethe_average_ebergy}
    \cU_{\rm Bethe}(\mb) = - \sum_{\alpha \in F} \sum_{\bx_\alpha} b_\alpha(\bx_\alpha) \ln \psi_\alpha(\bx_\alpha)-\sum_{j\in V}\sum_{x_j} b_j(x_j)\ln \phi_j(x_j)
\end{equation}
and $\cH_{\rm Bethe}(\mb)$ is the Bethe entropy 
\begin{equation}\label{eq:Bethe_entropy}
    \cH_{\rm Bethe}(\mb) = - \sum_{\alpha \in F} \sum_{\bx_\alpha} b_\alpha(\bx_\alpha) \ln b_\alpha(\bx_\alpha) + \sum_{j \in V} (N_j - 1) \sum_{x_j} b_j(x_j) \ln b_j(x_j)
\end{equation}
with $N_j$ denoting the degree of the variable node $j$, i.e., $N_j=|N(j)|$. In \eqref{eq:Bethe_energy}, we define $\mb=\{\mb_j,\,\mb_\alpha~:~ j\in V,\,\alpha\in F\}$. This is different to $\mb$ in \eqref{eq:free_energy}, which is a $J$-mode tensor. For the sake of conciseness, by abuse of notation, we use $\mb$ in both settings. For a factor tree, the two are connected through the relation $b(\bx)\sim (\prod_{\alpha\in F} b_\alpha(\bx_\alpha)) (\prod_{j\in V} b_j(x_j)^{1-N_j})$ \cite{WaiJor08}. In terms of Bethe free energy, the Bayesian inference problem reads
	\begin{subequations}\label{eq:Betheinference}
	\begin{eqnarray}
	 \min_{\mb} &&\cF_{\rm Bethe}(\mb) \label{eq:Betheinferencea}
    \\ 
   \text{subject to} && \sum_{\bx_\alpha \backslash x_j} b_\alpha(\bx_\alpha) = b_j(x_j), \quad  \forall j \in V, \alpha \in N(j), \label{eq:Betheinferenceb}
   \\&& \sum_{\bx_\alpha} b_\alpha(\bx_\alpha)= 1, \quad  \forall \alpha \in F. \label{eq:Betheinferencec}
	\end{eqnarray}
	\end{subequations}
The constraint \eqref{eq:Betheinferenceb} is to ensure that $\mb_\alpha, \mb_j$ are compatible and \eqref{eq:Betheinferencec} is to guarantee that they are in the probability simplex. It is easy to check that when the factor graph has no cycles, the Bethe free energy \eqref{eq:Bethe_energy} is strictly convex in the feasible set defined by the constraints \eqref{eq:Betheinferenceb}-\eqref{eq:Betheinferencec}, and is equal to the free energy \eqref{eq:free_energy}, i.e., $\cF_{\rm Bethe} = \cF$. Thus, \eqref{eq:Betheinference} is again a convex optimization problem. For general graphs, the Bethe free energy serves as a good approximation of the free energy \cite{YedFreWei05}, but is no longer convex.

\section{Multimarginal optimal transport as Bayesian inference}
\label{sec:momt_bayes}
%

Multimarginal optimal transport (MOT) extends the OT framework \eqref{eq:omt_bi_discrete} to the setting involving multiple distributions. In particular, in MOT, one aims to find a transport plan among a set of marginals $\boldsymbol\mu_1,\dots,\boldsymbol\mu_J$ with $J\geq 2$. 
In this setting, the transport cost is encoded in a tensor $\bC = [C(i_1,i_2,\ldots, i_J)]\in \mR^{d_1\times d_2\times\cdots\times d_J}$ where $C(i_1,i_2,\ldots, i_J)$ denotes the unit transporting cost corresponding to the locations $i_1, i_2,\ldots, i_J$, and the transport plan is in the same way described by a $J$-mode tensor  $\bB \in \mR_+^{d_1\times d_2\times\cdots\times d_J}$.

For a given transport plan $\bB$, the total cost of transportation is 
	\[
		\langle \bC, \bB\rangle := \sum_{i_1,i_2,\ldots, i_J} C(i_1,i_2,\ldots, i_J) B(i_1,i_2,\ldots, i_J).
	\]
Thus, similar to \eqref{eq:omt_bi_discrete}, MOT has a linear programming formulation
\begin{equation} \label{eq:omt_multi_discrete}
\begin{aligned}
\min_{\bB \in \mR_+^{d_1\times \dots \times d_J}} & \langle \bC, \bB \rangle  \\
\text{ subject to } & P_j (\bB) = \boldsymbol\mu_j,  \text { for } j \in \Gamma,
\end{aligned}
\end{equation}
where $\Gamma\subset \{1,2,\dots,J\}$ is an index set specifying which marginal distributions are given, and the projection on the $j$-th marginal of $\bB$ is computed as
\begin{equation} \label{eq:proj_discrete}
P_j(\bB) = \sum_{i_1,\dots,i_{j-1},i_{j+1},i_J} B (i_1,\dots,i_{j-1},i_j,i_{j+1},\dots,i_J).
\end{equation}
Note that the standard bi-marginal OT problem \eqref{eq:omt_bi_discrete} is a special case of the MOT problem \eqref{eq:omt_multi_discrete} with $J=2$ and $\Gamma=\{1,2\}$.

In the original MOT formulation \cite{GanSwi98,Pas15}, constraints are given on all the marginal distributions, viz., the index set $\Gamma = \{1,2,\dots,J\}$. However, in many applications \cite{Pas15,Nen16,HaaRinChe20}, only a subset of marginal distributions are explicitly given. For instance, the Barycenter problem \cite{AguCar11} is a MOT where the target distribution is not given.
In this work we consider the setting where constraints are only imposed on a subset of marginals, i.e., $\Gamma \subset \{1,2,\dots,J\}$.

\subsection{Entropy regularized MOT}
Although MOT \eqref{eq:omt_multi_discrete} is a standard linear program, its complexity grows exponentially as $J$ increases. 
This computational burden can be partly alleviated in an analogous manner as for the classical bi-marginal problem \eqref{eq:omt_bi_discrete}, which again yields an 
iterative scaling algorithm.
In particular, similarly to \eqref{eq:omt_bi_discrete_regularized}, one can add an entropy term
\begin{equation}
\cH(\bB) = - \sum_{i_1,\dots,i_J} B(i_1,\dots,i_J)  \ln~B(i_1,\dots,i_J)
\end{equation}
to \eqref{eq:omt_multi_discrete} to regularize the problem, resulting in the strictly convex optimization problem
\begin{equation} \label{eq:omt_multi_regularized}
\begin{aligned}
\min_{\bB \in \mR^{d_1\times \dots \times d_J}} & \langle \bC, \bB \rangle - \epsilon \cH(\bB) \\
\text{ subject to } & P_j (\bB) = \boldsymbol\mu_j,  \text { for } j \in \Gamma
\end{aligned}
\end{equation}
with $\epsilon>0$ being a regularization parameter. 

For the bi-marginal case, \eqref{eq:omt_multi_regularized} reduces to problem~\eqref{eq:omt_bi_discrete_regularized}. The iterative scaling algorithm \eqref{eq:Sinkhorn2d} can be generalized to the multi-marginal setting \cite{FraLor89} in order to solve \eqref{eq:omt_multi_regularized}.
From an optimization perspective, the iterative scaling algorithm amounts to a coordinate ascent method \cite{Wri15} in the dual problem of \eqref{eq:omt_multi_regularized}. 
The introduction of the entropy term in \eqref{eq:omt_multi_regularized} allows for closed-form expressions for the updates of the dual variables \cite{ElvHaaJakKar20}. 
Utilizing Lagrangian duality theory, one can show that the optimal solution to \eqref{eq:omt_multi_regularized} is of the form
\begin{equation}
\bB = \bK \odot \bU,
\end{equation}
where $\odot$ denotes element-wise multiplication and the tensors are given by 
	\begin{equation}\label{eq:K}
		\bK = \exp(- \bC/\epsilon)
	\end{equation}
and 
    \begin{equation}\label{eq:U}
    \bU= \bu_1 \otimes \bu_2 \otimes \dots \otimes \bu_J,
    \end{equation}
where the vectors $\bu_j \in \mathbb{R}^{d_j}$ are given by
\begin{equation} \label{eq:u_multi_omt}
\bu_j = \begin{cases} \exp\left( -\frac{1}{J} - \frac{\boldsymbol\lambda_j}{\epsilon} \right),& \text{ if } j \in \Gamma \\
\exp\left( -\frac{1}{J} \right) \ett,& \text{ otherwise,} \end{cases}
\end{equation}
and $\boldsymbol\lambda_j\in \mathbb{R}^{d_j}$ is the dual variable corresponding to the constraint $P_j (\bB) = \boldsymbol\mu_j$ on the $j$-th marginal.
Moreover, the dual of \eqref{eq:omt_multi_regularized} is
\begin{equation} \label{eq:multi_omt_dual}
\max_{\{\boldsymbol\lambda_j, j\in \Gamma\}}  -\epsilon \langle\bK, \bU \rangle - \sum_{j \in \Gamma} \boldsymbol\lambda_j^T \boldsymbol\mu_j.
\end{equation}
We emphasis that in \eqref{eq:multi_omt_dual}, $\bU$ is a function of the multipliers $\{\boldsymbol\lambda_j, j\in \Gamma\}$ as defined in \eqref{eq:U}-\eqref{eq:u_multi_omt}.

The iterative scaling algorithm iteratively updates the vectors $\bu_j$, for $j\in\Gamma$, in \eqref{eq:u_multi_omt} according to
\begin{equation} \label{eq:sinkhorn_multi}
\bu_j \leftarrow \bu_j \odot \boldsymbol\mu_j ./ P_j(\bK \odot \bU),
\end{equation}
for all $j\in\Gamma$. 
For future reference, we summarize the steps in Algorithm~\ref{alg:sinkhorn}. 
\begin{algorithm*}[tb]
   \caption{Iterative Scaling Algorithm for MOT}
   \label{alg:sinkhorn}
\begin{algorithmic}
    \STATE Compute $\bK=\exp(- \bC/\epsilon)$
   \STATE Initialize $\bu_1, \bu_2, \ldots, \bu_J$ to $\exp(-\frac{1}{J})\mathbf{1}$
   \WHILE{not converged}
   \FOR{$j \in \Gamma$}
        \STATE Compute $\bU= \bu_1 \otimes \bu_2 \otimes \dots \otimes \bu_J$
        \STATE Update $\bu_j$ as $\bu_j \leftarrow \bu_j \odot \boldsymbol\mu_j ./ P_j(\bK \odot \bU)$
    \ENDFOR
    \ENDWHILE
\end{algorithmic}
\end{algorithm*}
The Iterative Scaling algorithm (Algorithm~\ref{alg:sinkhorn}) is a special case of the iterative Bregman projection algorithm \cite{BauLew00,BenCarCut15}, which itself is a special case of a dual block coordinate ascent method \cite{Tse90,LuoTse93}, and thus enjoys a global convergence guarantee \cite{BauLew00,LuoTse93}.

Note that the standard Sinkhorn iterations \eqref{eq:Sinkhorn2d}
for the two-marginal case \eqref{eq:omt_bi_discrete_regularized} is a special case of Algorithm~\ref{alg:sinkhorn} when $J=2$ and $\Gamma=\{1, 2\}$. Indeed, in this case, Algorithm \ref{alg:sinkhorn} boils down to iterating
	\[
		\bu_1 \leftarrow \bu_1 \odot \boldsymbol\mu_1 ./ P_1(\bK \odot \bU), \qquad \bu_2 \leftarrow \bu_2 \odot \boldsymbol\mu_2 ./ P_2(\bK \odot \bU).
	\]
With $P_1(\bK \odot \bU) = \diag(\bu_1)\bK\diag(\bu_2)\ett = \diag(\bu_1) (\bK \bu_2)=\bu_1\odot (\bK\bu_2)$ and similarly $P_2(\bK \odot \bU)  
= \bu_2\odot(\bK^T \bu_1)$, it follows
	\[
		\bu_1 \leftarrow \boldsymbol\mu_1./\bK\bu_2, \qquad 
		\bu_2  \leftarrow  \boldsymbol\mu_2./\bK^T\bu_1,
	\]
which coincide with \eqref{eq:Sinkhorn2d}. 

Although Algorithm~\ref{alg:sinkhorn} is easy to implement and considerably faster than general linear programming solvers, its complexity still scales exponentially as $J$ grows since the number of elements in $\bB$ are $d_1d_2\ldots d_J$. The computational bottleneck of it lies in the calculation of the projections $P_j(\bB)$, for $j \in \Gamma$, in \eqref{eq:proj_discrete}. Generally, this computational burden is inevitable. However, in some cases it is possible to utilize structures in the cost tensor $\bC$ to make the computation of the projections more accessible \cite{BenCarCut15, ElvHaaJakKar20, HaaRinChe20}. In Section~\ref{sec:GraphMOT} we consider MOT problems with cost tensors that can be decomposed according to a graph. This graphical structure allows us to leverage the Bayesian inference tools \cite{KolFri09} in PGMs to compute the projections efficiently. Other than providing a workhorse for solving MOT problems with graphical structured cost, this connection between MOT and PGMs also presents new elements and perspective to Bayesian inference in PGMs, which is discussed in details in Section \ref{subsec:MOT_trees}. 

\subsection{MOT with graphical structures}
\label{sec:GraphMOT}
Consider the cases where the cost tensor $\bC$ can be decomposed according to a factor graph. More specifically, the cost tensor $\bC$ has the form
\begin{equation}\label{eq:cost_structure}
    C(\bx) = \sum_{\alpha\in F} C_{\alpha}(\bx_\alpha),
\end{equation}
where $F$ denotes the set of factors of a graph. Here, to be consistent with the notations in PGMs, we write the cost of associating $i_1, i_2, \ldots, i_J$ by $C(\bx)=C(x_1,x_2,\ldots, x_J)$ instead of $C(i_1, i_2, \ldots, i_J)$, but the two have exactly the same meaning; both $x_j$ and $i_j$ take values in a set with $d_j$ elements. Thus, by abuse of notation, we use $C(\bx)$ and $C(i_1, i_2, \ldots, i_J)$ interchangeably. 

A graph structured cost tensor \eqref{eq:cost_structure} occurs in various applications of the OT framework \cite{BenCarCut15, ElvHaaJakKar20}. For instance, in Barycenter problems \cite{AguCar11}, the cost $\bC$ can be decomposed into the sum of pairwise costs between the target distribution and each given marginal distribution. For general cost functions, it might be possible to approximate them using the structured cost \eqref{eq:cost_structure}. Thus the framework we establish can also be viewed as an efficient method to approximate the solution to any MOT problem. 

Denote the factor graph associated with the cost \eqref{eq:cost_structure} by $G=(V, F, E)$. Then the $j$-th mode of $\bC$ corresponds to node $j\in V$ and the marginal distribution of the $j$-th mode is the same as the marginal distribution of $x_j$ at node $j$. In this paper, we only consider the cases where the factor graph $G$ is connected but does not have any loop, that is, $G$ is a factor tree. We associate the cost $\bC$ with a probabilistic graphical model
	\[
		p(\bx) = \frac{1}{Z} \prod_{\alpha\in F} K_\alpha(\bx_\alpha)
	\]
where 
	\begin{equation}\label{eq:Kalpha}
		K_\alpha(\bx_\alpha) = \exp (-C_\alpha(\bx_\alpha)/\epsilon).
	\end{equation}
Clearly, $\bK$ in \eqref{eq:K} has the form 
	\[
		\bK = [K(i_1,i_2,\ldots,i_J)]= [K(\bx)] = [\prod_{\alpha\in F} K_\alpha(\bx_\alpha)],
	\]
and 
	\begin{equation}\label{eq:MGM}
		\bK\odot\bU = [K(\bx)U(\bx)] = [\left(\prod_{\alpha\in F} K_\alpha(\bx_\alpha)\right)\left(\prod_{j\in V} u_j(x_j)\right)].
	\end{equation}
From a PGM point of view, the (transformed) Lagrangian multipliers $\bu_j$, for $j\in\Gamma$, introduced by the Iterative Scaling algorithm are local potentials of the modified graphical model $K(\bx)U(\bx)$. The Lagrangian approach of solving the constrained optimization problem \eqref{eq:omt_multi_regularized} seeks multipliers $\bu_j$, for $j\in\Gamma$, such that the tensor $\bB=\bK\odot\bU$ satisfies all the constraints $P_j(\bB) = \boldsymbol\mu_j$, for $j\in\Gamma$. Thus, in the language of PGMs, to solve the MOT problem \eqref{eq:omt_multi_regularized}, one can search for a proper set of artificial local potentials $\bu_j$, for $j\in\Gamma$, such that the modified graphical model $K(\bx)U(\bx)$ in \eqref{eq:MGM} has the specified marginal distribution $\boldsymbol\mu_j$ on the $j$-th variable node for each $j\in \Gamma$. Note that $\bu_j = \exp(-1/J)\mathbf{1}$ is a uniform potential for all $j\notin \Gamma$ and thus does not affect the graphical model $\bK\odot\bU$. 

For fixed multipliers $\bu_1, \bu_2, \ldots, \bu_J$, calculating (with proper normalization) the projection $P_j(\bK \odot \bU)$ is exactly a Bayesian inference problem of inferring the $j$-th variable node over the modified graphical model $K(\bx)U(\bx)$. When $G$ does not have any loops, a condition we assume throughout, Bayesian inference can be achieved efficiently using the Belief Propagation algorithm. 
Generally, the marginal constraints $P_j(\bB)=\boldsymbol\mu_j$ can be imposed on any variable node $j\in V$. However, a marginal constraint on a non-leaf node will decompose the MOT problem \eqref{eq:omt_multi_regularized} into several independent MOT problems with constraints only on leaf nodes, see \cite{HaaRinChe20}. Thus, without loss of generality, we assume marginal constraints on leaf nodes only, that is, $\Gamma\subset L$ where $L\subset V$ denotes the set of leaf nodes of $G$.

\begin{ex}
Figure~\ref{fig:general_factor_clean} depicts a factor graph with leaf nodes $L = \{1,2,3,6 \}$. The shaded nodes in the figure represent the fixed distribution variables, thus, in this example, $\Gamma= \{1, 2, 3\} \subset L$. 
\end{ex}
\begin{figure}[tb]	
\centering
\includegraphics[scale=0.9]{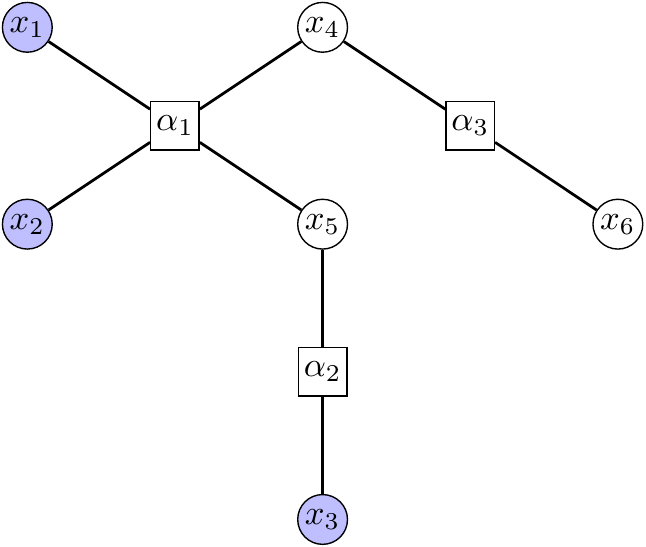}
\caption{Factor graph with some marginal distribution constraints on nodes $\{x_1,x_2,x_3\}$.}
\label{fig:general_factor_clean}
\end{figure}

Leveraging the graphical structure \eqref{eq:cost_structure} of the cost, based on the discussions above, we obtain a simple strategy to solve the MOT problem \eqref{eq:omt_multi_regularized}: We apply the Iterative Scaling algorithm and utilize the Belief Propagation algorithm to carry out the computation of $P_j(\bK \odot \bU)$ with the current multiplier $\bU$. The acceleration is tremendous for MOT problems with a large number of marginals; the Belief Propagation algorithm scales well for large problem while the complexity of the brute force projection using definition \eqref{eq:proj_discrete} grows exponentially as the dimension increases.
It turns out that some more tricks can be adopted to further improve the speed of the projection step $P_j(\bK \odot \bU)$. The full algorithm, which we call Iterative Scaling Belief Propagation (ISBP) algorithm will be presented and discussed in details in Section \ref{sec:ISBP}. 


\subsection{MOT and Bayesian inference}
\label{subsec:MOT_trees}

In the previous section, we have seen that in cases where the cost tensor $\bC$ in MOT problem \eqref{eq:omt_multi_regularized} has a graphical structure, one can take advantage of PGM methods, in particular the Belief Propagation algorithm, to accelerate the Iterative Scaling algorithm. In this section, we establish further connections between MOT and PGMs. These links add novel components to both the MOT theory and PGM theory. These connections also bring new insight and interpretation of Iterative Scaling Belief Propagation.

%
%
%

Clearly, the objective function of the entropy regularized MOT problem \eqref{eq:omt_multi_regularized} is exactly the free energy $\cF$ in \eqref{eq:variational_pgm} with 
	\begin{equation}\label{eq:psiphi}
		\psi_\alpha (\mx_\alpha)=\exp (-C_\alpha(\mx_\alpha)),\ \forall \alpha\in F,\forall \mx_\alpha \qquad \phi_j(x_j) = 1,\ \forall j\in V,\forall x_j.
	\end{equation}
Thus, the MOT problem \eqref{eq:omt_multi_regularized} can be written as 
\begin{equation}\label{eq:momt_free_energy}
\begin{aligned}
     \min_{\bB\in\mR_+^{d_1\times \dots \times d_J}} &~  \cF(\bB) \\ 
     \text{subject to} &~ P_j(\bB) =  \boldsymbol\mu_j,\quad \forall j \in \Gamma.
\end{aligned}
\end{equation}
Therefore, the entropic regularized MOT problem~\eqref{eq:omt_multi_regularized} with cost function that decouples according to a graph structure as in \eqref{eq:cost_structure} is equivalent to a Bayesian inference problem in a PGM with additional constraints on the marginal distributions of a set of variable nodes. In other words, \eqref{eq:momt_free_energy} is a constrained version of a Bayesian inference problem.

On the other hand, any Bayesian inference problem in a PGM can be rewritten in the constrained form \eqref{eq:momt_free_energy}. More specifically, consider the problem of inferring the posterior distribution of 
	\[
		\frac{1}{Z} \prod_{j \in \Gamma} \phi_j(x_j,y_j) \prod_{\alpha \in F} \psi_\alpha (\mathbf{x}_\alpha),
	\]
where $y_j$ is the observation associated with the variable node potential $\phi_j$. When the observations are fixed, say $y_j = \hat y_j$, then the standard Bayesian inference method replaces $\phi_j(x_j,y_j)$ by a local potential $\phi_j(x_j)$ and infers the marginal distributions of the resulting graphical model with only the nodes $\{x_j,\,\mx_\alpha\}$. Alternatively, the measurement $y_j = \hat y_j$ can be viewed as constraints on the node $y_j$ of an augmented graphical model which includes also the observation variable node $y_j$. In particular, the constraint is of the form $p(y_j) = \delta (y_j-\hat y_j)$, where $\delta$ denotes Dirac distribution. Thus, the posterior distribution can also be obtained by solving the constrained Bayesian inference problem \eqref{eq:momt_free_energy} over the augmented graphical models under the constraints that $p(y_j) = \delta (y_j-\hat y_j)$, for $j\in \Gamma$. This equivalence is illustrated in Figure \ref{fig:constrainedinference}. 
Therefore, from this point of view, the constrained Bayesian inference problem \eqref{eq:momt_free_energy} can also be viewed as a generalization of standard Bayesian inference. 

\begin{figure}[tb]
\centering
\begin{subfigure}{.43\textwidth}
\centering
\includegraphics[scale=0.9]{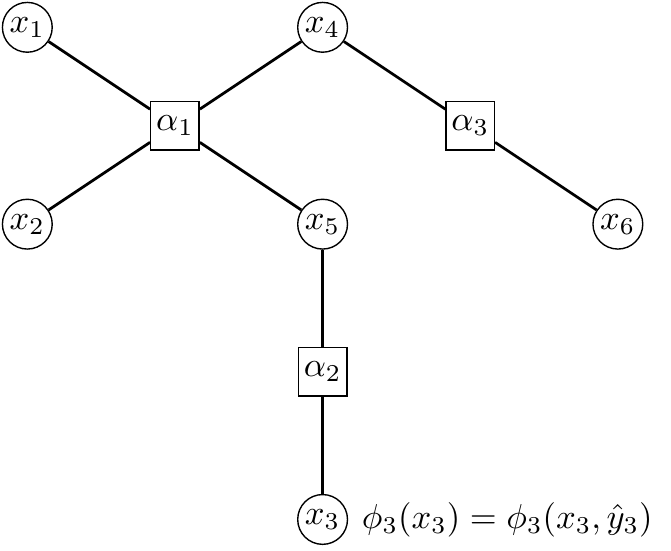}
\caption{Standard}
\end{subfigure} \hspace*{0.5cm}
\begin{subfigure}{.43\textwidth}
\centering
\includegraphics[scale=0.9]{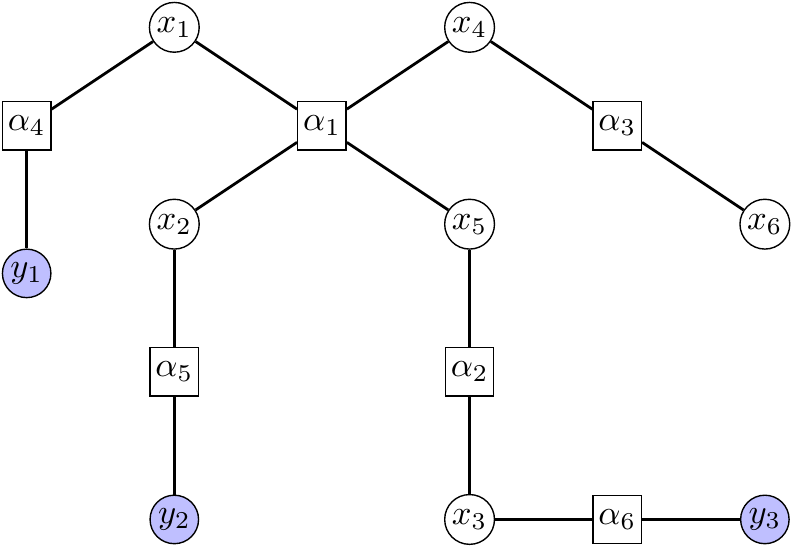}
\caption{Constrained}
\end{subfigure}
\caption{Equivalence between standard Bayesian inference and constrained Bayesian inference: (a) The local potentials of variables $x_1,~x_2,~x_3$ are induced by measurements $y_1=\hat y_1,~y_2=\hat y_2,~y_3=\hat y_3$ respectively, namely, $\phi_1(x_1) = \phi_1(x_1,\hat y_1),~ \phi_2(x_2) = \phi_2(x_2,\hat y_2), ~\phi_3(x_3) = \phi_3(x_3,\hat y_3)$; (b) The graph is augmented by the nodes $y_1,~y_2,~y_3$ and factors
$\psi_{\alpha_4} = \phi_1(x_1,y_1)$, $\psi_{\alpha_5} = \phi_2(x_2,y_2)$, and $\psi_{\alpha_6} = \phi_3(x_3,y_3)$. The measurements become marginal constraints $p(y_1)=\delta(y_1-\hat y_1), ~p(y_2)=\delta(y_2-\hat y_2), ~p(y_3)=\delta(y_3-\hat y_3)$.
}
\label{fig:constrainedinference}
\end{figure}

When the underlying factor graph associated with the cost function \eqref{eq:cost_structure} is in fact a factor tree, the free energy $\cF(\bB)$ is equal to the Bethe free energy (see Section \ref{sec:bp})
\begin{equation}\label{eq:Bethe_energy_momt}
    \cF_{\rm Bethe}(\bB) = - \sum_{\alpha \in F} \sum_{\bx_\alpha} B_\alpha(\bx_\alpha) \ln \psi_\alpha(\bx_\alpha) +\epsilon\,\sum_{\alpha \in F} \sum_{\bx_\alpha} B_\alpha(\bx_\alpha) \ln B_\alpha(\bx_\alpha)- \epsilon\,\sum_{j \in V} (N_j - 1) \sum_{x_j} B_j(x_j) \ln B_j(x_j),
    \end{equation}
where $\bB_\alpha$ is the marginal distribution on factor node $\alpha$ and $\bB_j$ denotes the marginal on variable node $j$, namely, $B_\alpha(\mx_\alpha) = \sum_{\mx\backslash \mx_\alpha} B(\mx)$, and $B_j(x_j) =  \sum_{\mx\backslash x_j} B(\mx)$. In \eqref{eq:Bethe_energy_momt}, $\bB$ is the collection of marginal distributions, that is, $\bB=\{\bB_j, \bB_\alpha~:~ j\in V, \alpha\in F\}$. This is different from the $J$-mode tensor $\bB$ in \eqref{eq:momt_free_energy}. However, with slight abuse of notation, we use the same symbol $\bB$ in both contexts. Again, due to tree structure, the two have a one-to-one correspondence with each other.

The marginals $\bB_\alpha$ and $\bB_j$ capture only local information around a factor variable or node variable and they have to satisfy certain conditions in order to be feasible marginal distributions of some joint distribution. In particular, they have to be compatible in the sense 
	\[
		\sum_{\bx_\alpha \backslash x_j} B_\alpha(\bx_\alpha) = B_j(x_j),\quad \forall j \in V,\, \forall x_j.
	\]
In terms of local marginals, the constraints in \eqref{eq:momt_free_energy} read
	\[
		B_j(x_j) = \mu_j (x_j),\quad \forall j\in \Gamma,\, \forall x_j.
	\]
Therefore, the MOT problem \eqref{eq:momt_free_energy} can be reformulated as
\begin{subequations}\label{eq:momt_bethe}
\begin{eqnarray}
    \min_{\bB} &&\cF_{\rm Bethe}(\bB) \label{eq:momt_bethe_a}
    \\
   \mbox{subject to} && B_j(x_j) = \mu_j(x_j), \quad \forall j \in \Gamma,~\forall x_j \label{eq:momt_bethe_b} 
   \\
   && \sum_{\bx_\alpha \backslash x_j} B_\alpha(\bx_\alpha) = B_j(x_j), \quad \forall j \in V, \alpha \in N(j),~\forall x_j \label{eq:momt_bethe_c} 
   \\ \label{eq:momt_bethe_d} &&\sum_{\bx_\alpha} B_\alpha(\bx_\alpha)= 1, \quad \forall \alpha \in F, 
\end{eqnarray}
\end{subequations}
where the last constraint \eqref{eq:momt_bethe_d}  is to ensure that the optimization variables $\{\bB_j, \bB_\alpha~:~ j\in V, \alpha\in F\}$ are in the probability simplex. Since the Bethe free energy is convex for factor trees, and the constraints are linear, Problem \eqref{eq:momt_bethe} is a convex optimization problem. One advantage of \eqref{eq:momt_bethe} over \eqref{eq:momt_free_energy} is that the size of optimization variables in \eqref{eq:momt_bethe} is considerably smaller than that in \eqref{eq:momt_free_energy}. More specifically, the optimization variables of \eqref{eq:momt_bethe} are local marginals which are either vectors $\bB_j$ or low-dimensional tensors $\bB_\alpha$, which is in contrast to the high-dimensional $J$-mode tensor $\bB$ in \eqref{eq:momt_free_energy}.

\subsection{Iterative Scaling Belief Propagation algorithm}\label{sec:ISBP}
In this section, we present the full Iterative Scaling Belief Propagation algorithm for the entropy regularized MOT problem \eqref{eq:omt_multi_regularized} (or equivalently \eqref{eq:momt_free_energy} and \eqref{eq:momt_bethe}). We start with a characterization of the solution to \eqref{eq:momt_bethe}. 


\begin{thm} \label{thm:is_bp}
The solution to the MOT problem~\eqref{eq:momt_bethe} is given by
\begin{subequations} \label{eq:solutionBethe}
\begin{eqnarray}\label{eq:solutionBethe1}
    B_\alpha (\bx_\alpha)  &\propto& K_\alpha (\bx_\alpha) \prod_{j \in N(\alpha)} n_{j\rightarrow \alpha}(x_j), \quad \forall \alpha \in F 
	\\\label{eq:solutionBethe2}
   B_j (x_j)  &\propto&  \prod_{\alpha \in N(j)}  m_{\alpha \rightarrow j}(x_j),~ \forall j\notin \Gamma
    \\\label{eq:solutionBethe3}
      B_j (x_j) &=& \mu_j(x_j),~ \forall j\in \Gamma 
\end{eqnarray}
\end{subequations}
where $m_{\alpha \rightarrow j},\,n_{j\rightarrow \alpha}$ are fixed points of the following iterations
\begin{subequations}\label{eq:is_bp_momt}
	\begin{eqnarray}
m_{\alpha\rightarrow j} (x_j) &\propto& \sum_{\bx_\alpha \backslash x_j} K_\alpha(\bx_\alpha) \prod_{i\in N(\alpha)\backslash j}n_{i\rightarrow \alpha}(x_i); \quad \forall j \in V,~  \forall \alpha \in N(j),~ \forall x_j, \label{eq:is_bp_momt1} \\
n_{j\rightarrow \alpha}(x_j) &\propto& \prod_{\beta \in N(j)\backslash \alpha} m_{\beta\rightarrow j}(x_j); \quad \forall j \notin \Gamma,~\forall \alpha \in N(j), ~\forall x_j, \label{eq:is_bp_momt2} \\
n_{j\rightarrow \alpha}(x_j) &\propto&\mu_j(x_j)  (m_{\alpha\rightarrow j}(x_j))^{-1}; \quad \forall j \in \Gamma,~\forall x_j. \label{eq:is_bp_momt3}
	\end{eqnarray}
\end{subequations}
Here $\propto$ indicates that a normalization step is needed.
\end{thm}
\begin{proof}
In order to solve the constrained optimization problem \eqref{eq:momt_bethe}, we introduce Lagrange multipliers $\eta_\alpha$ for the simplex constraints \eqref{eq:momt_bethe_d}, $\boldsymbol\lambda_{j,\alpha}$ for the marginalization compatibility constraints \eqref{eq:momt_bethe_c}, and $\boldsymbol\nu_j$ for the fixed-marginal constraints \eqref{eq:momt_bethe_b}, yielding the Lagrangian
\begin{equation}\label{eq:momt_lagrangian}
    \cL = \frac{1}{\epsilon}\cF_{\rm Bethe}(\bB) + \sum_{\alpha} \eta_\alpha \left( \sum_{\bx_\alpha} B_\alpha(\bx_\alpha) - 1 \right) + \sum_{j,x_j} \sum_{\alpha \in N(j)} \lambda_{j,\alpha}(x_j) \left(\sum_{\mx_\alpha \backslash x_j} B_{\alpha}(\bx_\alpha) -B_j(x_j)  \right) + \sum_{j \in \Gamma} \sum_{x_j} \nu_j(x_j) \left( B_j(x_j) - \mu_j(x_j)  \right).
\end{equation}
Note that we have used a scaled version $\frac{1}{\epsilon}\cF_{\rm Bethe}$ of the objective function. In view of \eqref{eq:psiphi}, \eqref{eq:Bethe_energy_momt} and \eqref{eq:Kalpha},
	\[
		\frac{1}{\epsilon}\cF_{\rm Bethe}(\bB) = - \sum_{\alpha \in F} \sum_{\bx_\alpha} B_\alpha(\bx_\alpha) \ln K_\alpha(\bx_\alpha) +\sum_{\alpha \in F} \sum_{\bx_\alpha} B_\alpha(\bx_\alpha) \ln B_\alpha(\bx_\alpha)-\sum_{j \in V} (N_j - 1) \sum_{x_j} B_j(x_j) \ln B_j(x_j).
	\]
 Setting the derivatives of the Lagrangian with respect to the local marginals $\bB_\alpha$ and $\bB_j$ to zero, we get that the minimizer satisfies
\begin{subequations}\label{eq:proofBethe}
\begin{eqnarray}\label{eq:proofBethe1}
    B_\alpha(\bx_\alpha)  &=&  K_\alpha(\bx_\alpha)~ \exp \left(  -1 - \sum_{j \in N(\alpha)} \lambda_{j,\alpha}(x_j)-\eta_\alpha  \right),
	\\\label{eq:proofBethe2}
    B_j(x_j) &=& \exp\left(-1 - \frac{1}{N_j - 1} \sum_{\alpha \in N(j)} \lambda_{j,\alpha}(x_j)  \right) \quad ~\mbox{if}~ N_j>1
    \\\label{eq:proofBethe3}
    0 &=& \sum_{\alpha \in N(j)} \lambda_{j,\alpha}(x_j) \quad ~\mbox{if}~ N_j=1, j \notin \Gamma
    \\\label{eq:proofBethe4}
    0 &=& \sum_{\alpha \in N(j)} \lambda_{j,\alpha}(x_j)-\nu_j(x_j) \quad ~\mbox{if}~ j \in \Gamma
\end{eqnarray}
\end{subequations}
Denote
	\begin{subequations}\label{eq:mn}
	\begin{eqnarray}\label{eq:n}
	n_{j\rightarrow \alpha} (x_j) &:=& \exp(-\lambda_{j,\alpha} (x_j))
	\\\label{eq:m}
	m_{\alpha\rightarrow j} (x_j) &:=& \sum_{\mx_\alpha \backslash x_j} K_\alpha(\mx_\alpha) \prod_{i\in N(\alpha)\backslash j} n_{i\rightarrow \alpha} (x_i).
	\end{eqnarray}
	\end{subequations}
The relation \eqref{eq:solutionBethe1} follows immediately from \eqref{eq:proofBethe1} and \eqref{eq:n}. This together with the constraint \eqref{eq:momt_bethe_c} and \eqref{eq:m} leads to
	\begin{equation}\label{eq:mnB}
		B_j(x_j) = \sum_{\bx_\alpha \backslash x_j} B_\alpha(\bx_\alpha) \propto n_{j\rightarrow\alpha} (x_j) m_{\alpha\rightarrow j}(x_j). 
	\end{equation}
By \eqref{eq:proofBethe3}, we obtain 
	\[
		n_{j\rightarrow \alpha} (x_j) = 1~\mbox{if}~ N_j=1, j \notin \Gamma.
	\]
It follows that 
	\[
		B_j(x_j) \propto m_{\alpha\rightarrow j} (x_j) ~\mbox{if}~ N_j=1, j \notin \Gamma,
	\]
which is \eqref{eq:solutionBethe2} for leaf nodes. We next show \eqref{eq:solutionBethe2} when $N_j>1$. To this end, we plug \eqref{eq:proofBethe1} and \eqref{eq:proofBethe2} into the constraint \eqref{eq:momt_bethe_c} and arrive at
	\begin{eqnarray*}
		n_{j\rightarrow\gamma} (x_j) m_{\gamma\rightarrow j}(x_j)&\propto&\sum_{\bx_\gamma \backslash x_j} B_\gamma(\bx_\gamma)
		\\&=& B_j(x_j)
		\propto \exp\left(-1 - \frac{1}{N_j - 1} \sum_{\beta \in N(j)} \lambda_{j,\beta}(x_j)  \right)
		\\&\propto& \prod_{\beta\in N(j)} n_{j\rightarrow \beta}(x_j)^{\frac{1}{N_j-1}}.
	\end{eqnarray*}
For fixed $j$, the above holds for all $\gamma\in N(j)$. Multiplying the above equation for all $\gamma\in N(j)\backslash \alpha$ yields
	\[
		\prod_{\gamma \in N(j)\backslash \alpha}\left( n_{j\rightarrow\gamma} (x_j) m_{\gamma\rightarrow j}(x_j)\right) \propto
		\prod_{\beta\in N(j)} n_{j\rightarrow \beta}(x_j),
	\]	
which is \eqref{eq:is_bp_momt2} after canceling out equal terms. Thus, in view of \eqref{eq:mnB},
	\[
		B_j(x_j) \propto n_{j\rightarrow\alpha} (x_j)m_{\alpha\rightarrow j} (x_j) \propto \prod_{\alpha \in N(j)}  m_{\alpha \rightarrow j}(x_j),~\mbox{if}~ N_j>1.
	\]
Finally, \eqref{eq:solutionBethe3} is clearly true due to constraints. This together with \eqref{eq:mnB} leads to \eqref{eq:is_bp_momt3}, which completes the proof.
\end{proof}
The updates in \eqref{eq:is_bp_momt} resemble the standard Belief Propagation algorithm \eqref{eq:BP}. In particular, the updates \eqref{eq:is_bp_momt1} and \eqref{eq:is_bp_momt2} are exactly the same as \eqref{eq:BP}. The update \eqref{eq:is_bp_momt3} is new and is due to the constraints \eqref{eq:momt_bethe_b} on the marginal distributions. Pictorially the message $\bm_{\alpha\rightarrow j}$ sent to a constrained node $j$ from node $\alpha$ bounces back to $\alpha$, in form of $\bn_{j\rightarrow \alpha}$. This is illustrated in Figure \ref{fig:ISBP}.
The update \eqref{eq:is_bp_momt3} in fact corresponds to the scaling step \eqref{eq:sinkhorn_multi} of the Iterative Scaling algorithm (Algorithm \ref{alg:sinkhorn}).
In particular, the multipliers $\{\bu_j~:~j\in \Gamma\}$ in \eqref{eq:u_multi_omt} relate to the messages as $\bu_j = \bn_{j\rightarrow \alpha}$, for $j\in \Gamma$. To see this, we note that the projection $P_j(\bK\odot \bU)$ requires solving a Bayesian inference problem with respect to the modified graphical model
	\begin{equation}\label{eq:middlePGM}
		p(\mx) = \frac{1}{Z}\prod_{\alpha\in F} K_\alpha(\mx_\alpha) \prod_{j\in \Gamma} u_j(x_j).
	\end{equation}
Upon convergence of the Belief Propagation algorithm \eqref{eq:BP}, it holds $P_j(\bK\odot \bU)= \bu_j\bm_{\alpha\rightarrow j}$, where $\alpha$ is the only factor node in $N(j)$ since $j\in \Gamma$ is a leaf node. Thus, the projection step \eqref{eq:sinkhorn_multi} reads
	\[
		\bu_j \odot \boldsymbol\mu_j ./ P_j(\bK \odot \bU) = \boldsymbol\mu_j ./\bm_{\alpha\rightarrow j} = \bn_{j\rightarrow \alpha}.
	\]	
\begin{figure}[tb]	
\centering
\includegraphics[scale=1]{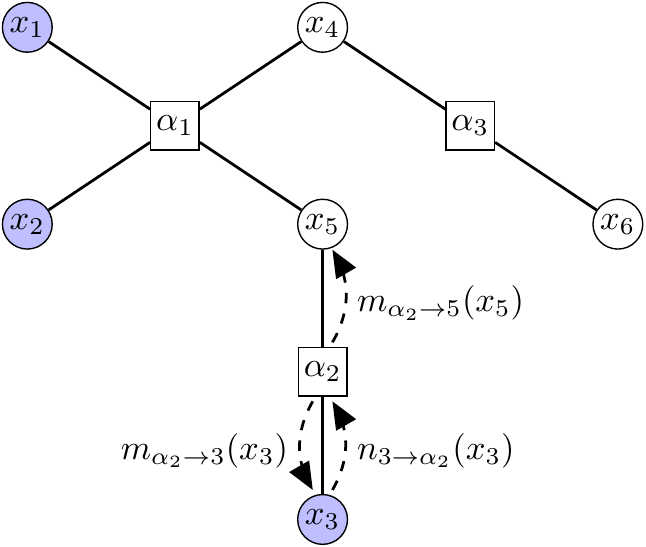}
\caption{Messages in ISBP}
\label{fig:ISBP}
\end{figure}

Therefore, the updates \eqref{eq:is_bp_momt} contain all the components of our ISBP algorithm with \eqref{eq:is_bp_momt1}-\eqref{eq:is_bp_momt2} being the Belief Propagation part and  \eqref{eq:is_bp_momt3} being the Iterative Scaling part. ISBP is a scheduling of these updates in a certain order. As discussed in Section \ref{sec:GraphMOT}, the key idea of ISBP is to implement the projection $P_j(\bK\odot \bU)$ in the iterative scaling step \eqref{eq:sinkhorn_multi} using Belief Propagation. In the contexts of the updates \eqref{eq:is_bp_momt}, it is equivalent to run \eqref{eq:is_bp_momt1}-\eqref{eq:is_bp_momt2} sufficiently many iterations to obtain the precise projection $P_j(\bK\odot \bU)$ and then run  \eqref{eq:is_bp_momt3}, which is essentially \eqref{eq:sinkhorn_multi}. How many iterations of  \eqref{eq:is_bp_momt1}-\eqref{eq:is_bp_momt2} are enough? One option is to run Belief Propagation over the whole graph $G$ with the most recent modified model $\bK\odot\bU$ to compute $P_j(\bK \odot \bU)$ for all $j\in V$. This is clearly sufficient but it is not necessary. Let $j_1, j_2, \ldots$ be a sequence taking values in $\Gamma$ in arbitrary order and suppose the Iterative Scaling algorithm is carried out in this order. Then after the $k$-th step, $\bu_{j_k}$ is updated, and the only projection required in the next step is $P_{j_{k+1}}(\bK \odot \bU)$. It turns out that to evaluate $P_{j_{k+1}}(\bK \odot \bU)$, it suffices to update all the messages on the path from $j_k$ to $j_{k+1}$. Compared to the naive Belief Propagation over the whole graph, this local updating strategy is considerably faster. The steps of the ISBP algorithm are summarized in Algorithm \ref{alg_iterative_scaling}.

\begin{algorithm*}[tb]
   \caption{Iterative Scaling Belief Propagation (ISBP) Algorithm for MOT}
   \label{alg_iterative_scaling}
\begin{algorithmic}
   \STATE Initialize the messages $m_{\alpha\rightarrow j} (x_j)$ and $n_{j\rightarrow \alpha}(x_j)$
   
   \STATE Update $m_{\alpha\rightarrow j} (x_j)$ and $n_{j\rightarrow \alpha}(x_j)$ using \eqref{eq:is_bp_momt1}-\eqref{eq:is_bp_momt2} until convergence
   \WHILE{not converged}
        \STATE Update  $n_{j_k\rightarrow \alpha}(x_{j_k})$ using \eqref{eq:is_bp_momt3}
        \STATE Update all the messages on the path from node $j_k$ to node $j_{k+1}$ according to \eqref{eq:is_bp_momt1} and \eqref{eq:is_bp_momt2}
    \ENDWHILE
\end{algorithmic}
\end{algorithm*}

Upon convergence of Algorithm~\ref{alg_iterative_scaling}, the solution to Problem~\eqref{eq:momt_bethe} can be obtained through \eqref{eq:solutionBethe}. The whole belief tensor $\bB$ can also be obtained through $\bB=\bK\odot \bU$ with $\bU=\bu_1\otimes \bu_2\otimes \ldots \otimes\bu_J$, where $\bu_j=\bn_{j\rightarrow \alpha}$ for $j\in \Gamma$, and $\bu_j=\exp(-\frac{1}{J})\mathbf{1}$ otherwise.


\section{Constrained Norm-product algorithm}
\label{sec:const_norm_prod}

%
One potential drawback of the ISBP algorithm lies in the fact that it is a two-loop algorithm with the outer loop being iterative scaling and inner loop being belief propagation. Such a two-loop structure might slow down the convergence rate, especially when the underlying graph is large. Moreover, the two loops have to coordinate closely to guarantee convergence. Such coordination is even more difficult, or impossible if a distributed implementation is needed. Thus, we seek to develop a single loop algorithm for the entropy regularized MOT problems. A natural question to ask is whether we can borrow ideas from the Bayesian inference literature. 
After all, the Belief Propagation algorithm is not the only algorithm for Bayesian inference. 

The answer is affirmative. In this section, we examine the Norm-product algorithm \cite{HazSha10}, another powerful Bayesian inference method, and extend it to a single loop algorithm for our MOT problems. 
Below we first review Norm-product algorithm for standard Bayesian inference problems in Section \ref{sec:norm_product}. The extensions to entropy regularized MOT, or equivalently, constrained Bayesian inference problems are presented in Section \ref{sec:cnp}.

\subsection{The Norm-product algorithm} \label{sec:norm_product}


Consider the Bayesian inference problem \eqref{eq:free_energy}. The Norm-product algorithm \cite{HazSha10} for Bayesian inference is based on the so called fractional entropy approximation
\begin{equation}\label{eq:fractional_entropy}
    \cH_{\rm frac}(\mb) = \sum_{\alpha\in F} \bar{c}_\alpha \cH(\mb_\alpha) + \sum_{j\in V} \bar{c}_j \cH(\mb_j),
\end{equation}
of the entropy $\cH(\mb)$ in \eqref{eq:var_entropy}. The coefficients $\bar c_\alpha$, for $\alpha \in F$, and $\bar c_j$, for $j \in V$ are defined as
\begin{equation}
    \begin{aligned}
         \bar{c}_\alpha =& \ c_\alpha + \sum_{j \in N(\alpha)} c_{j\alpha},\\
         \bar{c}_j =& \ c_j - \sum_{\alpha \in N(j)} c_{j\alpha},
    \end{aligned}
\end{equation}
for a set of real numbers $c_\alpha$, $c_j$, and $c_{j\alpha}$, for $j\in V$ and $\alpha \in F$, which are known as counting numbers \cite{MesJaiGlo09}.
Clearly, an equivalent formulation of the fractional entropy \eqref{eq:fractional_entropy} is 
\begin{equation}
    \label{eq:fractioanl_entropy_approximation}
    \cH_{\rm frac}(\mb) = \sum_{\alpha\in F} c_\alpha \cH(\mb_\alpha) + \sum_{j\in V} c_j\cH(\mb_j) + \sum_{j\in V} \sum_{\alpha \in N(j)} c_{j\alpha} (\cH(\mb_\alpha) - \cH(\mb_j)).
\end{equation}
The fractional entropy resembles the Bethe entropy \eqref{eq:Bethe_entropy}. In fact, for the choice of counting numbers $c_j = 1 -N_j$, $c_{j\alpha} = 0$, and $c_\alpha = 1$, the fractional entropy  \eqref{eq:fractioanl_entropy_approximation} reduces to the Bethe entropy. Moreover, just like the Bethe entropy, the fractional entropy approximation can be made exact when the underlying graph is a tree (see Section \ref{subsec:counting number}).

With the fractional entropy representation \eqref{eq:fractioanl_entropy_approximation}, the total free energy \eqref{eq:variational_pgm} is modified to the fractional free energy 
\begin{equation} \label{eq:frac_energy}
		\cF_{\rm frac} (\mb) = \cU_{\rm frac} (\mb) - \epsilon\, \cH_{\rm frac}(\mb),
\end{equation}
where 
	\[
		\cU_{\rm frac}(\mb)=\cU(\mb)=- \sum_{\alpha \in F} \sum_{\bx_\alpha} b_\alpha(\bx_\alpha) \ln \psi_\alpha(\bx_\alpha)-\sum_{j\in V}\sum_{x_j} b_j(x_j)\ln \phi_j(x_j)
	\] 
is the average energy defined as in \eqref{eq:var_av_energy}.
Thus, in terms of fractional free energy, the Bayesian inference problem \eqref{eq:free_energy} reads
\begin{subequations}\label{eq: norm-product optimization goal}
\begin{eqnarray}\label{eq: norm-product optimization goal_o}
     \min_{\mb} &&\cF_{\rm frac}(\mb)
    \\ 
   \text{subject to} && \sum_{\bx_\alpha \backslash x_j} b_\alpha(\bx_\alpha) = b_j(x_j), \quad  \forall j \in V, \alpha \in N(j), \label{eq: norm-product optimization goal_c1}
   \\&& \sum_{\bx_\alpha} b_\alpha(\bx_\alpha)= 1, \quad  \forall \alpha \in F. \label{eq: norm-product optimization goal_c2}
\end{eqnarray}
\end{subequations}
The constraints \eqref{eq: norm-product optimization goal_c1}-\eqref{eq: norm-product optimization goal_c2} are to ensure that $\mb_\alpha,\mb_j$ are indeed marginal distributions of some certain joint distribution. For a given graphical model, there are infinitely many different fractional free energy approximations determined by the counting numbers $c_j,\,c_\alpha, c_{j\alpha}$; some of them are convex and some of them are not. A sufficient condition for the convexity of the fractional free energy is as follows.

\begin{lemma}[\cite{HazSha10}]\label{lem:convex}
When the counting numbers satisfy $c_j\geq 0,\,c_{j\alpha}\geq 0$, and $c_\alpha>0$, for $j\in V$ and $\alpha \in F$, then the fractional free energy $\cF_{\rm frac}$ is strictly convex over the set defined by the constraints \eqref{eq: norm-product optimization goal_c1}-\eqref{eq: norm-product optimization goal_c2}. 
\end{lemma}

Denote the two sets corresponding to the constraints \eqref{eq: norm-product optimization goal_c1} and \eqref{eq: norm-product optimization goal_c2} by 
\begin{equation} \label{eq:set_consistency}
    \cM = \Big\{ \mb : \sum_{\bx_\alpha \backslash x_j} b_\alpha(\bx_\alpha) = b_j(x_j), \quad  \forall j \in V, \alpha \in N(j) \Big\}
   \end{equation} 
    and 
   \begin{equation} \label{eq:set_normalization}
    \cP = \Big\{ \mb: \sum_{\bx_\alpha} b_\alpha(\bx_\alpha)= 1, \quad  \forall \alpha \in F \Big\},
\end{equation}
respectively, and define 
\begin{equation}
    \hat{f}(\mb) =     - \sum_{\alpha\in F}\sum_{\mx_{\alpha}} b_\alpha   (\mx_\alpha) \ln \psi_\alpha(\mx_\alpha)  -
        \sum_{\alpha\in F} \epsilon\, c_\alpha \cH(\mb_\alpha)
\end{equation}
and
\begin{equation} \label{eq:h_j}
  \hat{h}_{j}(\mb) = - \sum_{x_j} b_j(x_j)\ln \phi_j(x_j)  - \epsilon\, c_j \cH(\mb_j) -
    \sum_{\alpha \in N(j)} \epsilon\, c_{j \alpha} (\cH(\mb_\alpha)-\cH(\mb_j)).
\end{equation}
Further denote $f(\mb) = \hat f(\mb) + \delta_\cP(\mb)$ and $h_j(\mb) = \hat h_j(\mb) + \delta_\cM(\mb)$ where $\delta$ is the indicator function. Then the Bayesian inference problem \eqref{eq: norm-product optimization goal} can be reformulated as 
\begin{equation}\label{eq:primaldualformulation}
\min_{\mb} f( \mb) + \sum_{j=1}^J h_j (\mb).
\end{equation}

In cases where $c_j\geq 0,\,c_{j\alpha}\geq 0$, and $c_\alpha>0$, for $j\in V$ and $\alpha \in F$, by Lemma \eqref{lem:convex}, $\hat f$ is strictly convex and $\hat h_j$ is convex for each $1\le j\le J$ for $\mb\in \cM\cap \cP$. 
The Norm-product algorithm relies on the reformulation \eqref{eq:primaldualformulation}.
In particular, it leverages a powerful primal-dual ascent algorithm (stated below in Lemma \ref{lem:Primal-Dual Ascent}) that is well studied in the convex optimization community to solve problem with the special structure of \eqref{eq:primaldualformulation}. The primal-dual ascent in Lemma \ref{lem:Primal-Dual Ascent} is derived from a more general algorithm known as dual block ascent \cite{LuoTse93} and thus inherit the nice convergence property of the latter. We refer the read to \cite{HazSha10} for more details on these algorithms.

\begin{lemma}[\cite{HazSha10}] \label{lem:Primal-Dual Ascent}
Consider the convex optimization problem $\min f + \sum_{j=1}^J h_j$ with $f(\mb) = \hat{f}(\mb) + \delta_{\mathcal{B}}(\mb)$ where $\mathcal{B} = \left\{ \mb: A\mb = \mathbf{c}\right\}$. The {\bf primal-dual ascent} algorithm initializes $\boldsymbol\lambda_1 = 0,\dots,\boldsymbol\lambda_J=0$ and repeatedly iterates the following steps for $j=1,\dots,J$ until convergence: 
\begin{subequations}\label{eq:pda}
	\begin{eqnarray}\label{eq:pdaa}
	\boldsymbol\nu_j & \leftarrow & \sum_{i \neq j} \boldsymbol\lambda_i \\\label{eq:pdab}
	\mb^* & \leftarrow & \argmin_{\mb \in dom(f)\cap dom(h_j)} \left\{ f(\mb)+h_j(\mb)+\mb^T\boldsymbol\nu_j \right\} \\\label{eq:pdac}
	\boldsymbol\lambda_j & \leftarrow & -\boldsymbol\nu_j - \nabla \hat f(\mb^*) + A^T \boldsymbol\sigma \text{ where $\boldsymbol\sigma$ is an arbitrary vector.}
	\end{eqnarray}
\end{subequations}
Suppose $\hat f$ is strictly convex and smooth, and $h_j,\,j=1,\ldots, J$ are convex and continuous over their domains, then $\mb^*$ in the above iteration converges to the unique global minimizer of $f+ \sum_{j=1}^J h_j$.
\end{lemma}

The Norm-product algorithm \cite{HazSha10} (Algorithm \ref{alg:Norm-product Belief Propagation}) is a direct application of the primal-dual ascent algorithm to the formulation \eqref{eq:primaldualformulation} of the Bayesian inference problem \eqref{eq: norm-product optimization goal}. It can be seen as a message-passing type algorithm for problem \eqref{eq: norm-product optimization goal}, where the dual variables $\boldsymbol\lambda_j$ in the primal-dual ascent algorithm work as ``messages'' between neighboring nodes.
To see this, note that $h_j(\mb)$ depends only on $\mb_\alpha$, where $\alpha \in N(j)$, and thus the corresponding dual variable $\boldsymbol\lambda_j$ depends only on $\bx_\alpha$, where $\alpha \in N(j)$. This sparsity is encoded by the representation $\boldsymbol\lambda_j = \{\lambda_{j,\alpha}(\bx_\alpha)\}$.
The relation between the dual variables $\boldsymbol\lambda_j$ in Lemma~\ref{lem:Primal-Dual Ascent} and messages in Algorithm~\ref{alg:Norm-product Belief Propagation} is given by $n_{j\rightarrow \alpha}(\bx_\alpha)=\exp (-\lambda_{j,\alpha}(\bx_\alpha))$.


\begin{algorithm*}[tb]
  \caption{The Norm-product Algorithm}
  \label{alg:Norm-product Belief Propagation}
\begin{algorithmic}
  \STATE Initialize $n_{j \rightarrow \alpha} (\mx_\alpha) = 1$ for all $j = 1, \cdots, J$, $\alpha \in N(j)$ and $\mx_\alpha$
  \WHILE{not converged}
        \FOR{$j = 1, 2, \ldots, J$}
            \STATE 
        \begin{eqnarray*}
             m_{\alpha \rightarrow j}(x_j) &=& 
            \left( 
                \sum_{\mx_\alpha \setminus x_j}
                    \left( 
                        \psi_\alpha(\mx_\alpha)\prod_{i\in N(\alpha) \setminus j} n_{i  \rightarrow \alpha}(\mx_\alpha)
                    \right)^{1/\epsilon \hat{c}_{j\alpha}}
            \right)^{\epsilon\hat{c}_{j\alpha}}, ~ \forall \alpha \in N(j), \forall x_j
            \\
             n_{j \rightarrow \alpha}(\mx_{\alpha}) &\propto& 
            \left(
                \frac{\phi_j^{1/\hat{c}_j}(x_j)
                        \prod_{\beta\in N(j)} m_{\beta\rightarrow j}^{1/\hat{c}_j}(x_j)}
                {m_{\alpha \rightarrow j}^{1/\hat{c}_{j\alpha}}(x_j)}
            \right)^{c_\alpha}
            \left(
                \psi_\alpha(\mx_\alpha)
                    \prod_{i\in N(\alpha)\setminus j} n_{i\rightarrow \alpha}(\mx_\alpha)
            \right)^{-c_{j\alpha}/\hat{c}_{j\alpha}}, ~ \forall \alpha \in N(j),\forall \mx_\alpha
        \end{eqnarray*}
        \ENDFOR
  \ENDWHILE
\end{algorithmic}
\end{algorithm*}

For more details on the derivation of the Norm-product algorithm as a primal-dual ascent method, see \cite{HazSha10}. Moreover, our development of the constrained Norm-product algorithm (Algorithm \ref{alg:constrained norm-product}) is similar to this, and is  provided in the appendix. Upon convergence of Algorithm~\ref{alg:Norm-product Belief Propagation}, the solution to \eqref{eq: norm-product optimization goal} has the form
\begin{subequations}\label{eq:convergemarginals}
\begin{eqnarray}
        b_j(x_j) &\propto& \left( \phi_j(x_j) \prod_{\alpha \in N(j)}m_{\alpha \rightarrow j}(x_j) \right)^{1/\epsilon \hat{c}_j},
        \\
        b_\alpha(\bx_\alpha) &\propto& \left(\psi_\alpha(\mx_\alpha)\prod_{j\in N(\alpha)} n_{j \rightarrow \alpha}(\mx_\alpha)\right)^{1/\epsilon c_{\alpha}},
\end{eqnarray}
\end{subequations}
where again $\propto$ indicates that a normalization step might be needed. Finally, note that for the special choice of counting numbers $c_j = 1 -N_j$, $c_{j\alpha} = 0$, and $c_\alpha = 1$, the Norm-product algorithm reduces to the Belief Propagation algorithm \eqref{eq:BP}. However, note that this choice of counting number does not satisfy the conditions in Lemma \ref{lem:convex} for convergence of the algorithm, although it is well known that the Belief Propagation algorithm is guaranteed to converge for trees \cite{JorGhaJaaSau99}. Therefore, even though formally the Norm-product algorithm can be viewed as a unifying framework for many message-passing algorithms, its convergence proof is restricted in some sense due to the strong requirement on the counting numbers.  


\subsection{Constrained Norm-product algorithm}
\label{sec:cnp}

In this section, we develop a Norm-product type algorithm for the entropy regularized MOT problem \eqref{eq:omt_multi_regularized}, or equivalently the constrained Bayesian inference problem \eqref{eq:momt_free_energy}.
Consider a modification of problem \eqref{eq: norm-product optimization goal} with constrained marginal distributions, which reads
\begin{subequations}\label{eq: constrained-norm-product-optimization} 
\begin{eqnarray}
     \min_{\mb} &&\cF_{\rm frac} (\mb) \label{eq: constrained-norm-product-optimization_o} \\
    \text{subject to}      && b_j(x_j) = \mu_j(x_j),\quad \forall j\in \Gamma, \label{eq: constrained-norm-product-optimization_c1} \\
    && \sum_{\mx_\alpha \setminus x_j} b_\alpha(\mx_\alpha) = b_j(x_j)\quad \forall j \in V,\alpha \in N(j) \label{eq: constrained-norm-product-optimization_c2} \\
            &&  \sum_{\mx_\alpha} b_\alpha (\mx_\alpha) = 1 \quad \forall \alpha \in F \label{eq: constrained-norm-product-optimization_c3} .
\end{eqnarray}
\end{subequations}

Problem \eqref{eq: constrained-norm-product-optimization} can be seen in the light of the entropy regularized MOT problem formulated as a free energy minimization problem \eqref{eq:momt_free_energy}. In particular, if the free energy $\cF$ is approximated by the fractional free energy $\cF_{\rm frac}$, then \eqref{eq:momt_free_energy} becomes \eqref{eq: constrained-norm-product-optimization}. Recall that in the MOT problem the factor and node potentials are $\boldsymbol\psi_\alpha=\exp (-\bC_\alpha)$ and $\boldsymbol\phi_j\equiv 1$ (cf. \eqref{eq:psiphi}).
However, the constrained Norm-product algorithm, which we develop in the following solves the Bayesian inference problem \eqref{eq: constrained-norm-product-optimization} for any potentials $\boldsymbol\psi_\alpha$ and $\boldsymbol\phi_j$.

Note that compared to \eqref{eq: norm-product optimization goal}, the modified problem \eqref{eq: constrained-norm-product-optimization} is only augmented by one linear constraint \eqref{eq: constrained-norm-product-optimization_c1}.
Thus, problem \eqref{eq: constrained-norm-product-optimization} can be formulated as in \eqref{eq:primaldualformulation} by changing the set $\cM$ in $h_j(\mb) = \hat h_j(\mb) + \delta_\cM(\mb)$ to 
\begin{equation}
    \cM = \Big\{ \mb : \sum_{\bx_\alpha \backslash x_j} b_\alpha(\bx_\alpha) = b_j(x_j), \forall j\in V, \alpha\in N(j),\, b_j(x_j) = \mu_j(x_j), \, \forall j \in \Gamma \Big\},
\end{equation}
instead of \eqref{eq:set_consistency}, and defining all other components of \eqref{eq:primaldualformulation} as in \eqref{eq:set_normalization}-\eqref{eq:h_j}.
The primal-dual ascent algorithm in Lemma~\ref{lem:Primal-Dual Ascent} can then be applied to \eqref{eq: constrained-norm-product-optimization}. The resulting Constrained Norm-product (CNP) algorithm is presented in Algorithm \ref{alg:constrained norm-product}. For a detailed derivation of the method see Appendix~\ref{sec:proofCNP}.

\begin{algorithm*}[tb]
  \caption{Constrained Norm-product (CNP) algorithm}
  \label{alg:constrained norm-product}
\begin{algorithmic}
  \STATE Set $n_{j \rightarrow \alpha} (\mx_\alpha) = 1$ for all $j = 1, \cdots,J$, $\alpha \in N(j)$ and $\mx_\alpha$
  \WHILE{not converged}
        \FOR{$j = 1, 2, \ldots, J$}
        
        \STATE 
        \begin{equation*}
            m_{\alpha \rightarrow j}(x_j) = 
            \left( 
                \sum_{\mx_\alpha \setminus x_j}
                    \left( 
                        \psi_\alpha(\mx_\alpha)\prod_{i\in N(\alpha) \setminus j} n_{i  \rightarrow \alpha}(\mx_\alpha)
                    \right)^{1/\epsilon \hat{c}_{j\alpha}}
            \right)^{\epsilon\hat{c}_{j\alpha}}, ~ \forall \alpha \in N(j), \forall x_j 
        \end{equation*}
        
        \STATE 
        \IF{$j \notin \Gamma$}
        \STATE 
        \begin{equation*}
            n_{j \rightarrow \alpha}(\mx_{\alpha}) \propto 
            \left(
                \frac{\phi_j^{1/\hat{c}_j}(x_j)
                        \prod_{\beta\in N(j)} m_{\beta\rightarrow j}^{1/\hat{c}_j}(x_j)}
                {m_{\alpha \rightarrow j}^{1/\hat{c}_{j\alpha}}(x_j)}
            \right)^{c_\alpha}
            \left(
                \psi_\alpha(\mx_\alpha)
                    \prod_{i\in N(\alpha)\setminus j} n_{i\rightarrow \alpha}(\mx_\alpha)
            \right)^{-c_{j\alpha}/\hat{c}_{j\alpha}}, ~ \forall \alpha \in N(j),\forall \mx_\alpha
        \end{equation*}
        \ELSIF{$j \in \Gamma$}
        \STATE
        \begin{equation*}
            n_{j \rightarrow \alpha}(\mx_{\alpha}) \propto 
            \left(
                \frac{\mu_j (x_j)}
                {m_{\alpha \rightarrow j}^{1/\epsilon \hat{c}_{j\alpha}}(x_j)}
            \right)^{\epsilon c_\alpha}
            \left(
                \psi_\alpha(\mx_\alpha)
                    \prod_{i\in N(\alpha)\setminus j} n_{i\rightarrow \alpha}(\mx_\alpha)
            \right)^{-c_{j\alpha}/\hat{c}_{j\alpha}}   , ~ \forall \alpha \in N(j), \forall \mx_\alpha         
        \end{equation*}
        \ENDIF
        \ENDFOR
  \ENDWHILE

\end{algorithmic}
\end{algorithm*}
Upon convergence of Algorithm~\ref{alg:constrained norm-product}, the solution to \eqref{eq: constrained-norm-product-optimization} is of the form \eqref{eq:convergemarginals}, as in the standard Norm-product algorithm. Moreover, the optimal marginal calculated through \eqref{eq:convergemarginals} satisfies the constraint $\mb_j=\boldsymbol\mu_j$ for all $j\in \Gamma$. 
Algorithm~\ref{alg:constrained norm-product} is presented for general constrained Bayesian inference problems \eqref{eq: constrained-norm-product-optimization}. Recall that the entropy regularized MOT problem \eqref{eq:momt_free_energy} is recovered as the special case, where the potentials are given by
$\boldsymbol\psi_\alpha=\exp (-\bC_\alpha)$ and $\boldsymbol\phi_j\equiv 1$, as in \eqref{eq:psiphi}. 

Compared with the standard Norm-product algorithm, the messages from variable nodes with marginal constraint to the neighboring factor nodes, i.e., $\bn_{j \rightarrow \alpha}$, for $j\in \Gamma,\,\alpha\in N(j)$, depend not only on the incoming messages to $j$ and $\alpha$, but also the given marginal $\boldsymbol\mu_j$. Moreover, in the case when the marginal constraint \eqref{eq: constrained-norm-product-optimization_c1} is absent, namely, $\Gamma=\emptyset$, Algorithm~\ref{alg:constrained norm-product} reduces to the standard Norm-product belief algorithm~\ref{alg:Norm-product Belief Propagation}.

\begin{remark}
The message updates in Algorithm \ref{alg:constrained norm-product} can be problematic when the denominators become zero. This scenario can occur when either the factor or node potentials $\psi_\alpha(\mx_\alpha)$ or $\phi_j(x_j)$ contain zero elements. 
Note that zero entries in the potential let the average energy \eqref{eq:var_av_energy} be unbounded if $b_j(x_j)$ or $b_\alpha(\mx_\alpha)$ are nonzero on the corresponding entries.
In implementations, this can be avoided by ignoring the updates involving zero denominators. See \cite[Appendix F]{HazSha10} for a more detailed discussions of this issue. 
\end{remark}

\subsection{Relations to Iterative Scaling Belief Propagation algorithm} 

Compared to the ISBP algorithm, the CNP algorithm is a single loop algorithm. Each iteration of Algorithm \ref{alg:constrained norm-product} requires visiting every variable node only once. In contrast, since Algorithm \ref{alg_iterative_scaling} has a double-loop structure and each inner-loop iteration requires updating throughout an entire path between two leaf nodes, the messages associated with most variable nodes will be updated multiple times in one iteration of the algorithm. Thus, the iteration complexity of the ISBP algorithm is higher than that of the CNP algorithm. This difference becomes more significant as the diameter/size of the underlying graph increases; for larger graphs, the inner-loop iteration of ISBP algorithm takes more updates. Apart from the iteration complexity, another potential advantage of the CNP algorithm is that its single loop structure allows for more flexible scheduling of the message passing/updating. In particular, it does not require any communication between inner and outer loop updates. Thus, it is easier to parallelize the Constrained Norm-product algorithm or develop a distributed version of it.  


Recall from Section~\ref{sec:norm_product} that the standard Norm-product method with counting numbers chosen as $c_\alpha = 1$, $c_j = 1- N_j$ and $c_{j\alpha} = 0$ reduces to the standard Belief propagation method as given in \eqref{eq:BP}.
It turns out that similar results can be established to relate the Iterative Scaling Belief Propagation algorithm and the Constrained Norm-product algorithm. In particular, with this set of counting numbers, the constrained Norm-Product algorithm reads
\begin{subequations} \label{eq:cbp}
\begin{eqnarray}
    \label{eq:factor2variable_is}
     m_{\alpha \rightarrow j}(x_j) &=& 
        \sum_{\mx_\alpha \setminus x_j}
            \left(
                \psi_\alpha(\mx_\alpha)\prod_{i\in N(\alpha) \setminus j} n_{i  \rightarrow \alpha}(x_j)
            \right), \quad \forall \alpha \in N(j), \forall x_j
\\
    \label{eq:variable2factor_free_is}
    n_{j \rightarrow \alpha}(x_j) &\propto& 
    \left(
        \phi_j(x_j)
                \prod_{\beta\in N(i)\setminus \alpha} m_{\beta\rightarrow j}(x_j)
    \right), \quad \forall j \notin \Gamma,\forall \alpha \in N(j),\forall x_j
\\
    \label{eq:variable2factor_constrained_is}
     n_{j \rightarrow \alpha}(x_j) &\propto& 
   \mu_j (x_j)
        (m_{\alpha \rightarrow j}(x_j))^{-1}, \quad \forall j \in \Gamma,\forall \alpha \in N(j), \forall x_j.
\end{eqnarray}
\end{subequations}
Note that in general the messages $\bn_{j\rightarrow \alpha}$ in the Constrained Norm-product algorithm depend on $\mx_\alpha$, but for this special choice of counting numbers, they depend only on $x_j$. 
The messages \eqref{eq:cbp} are exactly the same as the messages \eqref{eq:is_bp_momt} in the ISBP algorithm. 
If the messages in \eqref{eq:cbp} are scheduled in a specific way, then this becomes the Iterative Scaling Belief Propagation Algorithm~\ref{alg_iterative_scaling}. 
In particular, this is achieved by cycling through the nodes in $ \Gamma$, where for two successive nodes $j_1,j_2\in\Gamma$, one schedules the messages \eqref{eq:factor2variable_is} and \eqref{eq:variable2factor_free_is} on the path from $j_1$ to $ j_2 $, and finally the message $\bn_{j_2 \rightarrow \alpha}$ as in \eqref{eq:variable2factor_constrained_is}.
In this light, Algorithm~\ref{alg_iterative_scaling} may not only be understood as Iterative scaling Belief propagation, but also as constrained Belief propagation, i.e., an extension of the standard Belief propagation method, where the marginals on some nodes are fixed.

What if we update the messages \eqref{eq:cbp} following the scheduling of Algorithm \ref{alg:constrained norm-product}? 
In fact, this is a single-loop version of the ISBP algorithm, and we have empirically observed good convergence properties of it. 
However, the choice of counting numbers $c_\alpha = 1$, $c_j = 1- N_j$ and $c_{j\alpha} = 0$ does not yield a strictly convex objective function decomposition in the associated fractional variational inference problem \eqref{eq: constrained-norm-product-optimization} as discussed in Lemma \ref{lem:convex}.
Thus, the convergence proof of Algorithm \ref{alg:constrained norm-product} does not apply to this setting, and a global convergence proof remains an open problem.

\subsection{Counting numbers of fractional entropy}
\label{subsec:counting number} 

One way to guarantee the convergence of the Constrained Norm-product algorithm is to choose the counting numbers for the fractional entropy $\cH_{\rm frac}$ such that they satisfy the convexity conditions in Lemma~\ref{lem:convex}. 
Thus, a crucial question 
is whether, given a graphical model, such a choice of counting numbers exists, and how to find them.
This question has been discussed in \cite[Appendix E]{HazSha10} where several optimization based methods have been proposed.
In this section, we present a structured method to construct a feasible set of counting numbers that satisfy the assumptions in Lemma \ref{lem:convex}, viz., $c_j\geq 0,\,c_{j\alpha}\geq 0, c_\alpha>0$, for factor graphs, which are trees. In particular, we provide a closed form expression for the choice of counting numbers, which makes parameter tuning for the Constrained-norm product algorithm simple and intuitive.

The fractional entropy decomposition requires the fractional entropy $\cH_{\rm frac}(\mb)$ to be equal to the entropy $\cH(\mb)$, that is 
	\[
\cH(\mb) = \cH_{\rm frac} (\mb) = \sum_{\alpha\in F} c_\alpha \cH(\mb_\alpha) + \sum_{j\in V} c_j\cH(\mb_j) + \sum_{j\in V,\alpha \in N(j)} c_{j\alpha} (\cH(\mb_\alpha) - \cH(\mb_j)).
	\]
On the other hand, for a factor tree, the entropy equals the Bethe entropy, namely,
	\[
		\cH(\mb)=\cH_{\rm Bethe}(\mb)=\sum_{\alpha\in F}\cH(\mb_\alpha)-\sum_{j\in V}(N_j-1)\cH(\mb_j).
	\]
It follows that
\[
    \sum_{j\in V}(1-N_j)\cH(\mb_j) + \sum_{\alpha\in F}\cH(\mb_\alpha) = 
    \sum_{j\in V} (c_j - \sum_{\alpha \in N(j)} c_{i\alpha} )\cH(\mb_j) + 
    \sum_{\alpha\in F} (c_\alpha + \sum_{i \in N(j)} c_{j\alpha})\cH(\mb_\alpha).
\]
Hence, by identifying the coefficients, we see that finding a set of feasible convex counting numbers is achieved by finding $c_\alpha >0, c_j \geq 0, c_{j\alpha} \geq 0$ that satisfy the following equations
\begin{subequations}\label{eq:counting number detail}
\begin{eqnarray}
    &&c_j - \sum_{\alpha \in N(j)} c_{j\alpha} = 1 -N_j, \label{eq:varible_node coef}\\
    && c_\alpha + \sum_{j \in N(\alpha)} c_{j\alpha} = 1. \label{eq:factor_node coef} 
\end{eqnarray}
\end{subequations}

A direct consequence of \eqref{eq:counting number detail} is 
	\begin{equation}\label{eq:sum1}
        		\sum_{j \in V} c_j + \sum_{\alpha \in F} c_{\alpha} = 1.
    	\end{equation}
To see this, sum up \eqref{eq:varible_node coef} and \eqref{eq:factor_node coef} over all variable nodes and factor nodes. The left hand side becomes $\sum_{j \in V} c_j + \sum_{\alpha \in F} c_{\alpha}$ as all the terms $c_{j\alpha}$ get canceled. The right hand side becomes
	\[
		\sum_{j\in V} (1-N_j) +\sum_{\alpha \in F} 1 = -\sum_{j\in V} N_j +\sum_{j \in V} 1+\sum_{\alpha \in F} 1= -|E| + |V|=1.
	\]
The last equality is due to the fact that the factor graph is a tree (acyclic).  


The property \eqref{eq:sum1} can be generalized to subgraphs of $G$. Let $(j_*,\alpha_*)$ be any edge of $G$. If we cut this edge, then the tree $G$ is split into two trees $G_1$ and $G_2$, where $G_1$ contains the variable node $j_*$ and $G_2$ contains the factor node $\alpha_*$. This is illustrated in Figure \ref{fig:subgraph}. 
\begin{figure}[tb]
\centering
\includegraphics[width=0.4\textwidth]{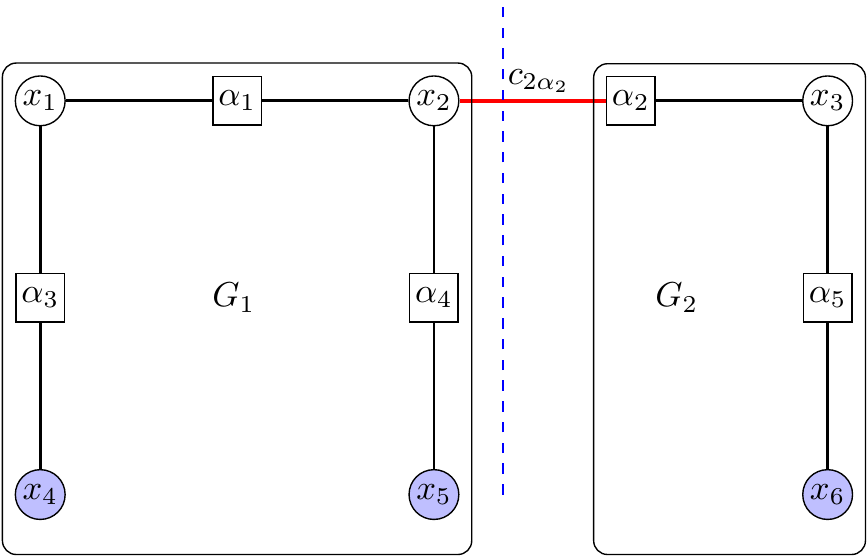}
\caption{Subgraphs $G_1$ and $G_2$ of $G$ by cutting edge $(2,\alpha_2)$.}
\label{fig:subgraph}
\end{figure}
Let $G_1=(V_1, F_1, E_1)$, then
	\begin{equation}\label{eq:edgeweight}
		c_{j_*\alpha_*} = \sum_{j \in V_1} c_j + \sum_{\alpha \in F_1} c_{\alpha}.
	\end{equation}
This relation \eqref{eq:edgeweight} can be established similarly to \eqref{eq:sum1}. It determines values for $c_{j\alpha}$ given $c_j, j\in V$ and $c_\alpha, \alpha\in F$.
Moreover, it guarantees that $c_{j\alpha}$ is non-negative, as long as $c_j$ and $c_\alpha$ are non-negative. 
Hence, based on \eqref{eq:sum1} and \eqref{eq:edgeweight}, we obtain a remarkably simple strategy to get a set of convex counting numbers $c_\alpha >0, c_j \geq 0, c_{j\alpha} \geq 0$.
\begin{prop}\label{prop:counting}
The following procedures lead to a feasible set of counting numbers $c_\alpha >0, c_j \geq 0, c_{j\alpha} \geq 0$ that solves \eqref{eq:counting number detail}:
\begin{itemize}
\item[i)] Choose $ c_\alpha>0, c_j\ge 0$ for $j\in V, \alpha\in F$ such that \eqref{eq:sum1} is satisfied;
\item[ii)] Iterate over each edge in the graph, split the graph along the edge and calculate the corresponding $c_{j\alpha}$ through \eqref{eq:edgeweight}. 
\end{itemize}
\end{prop}
\begin{proof}
Obviously, by construction, $c_\alpha > 0, c_j \geq 0, c_{j\alpha} \geq 0$ for all $j\in V, \alpha\in F$. We next show that they satisfy \eqref{eq:counting number detail}. To this end, denote the two subgraphs $G_1$ and $G_2$ discussed earlier by cutting edge $(j, \alpha)$ by $G_{1,j\alpha}=(V_{1,j\alpha},F_{1,j\alpha},E_{1,j\alpha})$ and $G_{2,j\alpha}=(V_{2,j\alpha},F_{2,j\alpha},E_{2,j\alpha})$, respectively. It follows, for any $j\in V$, 
\begin{equation*}
    \begin{aligned}
        c_j - \sum_{\alpha \in N(j)} c_{j\alpha} &= c_j - \sum_{\alpha \in N(j)}
        (
            \sum_{i\in V_{1,j\alpha}} c_{i} + 
            \sum_{\beta \in F_{1,j\alpha}}c_{\beta}
        ) \\
        & = c_j - \sum_{\alpha \in N(j)}
        ( 1 -
            \sum_{i\in V_{2,j\alpha}} c_{i} - 
            \sum_{\beta \in F_{2,j\alpha}}c_{\beta}
        ) \\
        & = \sum_{i\in V} c_i + \sum_{\beta \in F} c_{\beta} - N_j \times 1  = 1 - N_j,
    \end{aligned}
\end{equation*}
where the second last equality is due to the fact that $V = \{c_j\} \cup(\cup_{\alpha \in N(j)} V_{2,j\alpha})$ and $F = \cup_{\alpha \in N(j)} F_{2,j\alpha}$. This establishes \eqref{eq:varible_node coef}. The proof of \eqref{eq:factor_node coef} is similar.
\end{proof}

Proposition \ref{prop:counting} makes constructing a feasible set of counting numbers that induces convex fractional free energy (see Lemma \ref{lem:convex}) as easy as finding $c_\alpha > 0, c_j \geq 0, j\in V, \alpha\in F$ that satisfy \eqref{eq:sum1}. One choice we found effective is 
	\[
		c_j = c_\alpha = \frac{1}{|V|+|F|}
	\]
for all $j\in V, \alpha\in F$. For the specific example in Figure \ref{fig:subgraph}, this choice leads to $c_j = c_\alpha = \frac{1}{11}$. The value of $c_{2\alpha_2}$ is $\frac{7}{11}$ by \eqref{eq:edgeweight}.

\section{Numerical examples}\label{sec:example}
In this section we present two sets of numerical experiments based on our framework. The first set of experiments is to validate the correctness of Constrained Norm-Product and Iterative Scaling Belief Propagation algorithm, and compare their performance. The second experiment is to illustrate potential applications of our framework in nonlinear filtering for collective dynamics.

\subsection{Performance evaluation}

We implement three algorithms: Constrained Norm-Product (CNP) (Algorithm \ref{alg:constrained norm-product}), Iterative Scaling Belief Propagation (ISBP) (Algorithm \ref{alg_iterative_scaling}) and Vanilla Iterative Scaling (Vanilla IS) (Algorithm \ref{alg:sinkhorn}) on four different type of graphs: line graphs, hidden Markov models (HMMs), and two star shape graphs (see Figure \ref{fig:graph}). 
\begin{figure}[tb]
    \begin{subfigure}{.25\textwidth}
    \centering
    \vspace{1.5cm}
    \includegraphics[scale=0.45]{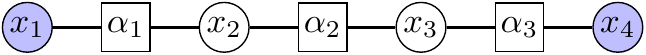}
    \vspace{1.5cm}
    \caption{Line}
    \label{fig:grapha}
    \end{subfigure}%
    \centering
    \begin{subfigure}{.25\textwidth}
    \centering
    \vspace{0.6cm}
    \includegraphics[scale=0.45]{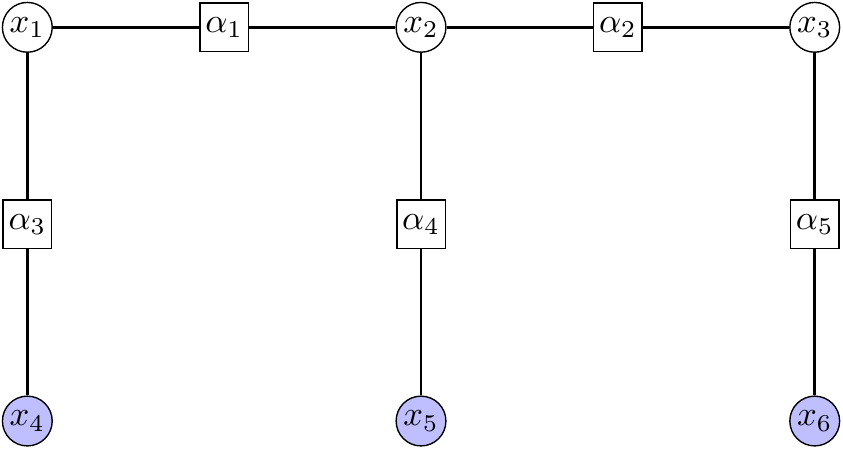}
    \vspace{0.6cm}
    \caption{HMM}
    \label{fig:graphb}
    \end{subfigure}%
    \begin{subfigure}{.25\textwidth}
    \centering
    \includegraphics[scale=0.45]{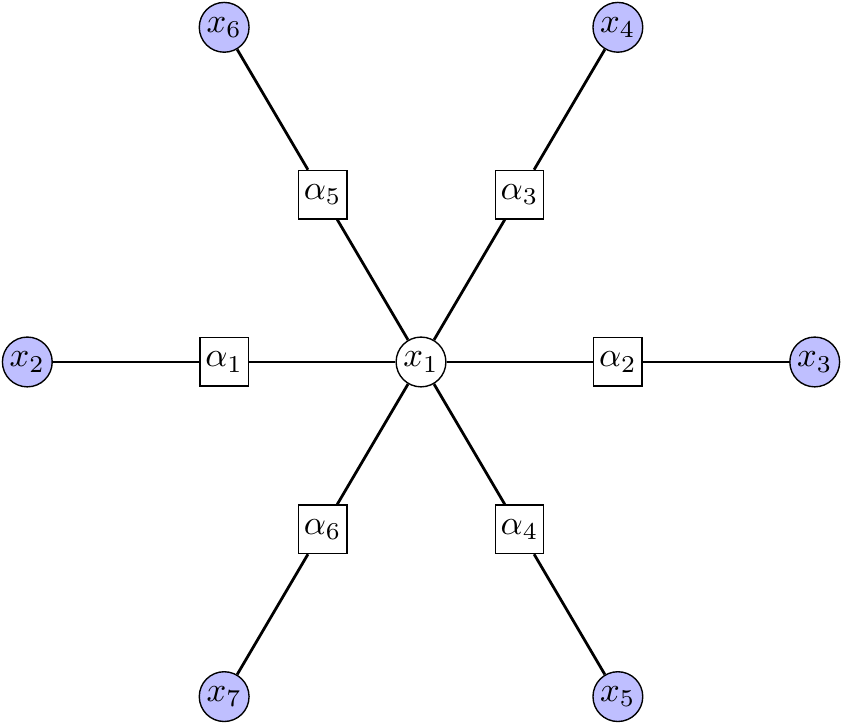}
    \caption{Star}
    \label{fig:graphc}
    \end{subfigure}%
    \begin{subfigure}{.25\textwidth}
    \centering
    \includegraphics[scale=0.45]{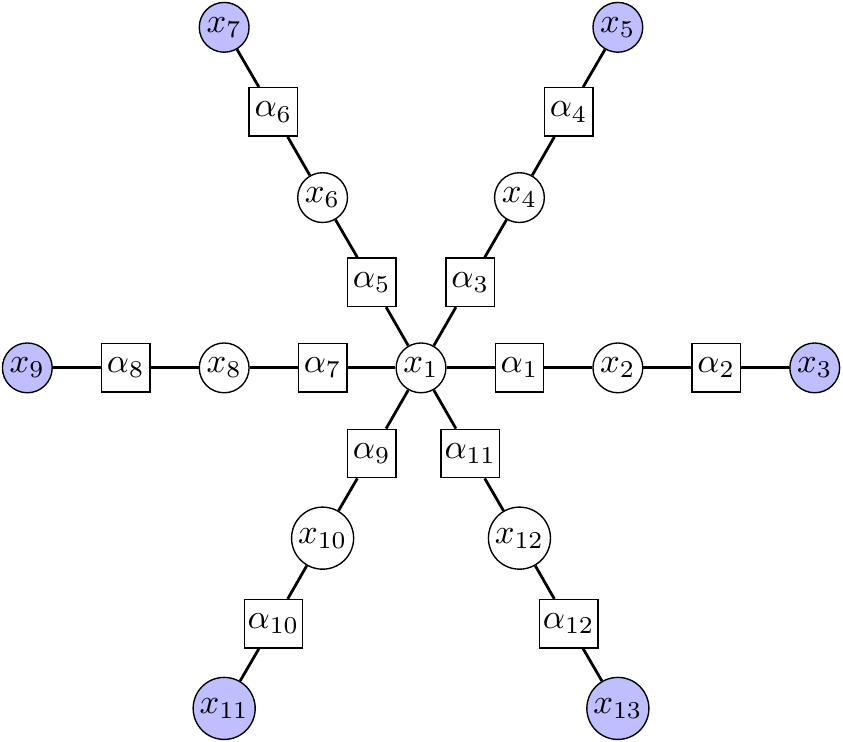}
    \caption{Long Star}
    \label{fig:graphd}
    \end{subfigure}%
    \caption{Testing graphical models}
    \label{fig:graph}
\end{figure}
For the line graph (Figure \ref{fig:grapha}), the constraints of marginal distribution are on the head and tail nodes. This corresponds to a standard OT problem with two marginals. HMMs (Figure \ref{fig:graphb}) are widely used in many real applications. In the standard HMM framework, the measurements are deterministic values, which can be equivalently viewed as Dirac distributions on the observation/measurement nodes. In our MOT framework, these observation nodes are associated with marginal distribution constraints, which can be viewed as a relaxation of the standard HMM where the deterministic measurement are replaced with ``soft'' stochastic measurements. 
The star graph (Figure \ref{fig:graphc}) structure has marginal constraints on all the leaf nodes. This corresponds to the Barycenter problem over the Wasserstein space \cite{AguCar11}, which has found applications in information fusion \cite{ElvHaaJakKar20}.

We test the algorithms with several different configurations. In particular, we vary the number of discrete states at each variable node $d_1=d_2=\cdots=d_J=d$ as well as the number of nodes $J$ in the tests. Throughout, we let $\epsilon=1$. The factor potentials and the counting numbers are set consistently for all the experiments. In particular, the factor potentials are chosen in a way such that the variable nodes connecting to a common factor node are strongly correlated. In our examples, all the factors are connected to only two variable nodes. This choice amounts to taking diagonally dominant matrices as potentials. The counting numbers are selected using the strategy in Proposition \ref{prop:counting} by setting $c_j=0,\,\forall j\in V$ and $c_\alpha=c_\beta,\,\forall \alpha,\beta\in F$.
In all our experiments, we observe that the three algorithms converge to the same solutions.
To fairly compare the computation complexity of the three algorithms, we use a unified stopping criteria; the algorithms stop when the relative error with respect to the ``ground truth'' $\mb^*$ in terms of the 1-norm is less than $10^{-4}$. The ``ground truth'' $\mb^*$ is obtained by running one of the algorithms (e.g., Vanilla IS) for sufficiently many iterations so that the duality gap is less than $10^{-8}$. 

%
\begin{figure}[tb]
    \centering
    \begin{subfigure}{.25\textwidth}
    \centering
    \includegraphics[width=1\textwidth]{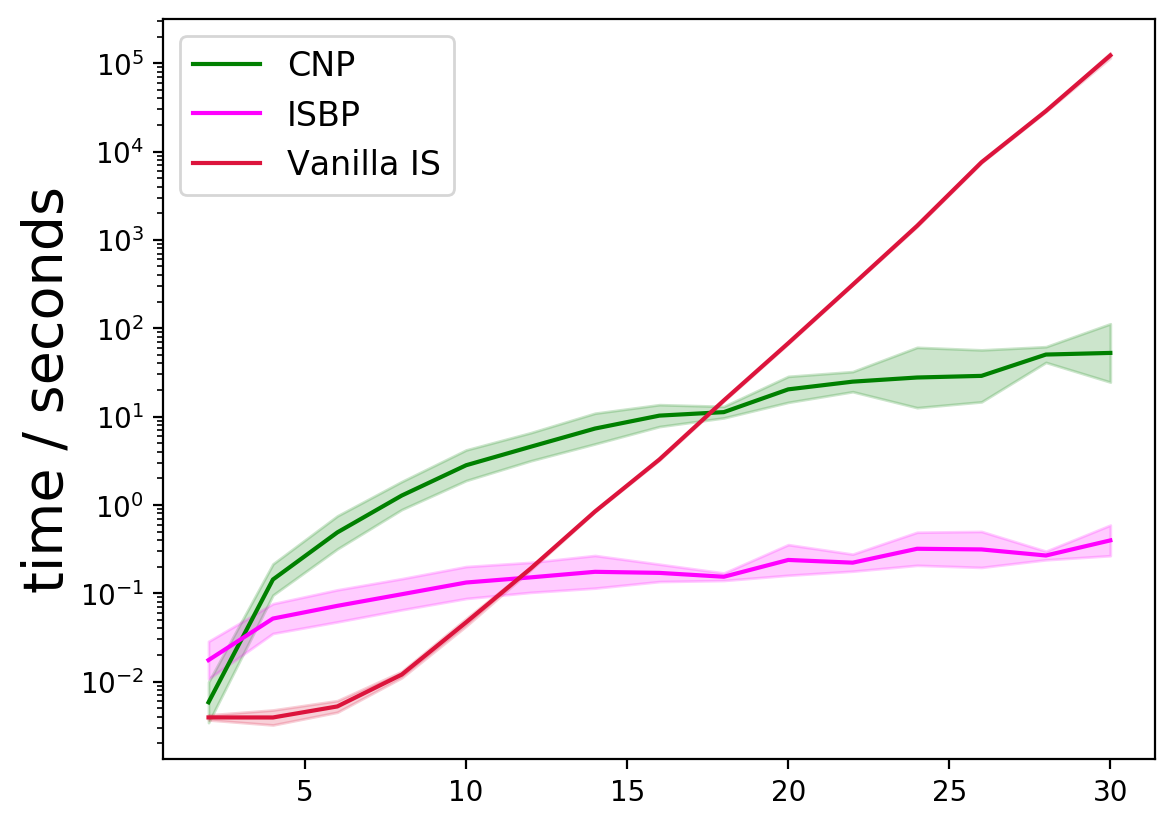}
    \end{subfigure}%
    \begin{subfigure}{.24\textwidth}
    \centering
    \includegraphics[width=1\textwidth]{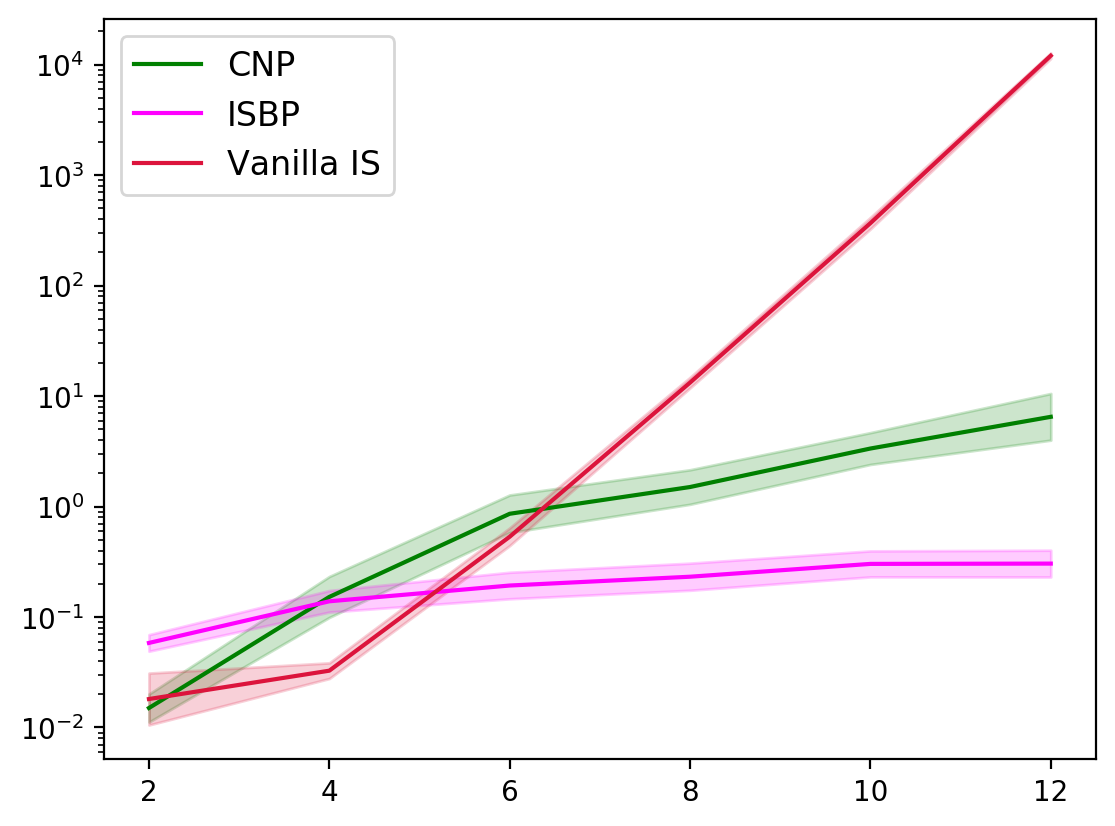}
    \end{subfigure}%
    \begin{subfigure}{.24\textwidth}
    \centering
    \includegraphics[width=1\textwidth]{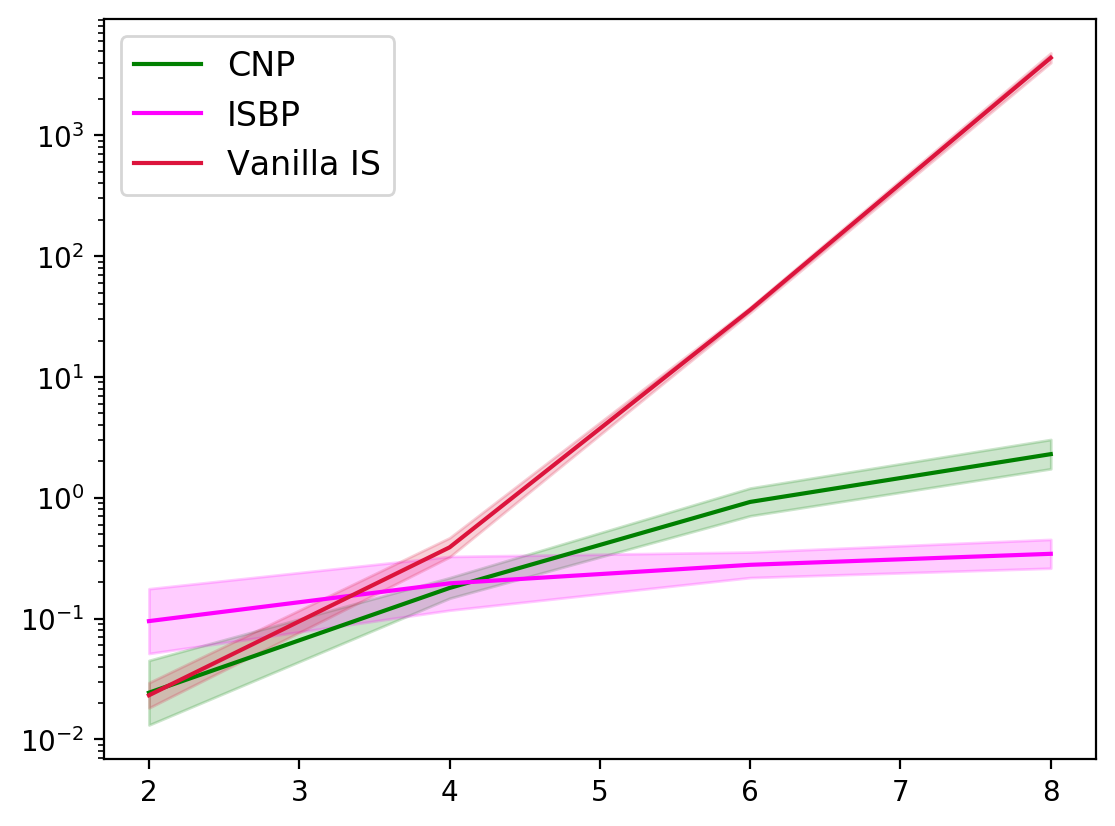}
    \end{subfigure}

    \centering
    \begin{subfigure}{.25\textwidth}
    \centering
    \includegraphics[width=1\textwidth]{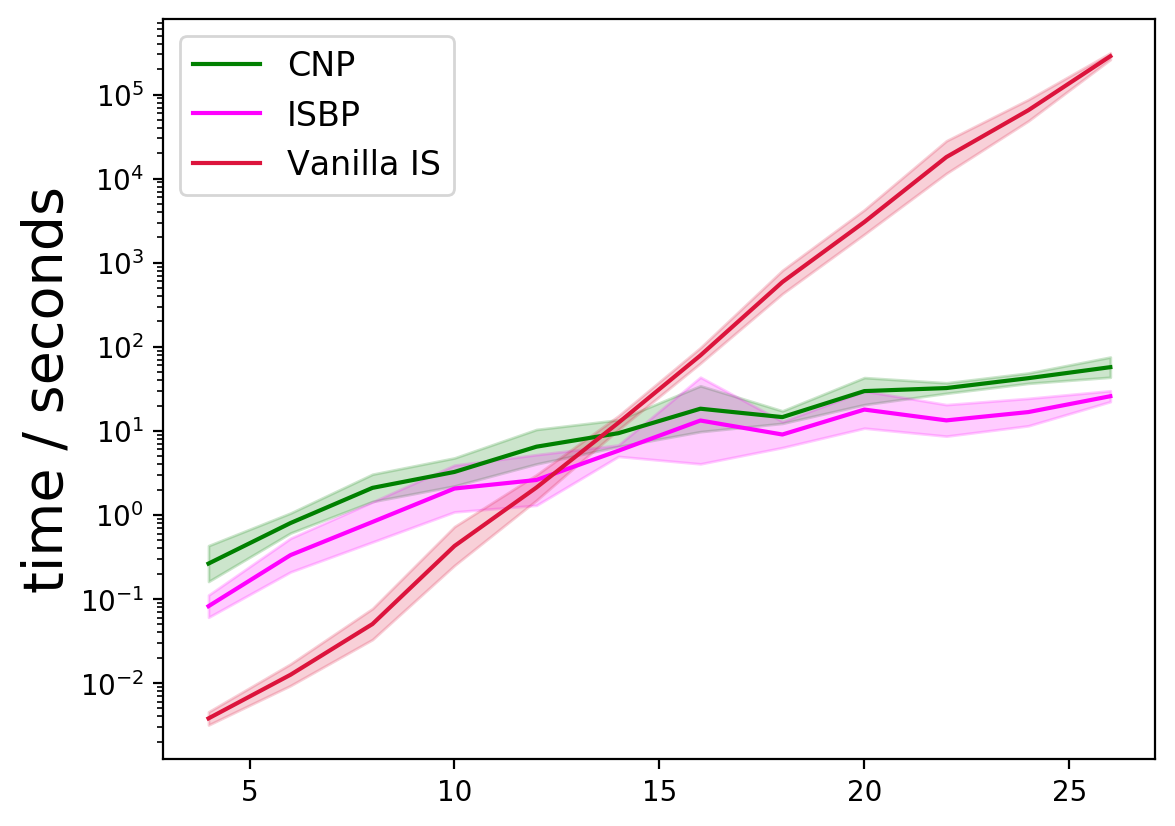}
    \end{subfigure}%
    \begin{subfigure}{.24\textwidth}
    \centering
    \includegraphics[width=1\textwidth]{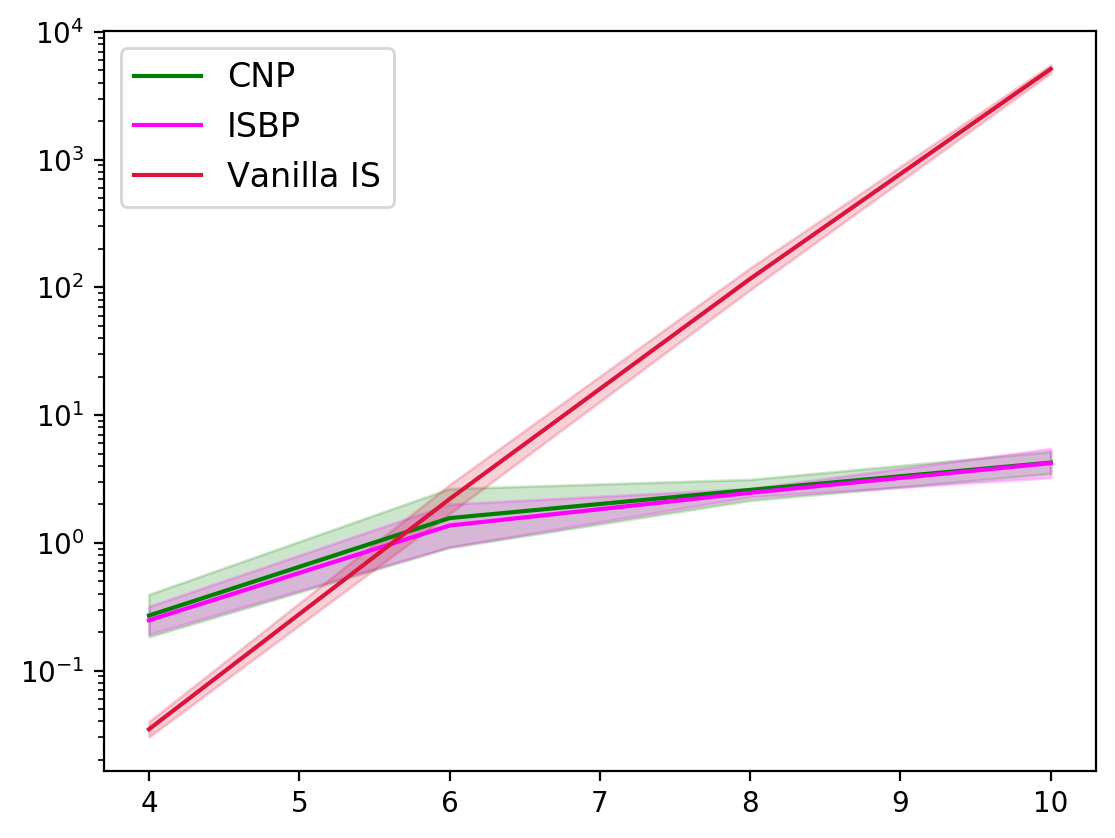}
    \end{subfigure}%
    \begin{subfigure}{.24\textwidth}
    \centering
    \includegraphics[width=1\textwidth]{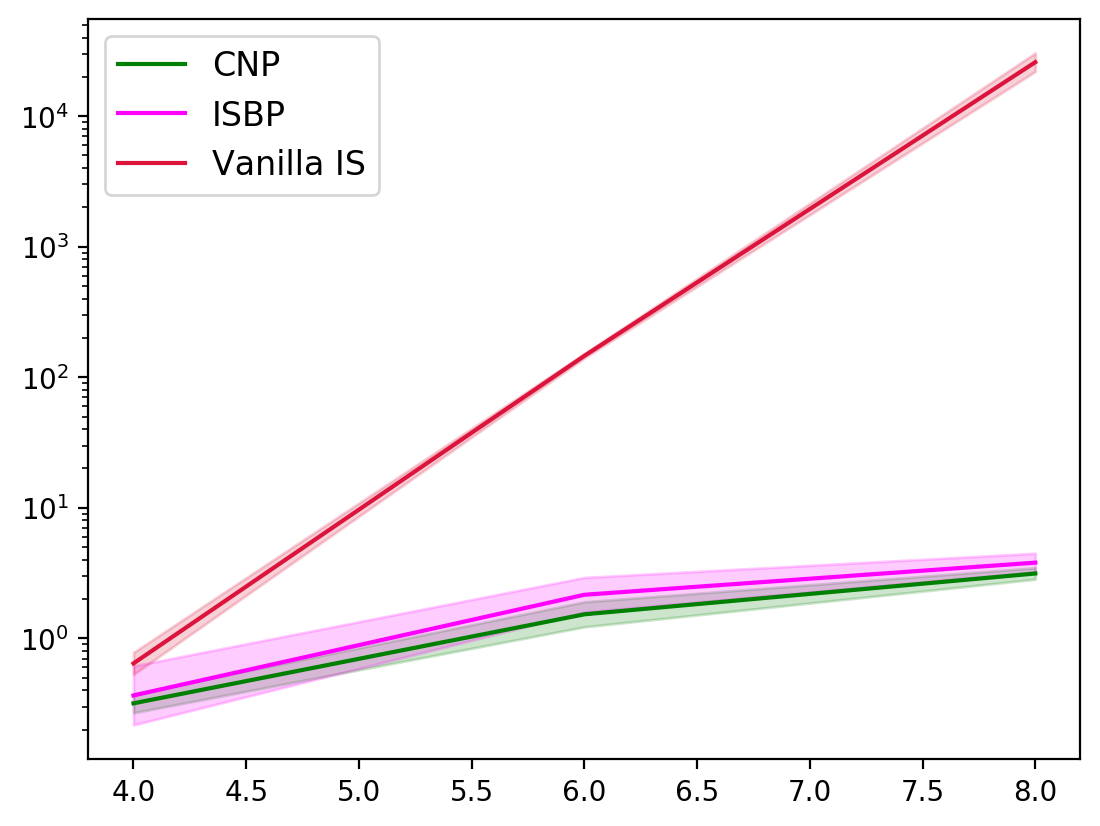}
    \end{subfigure}

    \centering
    \begin{subfigure}{.25\textwidth}
    \centering
    \includegraphics[width=1\textwidth]{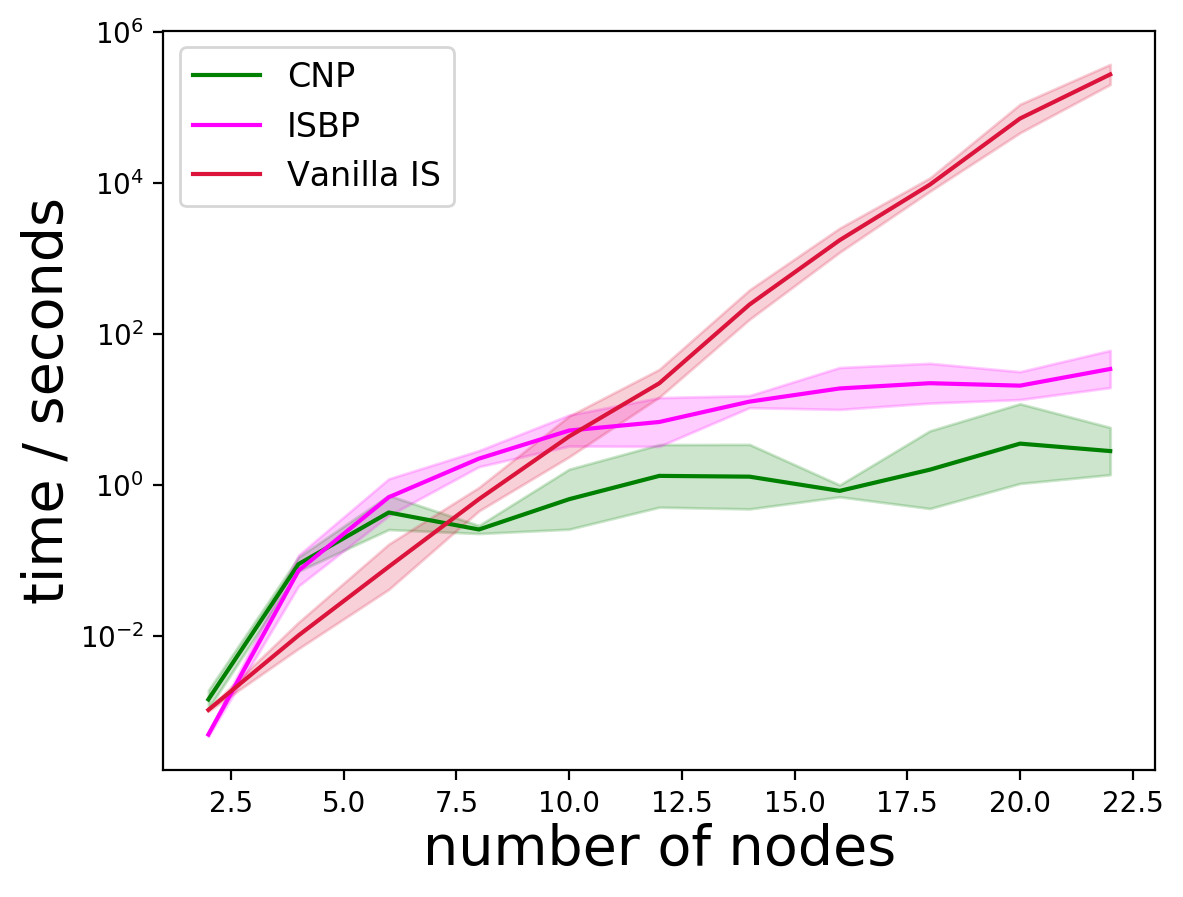}
    \caption{$d=2$}
    \end{subfigure}%
    \begin{subfigure}{.24\textwidth}
    \centering
    \includegraphics[width=1\textwidth]{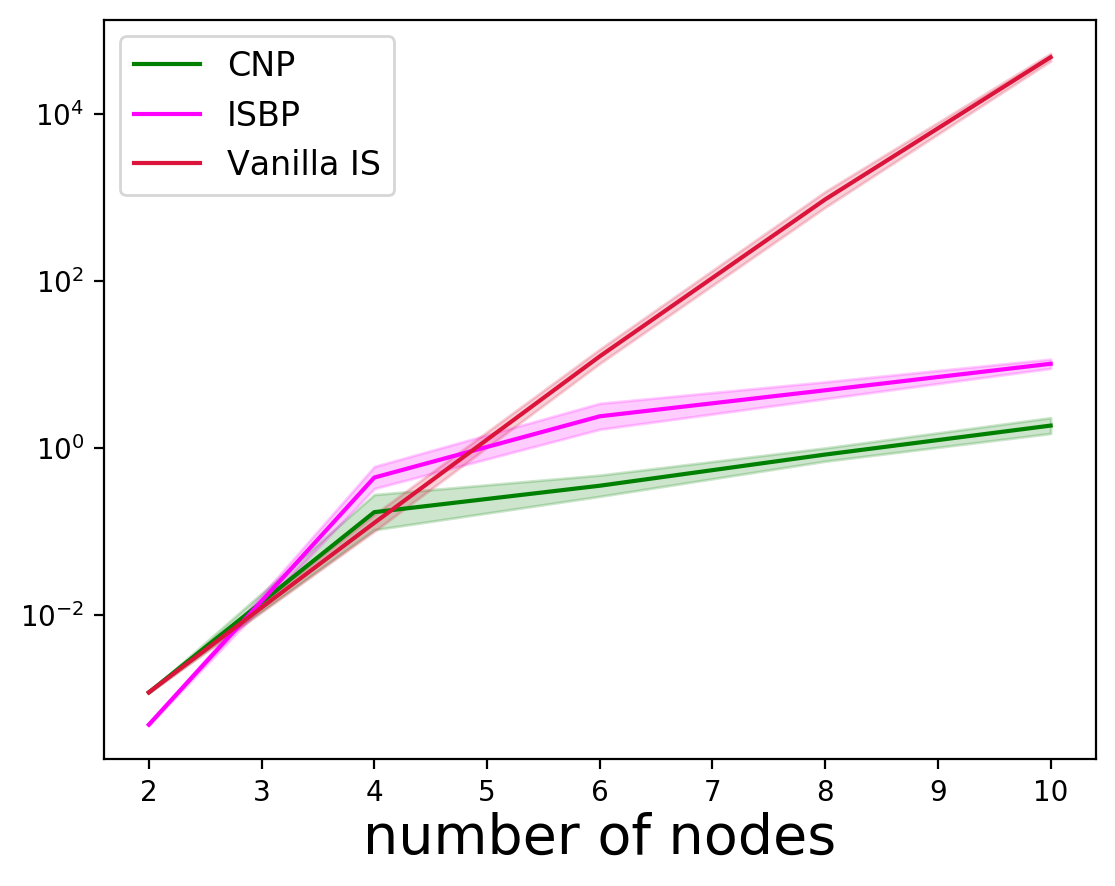}
    \caption{$d=5$}
    \end{subfigure}%
    \begin{subfigure}{.24\textwidth}
    \centering
    \includegraphics[width=1\textwidth]{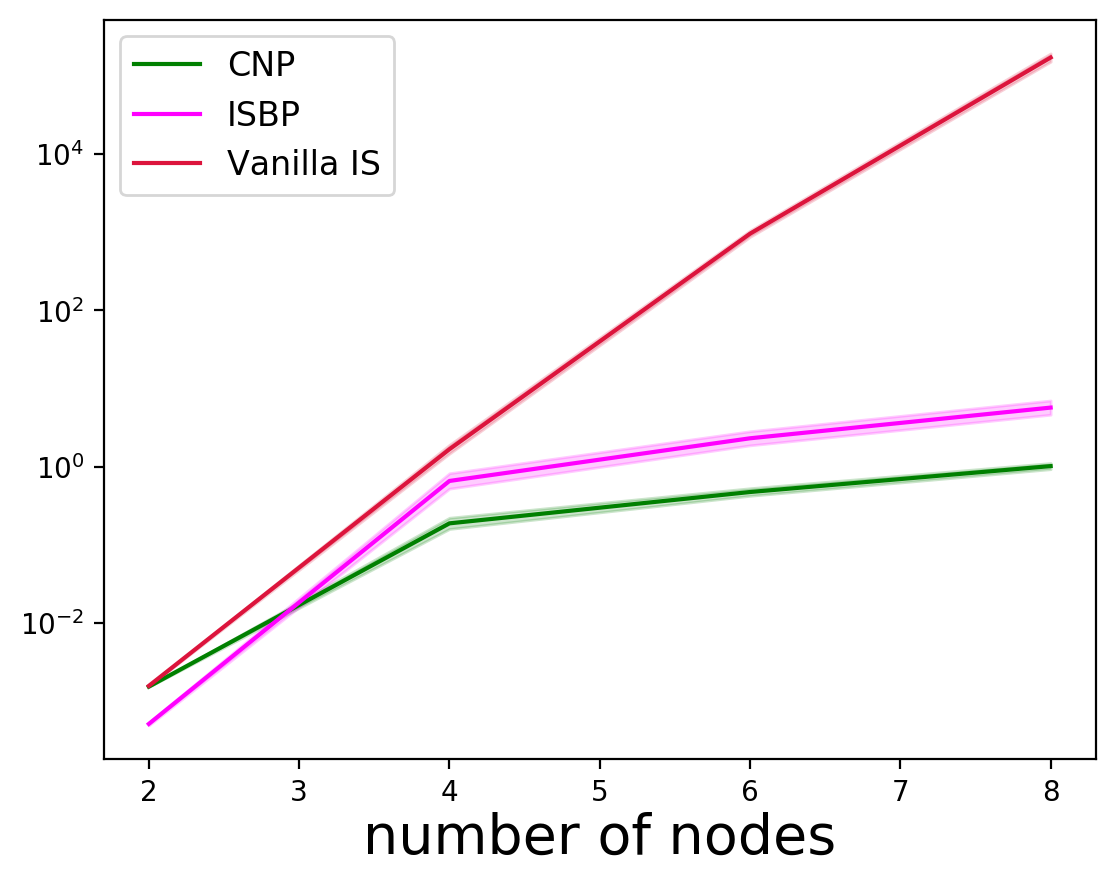}
    \caption{$d=10$}
    \end{subfigure}
    \caption{Comparison among CNP, ISBP and Vanilla IS. The three rows, from top to bottom, correspond to examples with line graph, HMM, and star graph, respectively. The subplots in the same column have the same value of $d$.}
    \label{fig:comparision}
\end{figure}
Figure~\ref{fig:comparision} depicts the evaluation results of the three algorithms under different configurations. The $y$-axis shows the total time consumption of the algorithms before they stop. The $x$-axis represents the size of the graphs, more specifically, the number of nodes $J$ of the graphs being used. Thus, each subplot showcases the relation between computational efficiency and the number of nodes of the graphs. The dependence of the computational complexity on the number $d$ of discrete states at each node can be understood by comparing the subplots along each rows. Each row of subplots corresponds to a type of graph. Thus, the effect of the graph topology on the computational complexity can be captured by comparing the subplots in the same column. 
From the results it can be seen that, for all types of graphs, and all values of $d$, the complexity of the Vanilla Iterative Scaling grows exponentially as the number of nodes $J$ increases. In contrast, both CNP and ISBP scale much better than Vanilla IS when $J$ increases. Moreover, CNP and ISBP seem to be less sensitive to the number of discrete states $d$ at each node, compared with Vanilla IS. 

To comprehensively compare the performances of ISBP and CNP, we conduct several more experiments on graphs of larger sizes where the Vanilla Iterative Scaling is no longer applicable. Besides the three graphs used in the previous experiment, we study an additional star shape graph with more nodes on each branch (Figure \ref{fig:graphd}). The stopping criteria is the same as before; the algorithms stop when the relative error with respect to a ``ground truth'' $\mb^*$ in terms of 1-norm is less than $10^{-4}$. Since Vanilla IS is computationally forbidden for large $J$, we run the ISBP algorithm for sufficiently many iterations to obtain $\mb^*$. 
The experiments results are summarized in Figure \ref{fig:1000}. The presentation of the results in Figure \ref{fig:1000} is similar to that in Figure \ref{fig:comparision}, so that we can understand the dependence of the computational complexity over the number of discrete states $d$, number of nodes $J$ and graph topology. 
From the figures we can see that the two algorithms CNP and ISBP, have comparable performances. Both of them scale well when $J$ and $d$ increase. ISBP behaves better on line graphs and HMMs, while CNP is faster on star shaped graphs. 
%
%
%
%
\begin{figure}[tb]
    \centering
    \begin{subfigure}{.25\textwidth}
    \centering
    \includegraphics[width=1\textwidth]{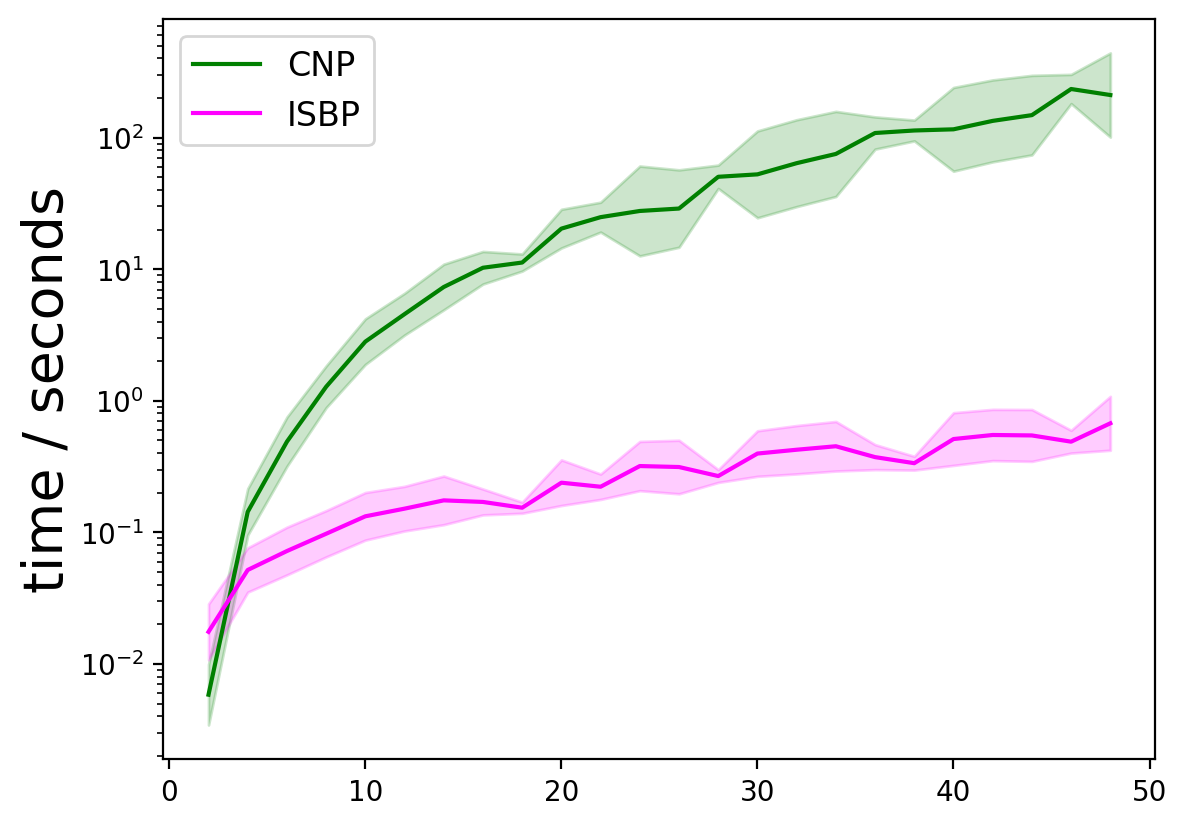}
    \end{subfigure}%
    \begin{subfigure}{.24\textwidth}
    \centering
    \includegraphics[width=1\textwidth]{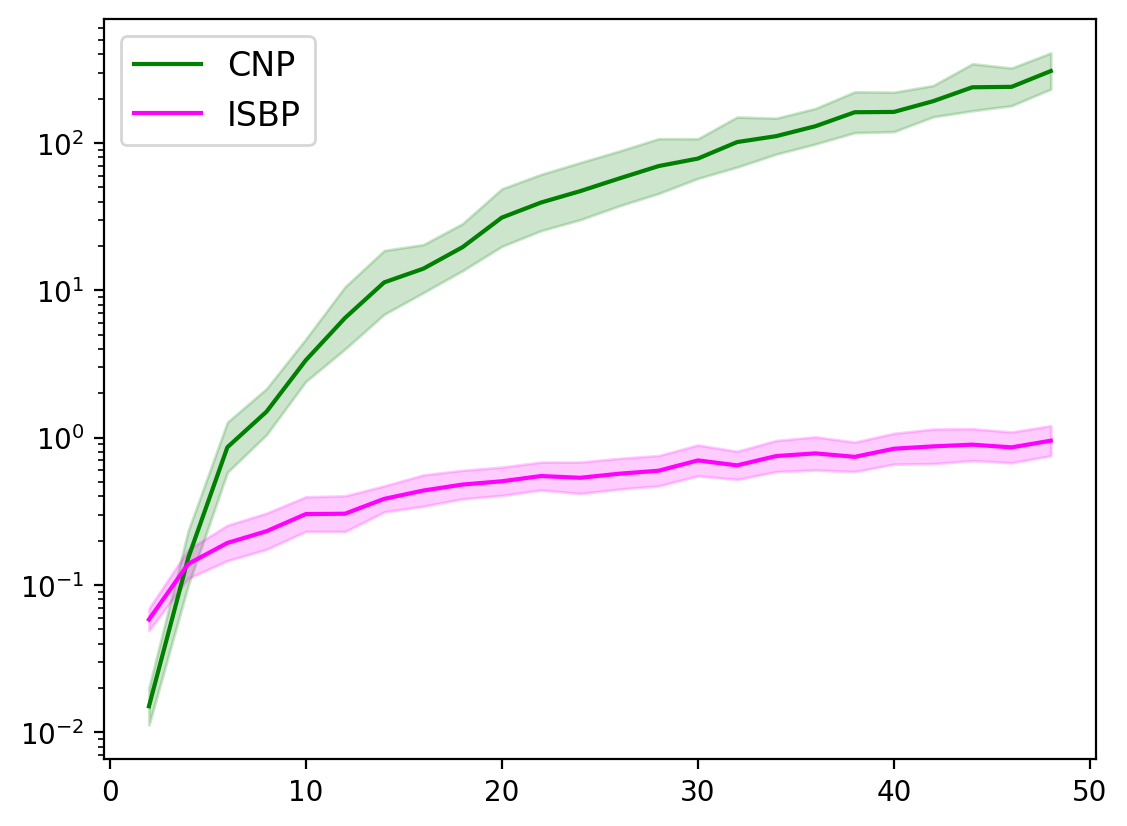}
    \end{subfigure}%
    \begin{subfigure}{.24\textwidth}
    \centering
    \includegraphics[width=1\textwidth]{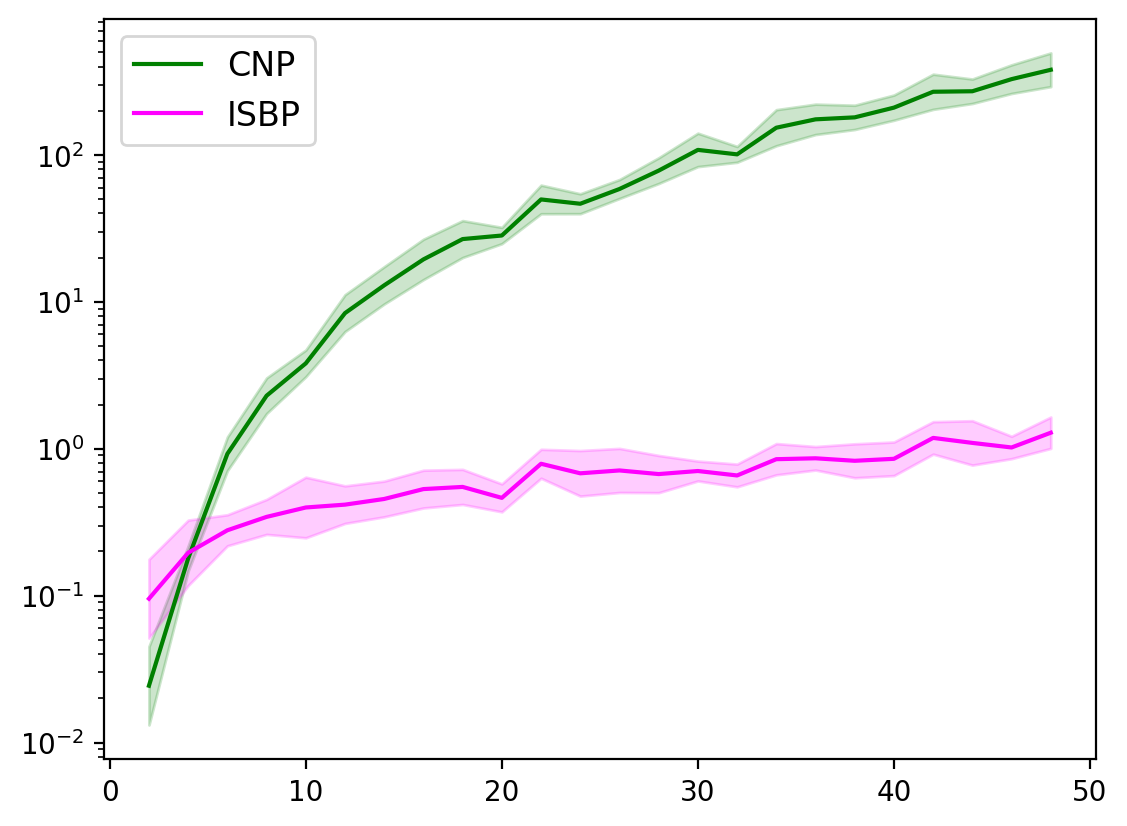}
    \end{subfigure}%
    \begin{subfigure}{.24\textwidth}
    \centering
    \includegraphics[width=1\textwidth]{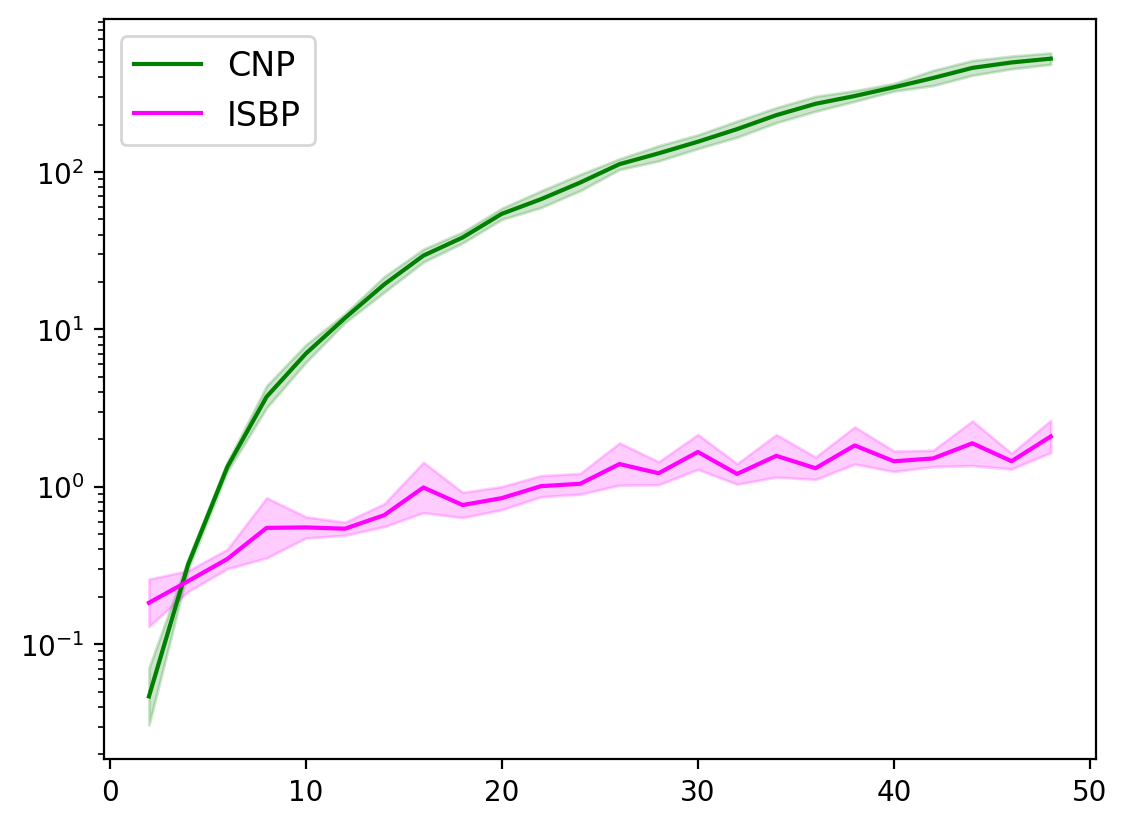}
    \end{subfigure}
    
    \centering
    \begin{subfigure}{.25\textwidth}
    \centering
    \includegraphics[width=1\textwidth]{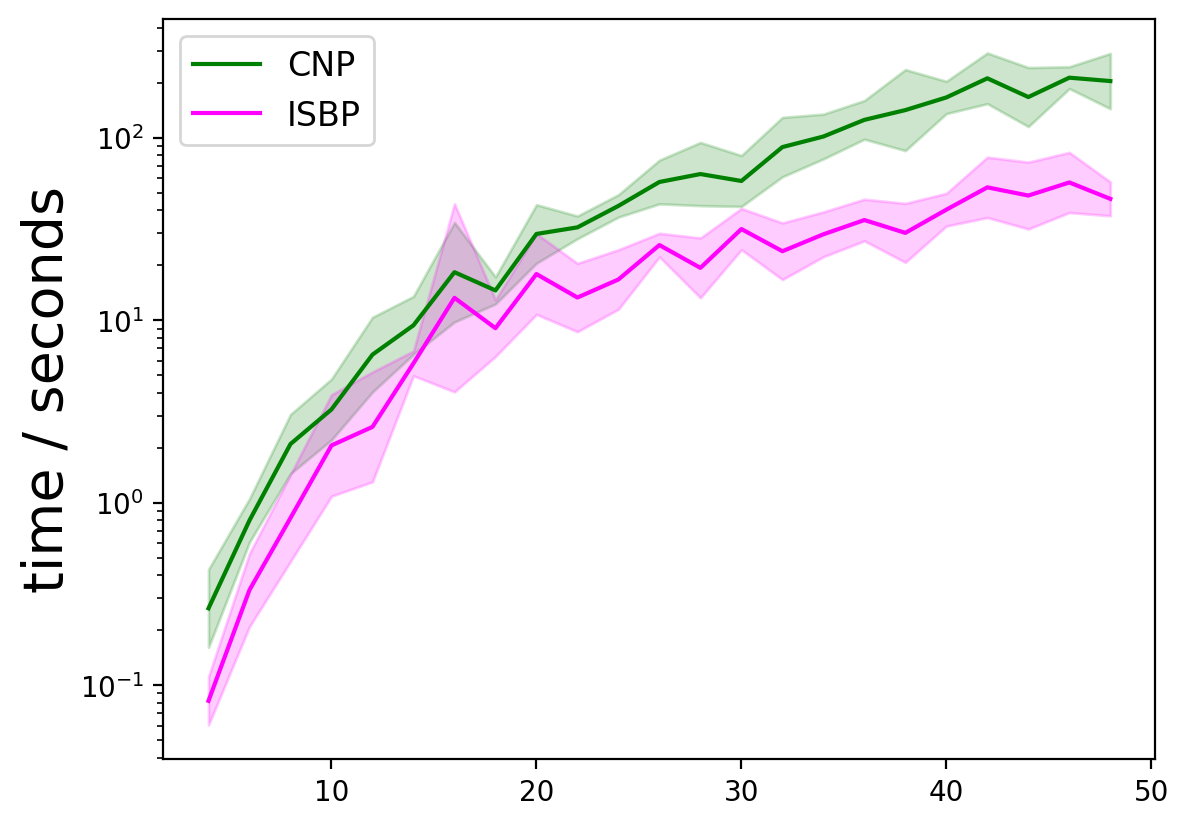}
    \end{subfigure}%
    \begin{subfigure}{.24\textwidth}
    \centering
    \includegraphics[width=1\textwidth]{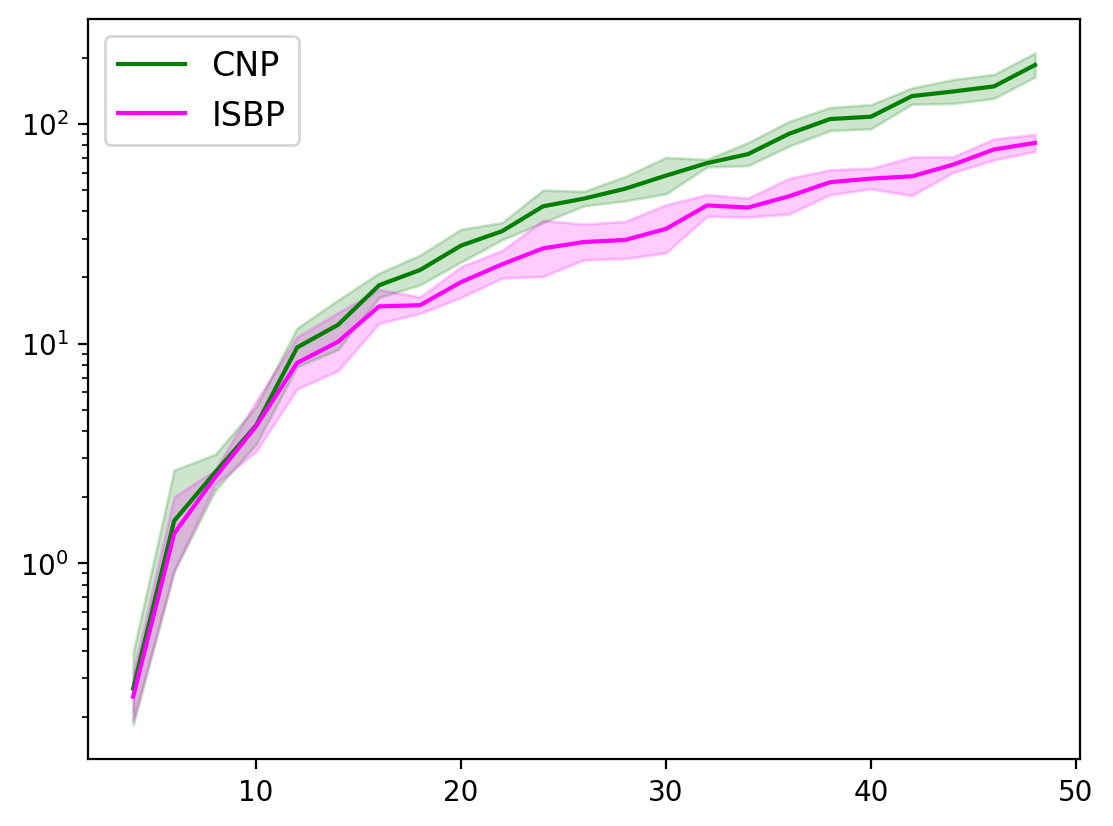}
    \end{subfigure}%
    \begin{subfigure}{.24\textwidth}
    \centering
    \includegraphics[width=1\textwidth]{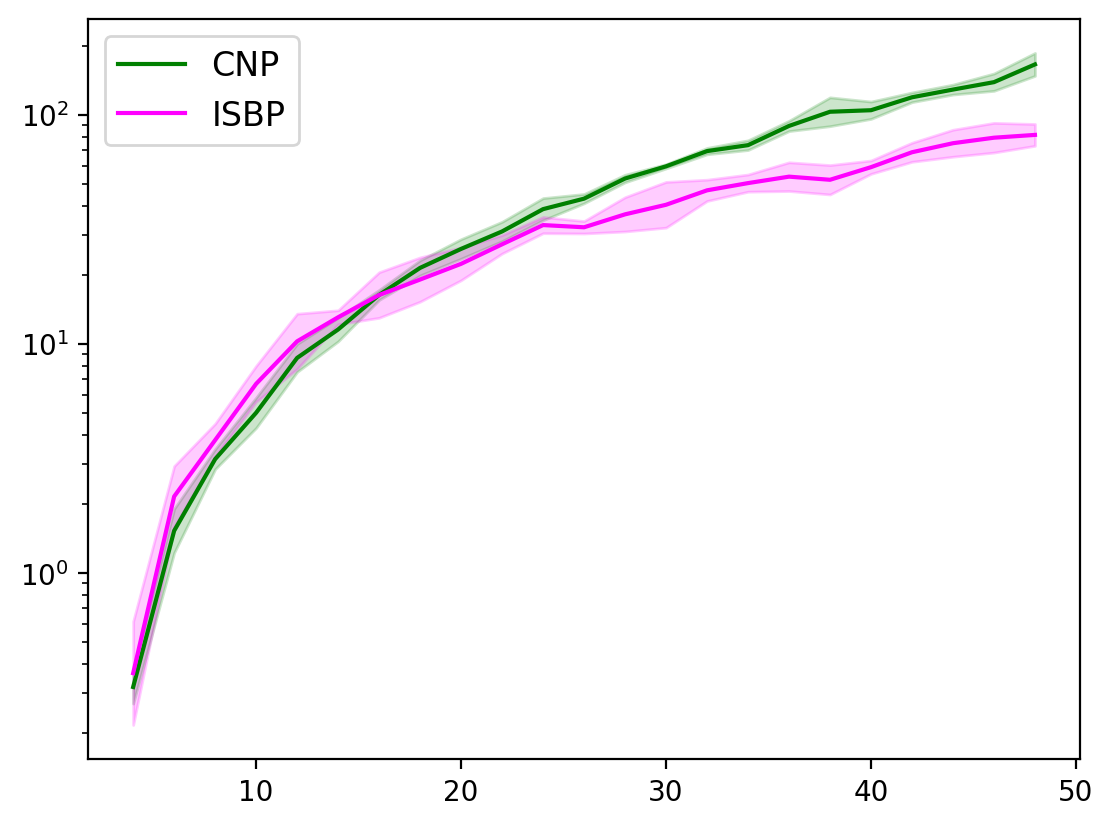}
    \end{subfigure}%
    \begin{subfigure}{.24\textwidth}
    \centering
    \includegraphics[width=1\textwidth]{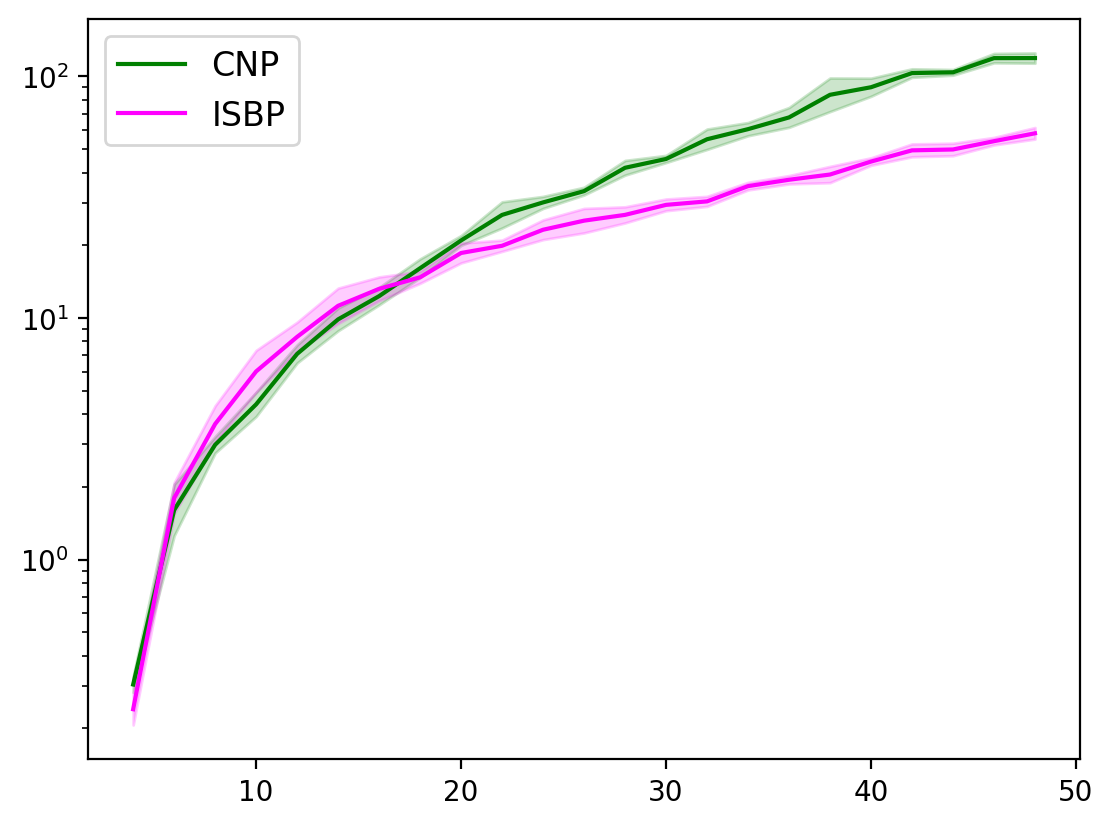}
    \end{subfigure}
    
    \centering
    \begin{subfigure}{.25\textwidth}
    \centering
    \includegraphics[width=1\textwidth]{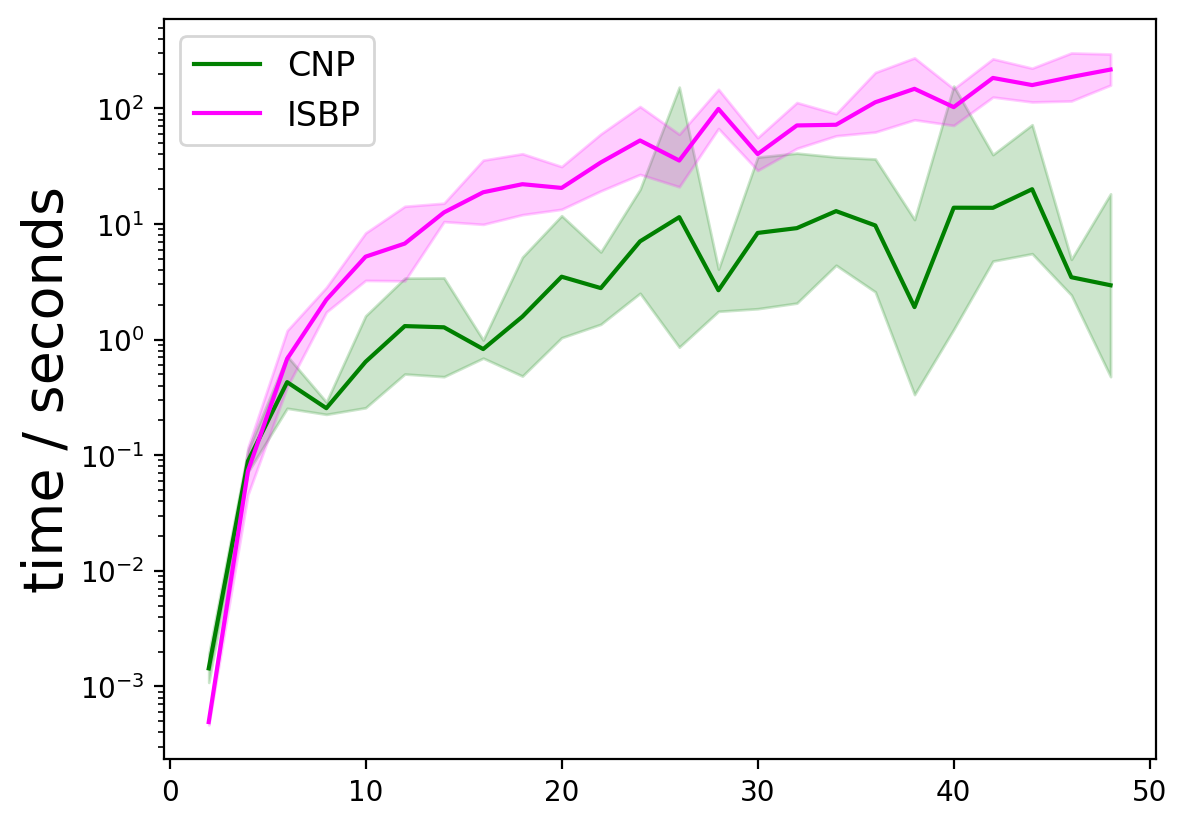}
    \end{subfigure}%
    \begin{subfigure}{.24\textwidth}
    \centering
    \includegraphics[width=1\textwidth]{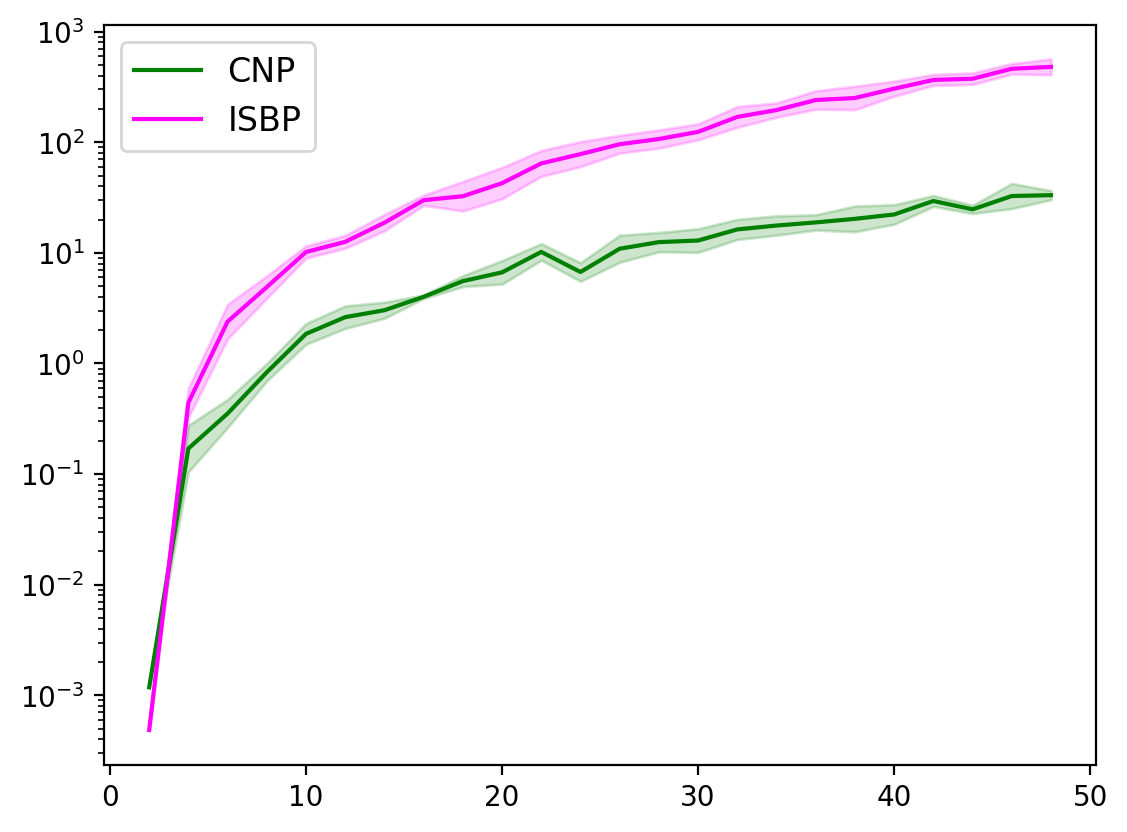}
    \end{subfigure}%
    \begin{subfigure}{.24\textwidth}
    \centering
    \includegraphics[width=1\textwidth]{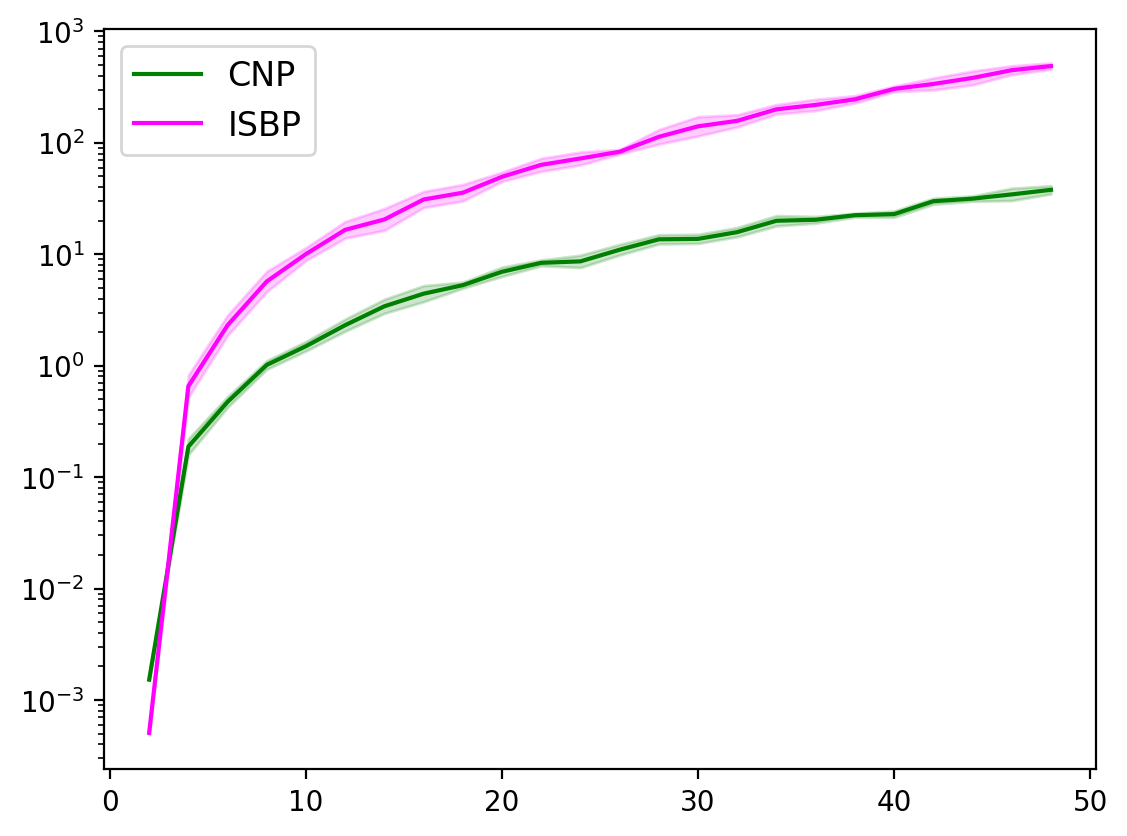}
    \end{subfigure}%
    \begin{subfigure}{.24\textwidth}
    \centering
    \includegraphics[width=1\textwidth]{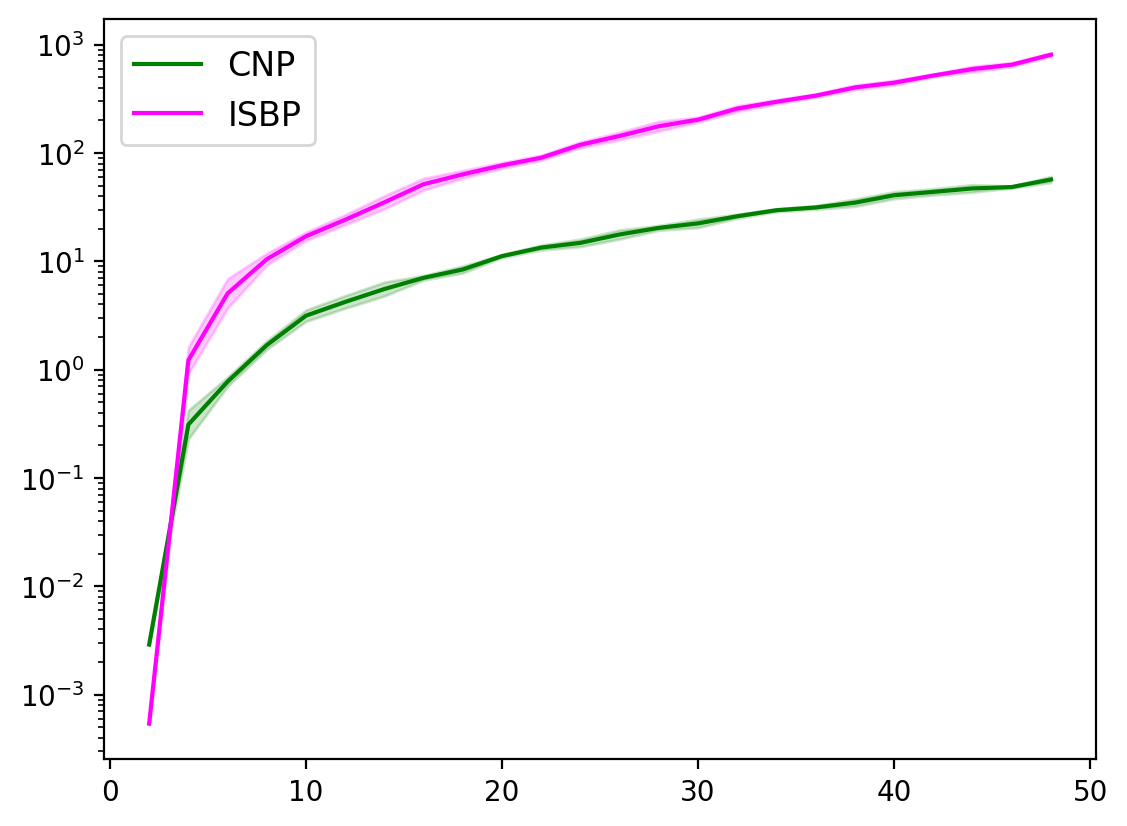}
    \end{subfigure}

    \centering
    \begin{subfigure}{.25\textwidth}
    \centering
    \includegraphics[width=1\textwidth]{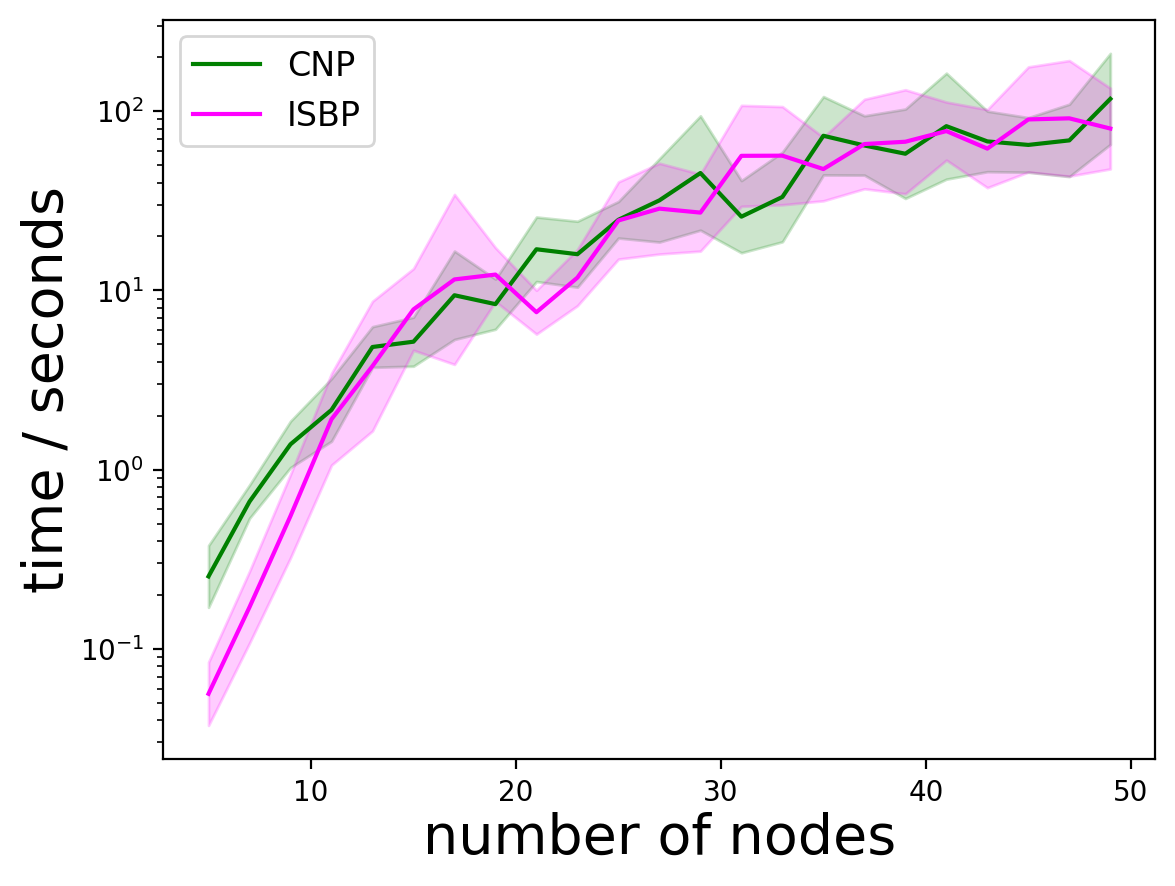}
    \caption{$d=2$}
    \end{subfigure}%
    \begin{subfigure}{.24\textwidth}
    \centering
    \includegraphics[width=1\textwidth]{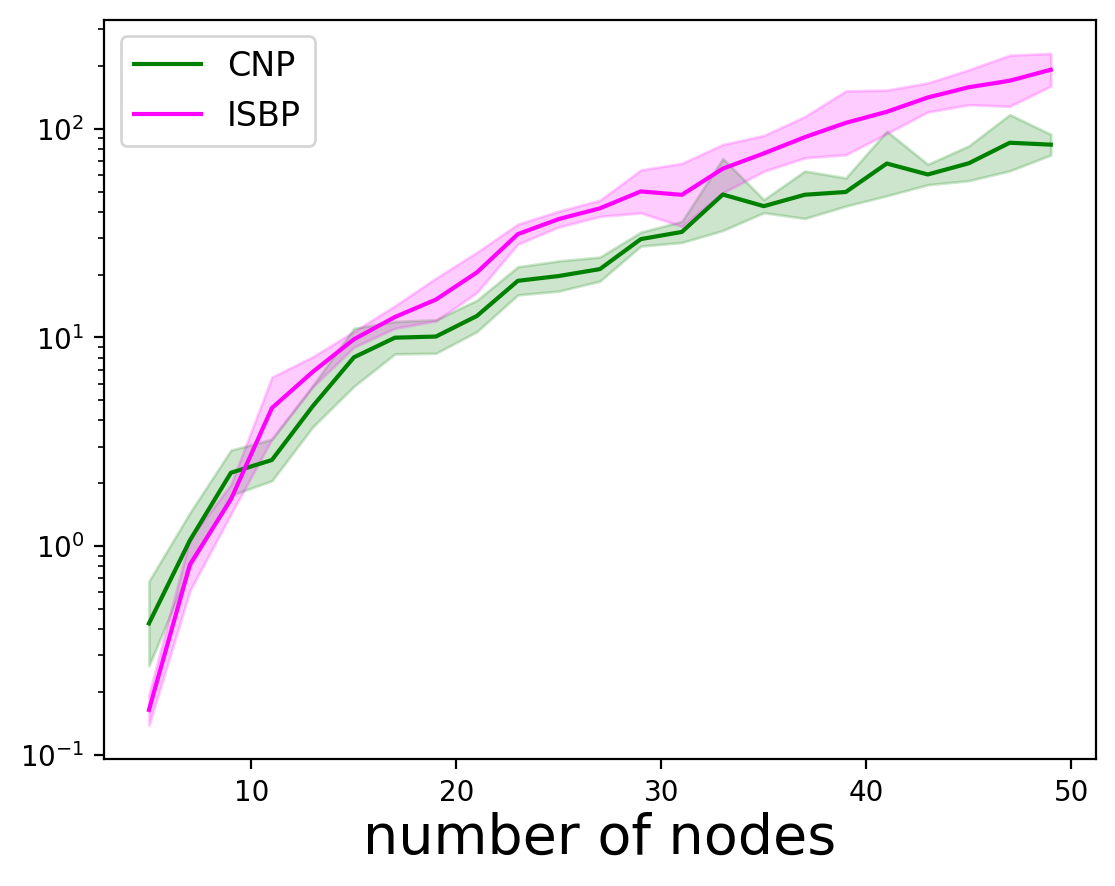}
    \caption{$d=5$}
    \end{subfigure}%
    \begin{subfigure}{.24\textwidth}
    \centering
    \includegraphics[width=1\textwidth]{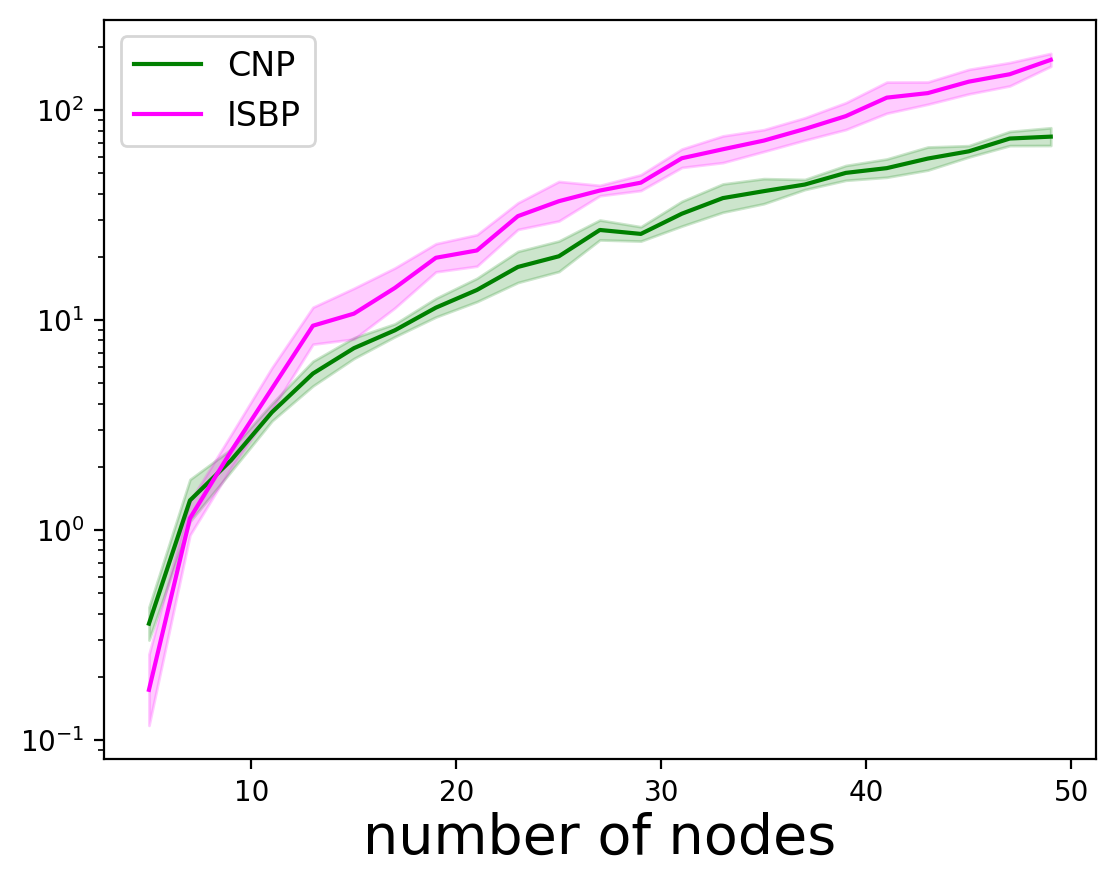}
    \caption{$d=10$}
    \end{subfigure}%
    \begin{subfigure}{.24\textwidth}
    \centering
    \includegraphics[width=1\textwidth]{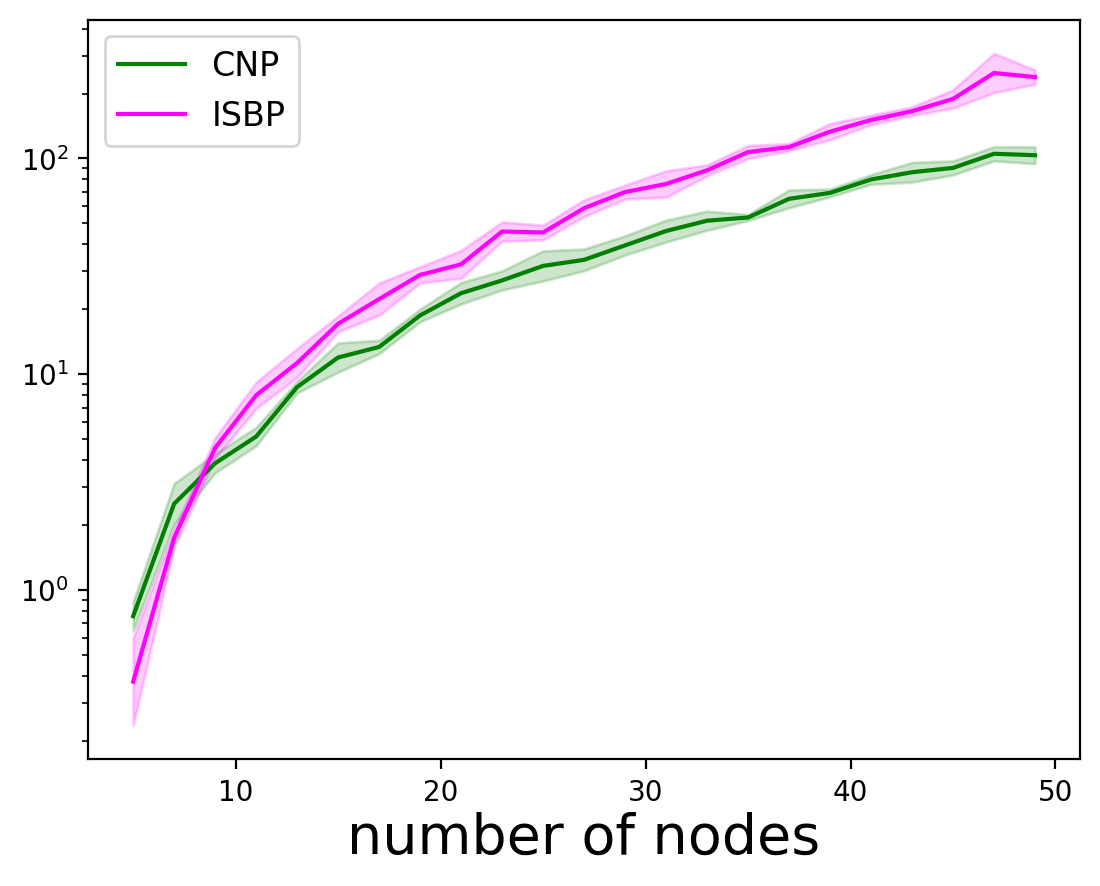}
    \caption{$d=20$}
    \end{subfigure}

    \caption{Performance evaluation of CNP and ISBP. The four rows, from top to bottom, are associated with line graph, HMM, star graph and long star graph respectively. The subplots in the same column share the same value of $d$.}
    \label{fig:1000}
\end{figure}

\subsection{Filtering for collective dynamics} 
To demonstrate the performance of the methods developed in this paper, 
we present an example to demonstrate the effectiveness of our MOT framework in filtering problems for collective dynamics.
Consider a synthetic particle ensemble with $10000$ agents moving over a $20 \times 20$ grid, aiming from bottom-left corner to top-right corner. 
The dynamics of the agents follow a log-linear distribution characterized by four factors: the distance between two positions, the angle between the movements direction and an external force, the angle between the direction of movement and the direction to the goal, and the preference to stay in the original cell. 
The weights for the log-linear model associated with these four factors are set to be $(3,5,5,10)$. This model has been used to model the migration of birds \cite{SunSheKum15,SinHaaZha20}.

There are 16 sensors placed over the grid as shown in Figure~\ref{fig:ensemble_flowa}. These sensors can not measure the exact locations of the agents. Instead, the measurement of each sensor is a count of agents it currently observes. The probability of an observation decreases exponentially as the distance between the sensor and the agent increases.This type of sensors show up in many real applications. 
For instance, the sensors can be Wi-Fi hotspots, or cell phone based stations, which can measure the number of phones connected to them. 
Our goal is to estimate the movement of the whole population using this limited sensor information. 

This filtering problem for collective dynamics can be modeled as a MOT problem, or equivalently a constrained marginal inference problem in our framework. In particular, the agents form a HMM and the sensor measurements correspond to constraints on marginal distributions over the observation nodes. The number of discrete states at each node is $d=20\times 20 = 400$ and the number of nodes depends on the number of time steps. 
We simulate the model for 15 time steps and run both ISBP and CNP to infer the marginal distributions of the free nodes in order to estimate the group behavior of the $10000$ agents. The agents start in two clusters: one in the left-bottom and one in the center-bottom; both aim to reach the right-top corner of the grid in 15 time steps. The results are depicted in Figure~\ref{fig:ensemble_flowb}. Both CNP and ISBP give the same estimation result and thus we only display one of them in the figures. As can be seen from the plots, even though the sensor data (center column) is hard to interpret visually, our constrained marginal inference framework can still infer the population movements to a satisfying accuracy. 

\begin{figure*}[tb]
    \centering
    \begin{subfigure}[b]{0.48\textwidth}
        \centering
        \includegraphics[scale = 0.6]{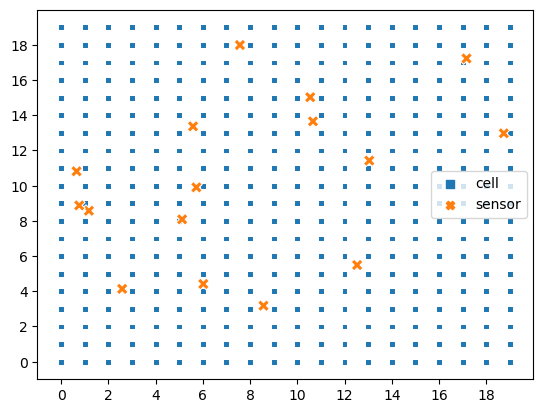}
        \caption{Sensor locations}
        \label{fig:ensemble_flowa}
    \end{subfigure} \hspace{0.5cm}
    \begin{subfigure}[b]{0.48\textwidth}
        \centering
        \includegraphics[scale=1.3]{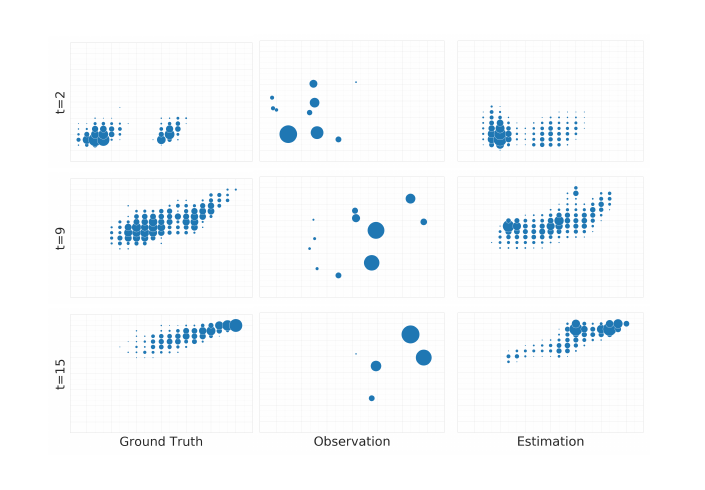}
        \caption{Simulation results}
        \label{fig:ensemble_flowb}
    \end{subfigure}
    \caption{Movement estimation of 10000 agents over a $20 \times 20$ grid for 15 time steps: (a) displays the grid and the locations of the sensors; (b) shows the estimation results. The three columns, from left to right, represents the simulated movement of agents at three time steps $t=2, 9, 15$, the sensor measurements, and the estimated agent distributions respectively. The size of the blue dots is proportional to the number of agents.}
    \label{fig:ensemble_flow}
    \vspace{-0.3cm}
\end{figure*}

\section{Conclusion}\label{sec:conclusion}
We studied multi-marginal optimal transport problems and pointed out an unexpected connection to probabilistic graphical models. This
relation between MOT on graphs and constrained PGMs provides a completely new perspective of both MOT and PGMs, which may have far reaching impact in the future development of both subjects, in both theory and applications. This connection also enables us to adapt the rich class of algorithms in PGMs to tackle difficult MOT problems. In this work, to highlight the key idea of this line of research, we focused on MOT on trees with discrete states. The next step is to generalize the results to more general graphs with cycles as well as continuous state spaces. These are more challenging problems in PGMs, for which exact Bayesian inference is usually too expensive, and one needs to turn to variational inference or sampling based methods \cite{KolFri09}.



\appendix
\subsection{Derivation of Constrained Norm-product algorithm (Algorithm \ref{alg:constrained norm-product})}\label{sec:proofCNP}
We follow the primal-dual ascent algorithm stated in Lemma \ref{lem:Primal-Dual Ascent}. 
Denote $\boldsymbol\lambda _j = \left\{ \lambda_{j,\alpha}(\mx_\alpha)\right\}$ and $\boldsymbol\nu_{j} = \left\{ \nu_{j,\alpha}(\mx_\alpha)\right\}$. For the sake of convenience, we introduce the following notation
\begin{align}
    \label{eq:hat_psi} \hat{\psi}_{j,\alpha}(\mx_\alpha) &:= \psi_\alpha(\mx_\alpha)\exp(-\nu_{j,\alpha}(\mx_\alpha)), \\
    \hat{c}_{j\alpha} &:= c_\alpha + c_{j\alpha}, \nonumber \\
    \hat{c}_j &:= c_j + \sum_{\alpha \in N(j)} c_\alpha . \nonumber
\end{align}

For a fixed $1\le j\le J$, the step \eqref{eq:pdab} in the primal-dual ascent algorithm requires solving
	\[
		\min_{\mb \in dom(f)\cap dom(h_j)} f(\mb)+h_j(\mb)+\mb^T\boldsymbol\nu_j.
	\]
Recall that $f(\mb) = \hat f(\mb) + \delta_\cP(\mb)$, $h_j(\mb) = \hat h_j(\mb) + \delta_\cM(\mb)$ where
	\begin{eqnarray*}
    \cM &=& \Big\{ \mb : \sum_{\bx_\alpha \backslash x_j} b_\alpha(\bx_\alpha) = b_j(x_j), \forall j\in V, \alpha\in N(j),\, b_j(x_j) = \mu_j(x_j), \, \forall j \in \Gamma \Big\},
\\
    \cP &=& \Big\{ \mb: \sum_{\bx_\alpha} b_\alpha(\bx_\alpha)= 1, \quad  \forall \alpha \in F \Big\},
\\
    \hat{f}(\mb) &=&    - \sum_{\mx_{\alpha},\alpha\in F} b_\alpha   (\mx_\alpha) \ln \psi_\alpha(\mx_\alpha)  -
        \sum_{\alpha\in F} \epsilon\, c_\alpha \cH(\mb_\alpha),
\\
  \hat{h}_{j}(\mb) &=& - \sum_{x_j} b_j(x_j)\ln \phi_j(x_j)  - \epsilon\, c_j \cH(\mb_j) -
    \sum_{\alpha \in N(j)} \epsilon\, c_{j \alpha} (\cH(\mb_\alpha)-\cH(\mb_j)).
\end{eqnarray*}
Thus, for any fixed $1\le j\le J$, step \eqref{eq:pdab} of the primal-dual ascent algorithm can be reformulated as
\begin{subequations}\label{eq:general}
\begin{eqnarray}\label{eq:generala}
    \min_{\mb_j,\mb_\alpha,\alpha \in N(j)}\!\!\!\!&&\!\!\!\! 
    \Bigg\{
        \!-\!\sum_{x_j}b_j(x_j)\ln\phi_j(x_j)
        \!-\!\sum_{\alpha \in N(j)}\sum_{\mx_\alpha} b_\alpha (\mx_\alpha)\ln\hat{\psi}_{j,\alpha}(\mx_\alpha)
        \!-\!\epsilon\hat{c}_{j}\cH(\mb_j)
        \!-\!\sum_{\alpha \in N(j)} \epsilon \hat{c}_{j\alpha}(\cH(\mb_\alpha) - \cH(\mb_j))
    \Bigg\}
\\\label{eq:generalb}
    \mbox{subject to}\!\!\!\!&&\!\!\!\!
    \sum_{\mx_\alpha} b_\alpha(\mx_\alpha) = 1 , \quad
    \sum_{\mx_\alpha \backslash x_j} b_\alpha(\mx_\alpha) = b_j(x_j), \quad
    \forall x_j, \alpha \in N(j), 
    \\\label{eq:generalc}
    &&\!\!\!\!b_j(x_j) = \mu_j(x_j) , \quad \forall x_j, ~\mbox{if}~j \in \Gamma.
\end{eqnarray}
\end{subequations}
Note that when $j\in \Gamma$, the above problem has an extra constraint \eqref{eq:generalc} compared to the cases where $j\notin \Gamma$.

We next derive a closed form solution $\mb_j^*, \mb_\alpha^*$ to \eqref{eq:general}. The constraint \eqref{eq:generalb} implies that $\mb_j$ is a marginal distribution of $\mb_\alpha$, thus, $\mb_\alpha$ can be rewritten in terms of conditional distribution $\mb_{\alpha\mid j}$ as
	\begin{equation}\label{eq:conditional}
		b_\alpha(\mx_\alpha) = b_j(x_j) b_{\alpha\mid j}(\mx_\alpha\mid x_j).
	\end{equation}
The entropy $\cH(\mb_\alpha)$ can be rewritten as \cite{CovTho12}
	\[
		\cH(\mb_\alpha) = \cH(\mb_j) + \sum_{x_j}b_j(x_j) \cH(\mb_{\alpha | j})
	\]
where
	\[
		\cH(\mb_{\alpha | j}) = -\sum_{\mx_\alpha\backslash x_j} b_{\alpha\mid j}(\mx_\alpha\mid x_j) \ln b_{\alpha\mid j}(\mx_\alpha\mid x_j).
	\]
Thus, in terms of new variables $\mb_j, \mb_{\alpha\mid j}, \alpha\in N(j)$, the optimization problem \eqref{eq:general} reads 
\begin{equation}
    \label{eq:condition}
    \min_{\mb_j}
    \Bigg\{
        -\sum_{x_j}b_j(x_j)\ln\phi_j(x_j)
        -\epsilon\hat{c}_j\cH(\mb_j)
        +\sum_{x_j}b_j(x_j)
            \sum_{\alpha \in N(j)}\epsilon\hat{c}_{j\alpha}
                \bigg[
                    \underbrace{
                        \min_{\mb_{\alpha\mid j}}
                        -\sum_{\mx_\alpha \setminus x_j} b_{\alpha\mid j} (\mx_\alpha \mid x_j)
                        \ln\hat{\psi}_{j,\alpha}^{1 / (\epsilon \hat{c}_{j\alpha})}(\mx_\alpha)
                        -\cH(\mb_{\alpha \mid j})
                    }_{\star}
                \bigg]
    \Bigg\}
\end{equation}
together with the extra constraint $\mb_j = \boldsymbol\mu_j$ if $j\in\Gamma$. One advantage of this reformulation is that the problem now can be optimized over $\mb_{\alpha\mid j}$ first and then over $\mb_j$. 

Minimizing \eqref{eq:condition} over $\mb_{\alpha\mid j}$ is a standard exercise, and the minimizer is
	\[
		b_{\alpha\mid j}^*(\mx_\alpha\mid x_j) =  \hat{\psi}_{j,\alpha}(\mx_\alpha)^{1/\epsilon \hat{c}_{j\alpha}}/\sum_{\mx_\alpha \setminus x_j} \hat{\psi}_{j,\alpha}(\mx_\alpha)^{1/\epsilon \hat{c}_{j\alpha}}.
	\]
Thus, the value for block $(\star)$ is 
\begin{align*}
    (\star) = -\ln{\sum_{\mx_\alpha \setminus x_j} \hat{\psi}_{j,\alpha}(\mx_\alpha)^{1/\epsilon \hat{c}_{j\alpha}}}.
\end{align*}
Denote
\begin{equation} \label{eq:def_m}
    m_{\alpha \rightarrow j}(x_j) = \left(
        \sum_{\mx_\alpha \setminus x_j} \hat{\psi}_{j,\alpha}(\mx_\alpha)^{1/\epsilon \hat{c}_{j\alpha}}
    \right)^{\epsilon \hat{c}_{j\alpha}},
\end{equation}
then \eqref{eq:condition} with optimal $\mb_{\alpha\mid j}$ can be simplified as
\begin{equation}\label{eq:outerloop}
    \min_{\mb_j}
    \left[ 
      -\cH(\mb_j)
      -\sum_{x_j}b_j(x_j)\ln\phi_j^{1/ \epsilon\hat{c}_j}(x_j)
      \prod_{\alpha \in N(j)} m_{\alpha\rightarrow j}^{1/ \epsilon\hat{c}_j}(x_j)    
    \right].
\end{equation}
When $j\in\Gamma$, $\mb_j = \boldsymbol\mu_j$ is the only feasible point; $\mb_j^*=\boldsymbol\mu_j$. When $j\notin \Gamma$, \eqref{eq:outerloop} is again a standard exercise with the unique minimizer being
\begin{equation} \label{eq:belief_i_optimal}
        b_j^*(x_j) \propto \left( 
            \phi_j(x_j)\prod_{\alpha \in N(j)} m_{\alpha\rightarrow j}(x_j)    
        \right)^{1/ \epsilon\hat{c}_{j}}.
\end{equation}
Combining $\mb_j^*$ and $\mb_{\alpha | j}^*$, we obtain
\begin{equation} \label{eq:belief_alpha_optimal}
    b_\alpha^*(\mx_\alpha) = b_j^*(x_j)b_{\alpha | j}^*(\mx_\alpha | x_j) = 
    \frac{b_j^*(x_j)}{m_{\alpha\rightarrow j}^{1/ \epsilon\hat{c}_{j\alpha}}(x_j)}
    \hat{\psi}_{j,\alpha}(\mx_{\alpha})^{1/ \epsilon\hat{c}_{j\alpha}}. 
\end{equation}

Next, we move to step \eqref{eq:pdac} of the primal-dual ascent algorithm (Lemma \ref{lem:Primal-Dual Ascent}), which reads
	\[
		\boldsymbol\lambda_j \leftarrow -\boldsymbol\nu_j - \nabla \hat f(\mb^*) + A^T \boldsymbol\sigma  
	\]
with $\boldsymbol\sigma$ being an arbitrary vector. The vector $A^T \boldsymbol\sigma$ spans the orthogonal space to the domain $\cB$ (see Lemma \ref{lem:Primal-Dual Ascent}) of $f$. In our problem, $\cB=\cP$ is the probability simplex, thus $A^T \boldsymbol\sigma$ is in alignment with ${\bf 1}$, the vector with all 1 entries. It follows that 
\begin{equation}
    \lambda_{j,\alpha}(\mx_\alpha) = 
    -\nu_{j,\alpha}(\mx_\alpha)
    - \nabla\hat{f} (b_\alpha^* (\mx_\alpha))
    + \sigma_{\alpha} {\bf 1}.
\end{equation}
The value of $\nabla \hat{f} (b_\alpha^*(\mx_\alpha))$ is
\[
    \nabla \hat{f} (b_\alpha^*(\mx_\alpha)) = 
    -\ln{\psi_\alpha (\mx_\alpha)} + \epsilon c_\alpha
        \left( 
            \ln{b_\alpha^*(\mx_\alpha)+1}    
        \right).
\]

Define $n_{j \rightarrow \alpha}(\mx_\alpha)$ as 
\begin{equation}\label{eq:define_n}
    n_{j \rightarrow \alpha}(\mx_\alpha) := \exp(-\lambda_{j,\alpha}(\mx_{\alpha})),
\end{equation}
then, in view of \eqref{eq:belief_alpha_optimal},
\begin{align}\nonumber
    n_{j \rightarrow \alpha}(\mx_\alpha) & \propto \exp(\nu_{j,\alpha}(\mx_\alpha) - \ln \psi_\alpha(\mx_\alpha))(b_\alpha^*(\mx_\alpha))^{\epsilon c_\alpha} \\\nonumber
    &= \hat{\psi}_{j,\alpha}^{-1}(\mx_\alpha)\left(\frac{b_j^*(x_j)}{m_{\alpha \rightarrow j}^{1/\epsilon \hat{c}_{j\alpha}}(x_j)}\right)^{\epsilon c_\alpha} \hat{\psi}_{j,\alpha}^{c_\alpha/\hat{c}_{j\alpha}}(\mx_\alpha) \\\nonumber
    &= \left(\frac{b_j^*(x_j)}{m_{\alpha \rightarrow j}^{1/\epsilon \hat{c}_{j\alpha}}(x_j)}\right)^{\epsilon c_\alpha} \hat{\psi}_{j,\alpha}^{c_\alpha/\hat{c}_{j\alpha} -1 }(\mx_\alpha) \\ \label{eq:nupdate}
     &=\left(\frac{b_j^*(x_j)}{m_{\alpha \rightarrow j}^{1/\epsilon \hat{c}_{j\alpha}}(x_j)}\right)^{\epsilon c_\alpha} 
    \hat{\psi}_{j,\alpha}^{-c_{j \alpha}/\hat{c}_{j\alpha}}(\mx_\alpha),
\end{align}
where the last equation is due to $\hat{c}_{j\alpha} = c_\alpha + c_{j\alpha}$.
By step \eqref{eq:pdaa} of the primal-dual ascent algorithm, $\nu_{j,\alpha}=\sum_{i \in N(\alpha) \setminus j}\lambda_{i,\alpha}(\mx_\alpha)$. Combining it with \eqref{eq:hat_psi} and \eqref{eq:define_n} yields
\begin{eqnarray}
\hat{\psi}_{j,\alpha}(\mx_\alpha) &=& \psi_\alpha(\mx_\alpha) \prod_{i \in N(\alpha) \setminus j} \exp(-\lambda_{i,\alpha}(\mx_\alpha))\nonumber
\\
\label{eq:hat_psi_update}
  &=& \psi_\alpha(\mx_\alpha) \prod_{i \in N(\alpha) \setminus j} n_{i \rightarrow \alpha}(\mx_\alpha).
\end{eqnarray}

Finally, plugging \eqref{eq:hat_psi_update} into \eqref{eq:def_m} leads to 
\[
	 m_{\alpha \rightarrow j}(x_j) = 
    \left( 
        \sum_{\mx_\alpha \setminus x_j}
            \left( 
                \psi_\alpha(\mx_\alpha)\prod_{i\in N(\alpha) \setminus j} n_{i  \rightarrow \alpha}(\mx_\alpha)
            \right)^{1/\epsilon \hat{c}_{j\alpha}}
    \right)^{\epsilon\hat{c}_{j\alpha}}.
\]
Plugging \eqref{eq:hat_psi_update} into \eqref{eq:nupdate}, in view of the different forms of $\mb_j^*$ for $j\in\Gamma$ and $j\notin\Gamma$, we obtain
\[
n_{j \rightarrow \alpha}(\mx_{\alpha}) \propto
    \left(
        \frac{\phi_j^{1/\hat{c}_j}(x_j)
                \prod_{\beta\in N(j)} m_{\beta\rightarrow j}^{1/\hat{c}_j}(x_j)}
        {m_{\alpha \rightarrow j}^{1/\hat{c}_{j\alpha}}(x_j)}
    \right)^{c_\alpha}
    \left(
        \psi_\alpha(\mx_\alpha)
            \prod_{i\in N(\alpha)\setminus j} n_{i\rightarrow \alpha}(\mx_\alpha)
    \right)^{-c_{j\alpha}/\hat{c}_{j\alpha}}
\]
for $j\notin \Gamma$, and 
\[
 n_{j \rightarrow \alpha}(\mx_{\alpha}) \propto
    \left(
        \frac{\mu_j (x_j)}
        {m_{\alpha \rightarrow j}^{1/\epsilon \hat{c}_{j\alpha}}(\mx_\alpha)}
    \right)^{\epsilon c_\alpha}
    \left(
        \psi_\alpha(\mx_\alpha)
            \prod_{i\in N(\alpha)\setminus j} n_{i\rightarrow \alpha}(\mx_\alpha)
    \right)^{-c_{j\alpha}/\hat{c}_{j\alpha}}
\]
for $j\in \Gamma$.
This concludes the derivation.

\vspace*{.1in}

\bibliographystyle{IEEEtran}
\bibliography{./refs}

\end{document}